\documentclass[a4paper, UKenglish, cleveref, autoref]{article}

\usepackage[a4paper, margin=1.0in]{geometry}
\usepackage[T1]{fontenc}
\usepackage[mathscr]{eucal}
\usepackage{mathtools}
\usepackage{enumitem}
\usepackage{multirow}
\usepackage{amsmath}
\usepackage{amsthm}
\usepackage{graphicx}
\usepackage{amssymb}
\usepackage[all]{xy}
\usepackage[usenames,dvipsnames,svgnames,table]{xcolor}
\usepackage{tikz-cd}
\usepackage{tikz}
\tikzcdset{scale cd/.style={every label/.append style={scale=#1},
    cells={nodes={scale=#1}}}}
\usepackage{soul,hyperref}
\usepackage{aliascnt}
\usepackage[normalem]{ulem}
\usetikzlibrary{fadings,shadings}
\usepackage{dsfont}
\usepackage{microtype}
\usepackage{csquotes}
\usepackage{float}
\usepackage{quiver}
\usepackage{blkarray}

\usetikzlibrary{matrix, arrows}
\usepackage{tikz-3dplot}
\tdplotsetmaincoords{60}{120} 
\tikzcdset{scale cd/.style={every label/.append style={scale=#1},
    cells={nodes={scale=#1}}}}

    \newcommand{\rn}[2]{
    \tikz[remember picture,baseline=(#1.base)]\node [inner sep=0] (#1) {$#2$};%
}

\makeatletter

\newcommand{\ignore}[1]{}

\definecolor{lightred}{rgb}{1, 0.5, 0.5}
\definecolor{darkred}{rgb}{0.8, 0.3, 0.3}
\definecolor{darkerred}{rgb}{0.6, 0.1, 0.1}

\definecolor{lightblue}{rgb}{0.4, 0.5, 1}
\definecolor{darkblue}{rgb}{0.3, 0.3, 0.8}
\definecolor{darkerblue}{rgb}{0.1, 0.1, 0.6}

\definecolor{lightpurple}{rgb}{0.8, 0.5, 1}
\definecolor{darkpurple}{rgb}{0.8, 0.3, 0.8}
\definecolor{darkerpurple}{rgb}{0.6, 0.1, 0.6}

\definecolor{darkgreen}{rgb}{0.1, 0.6, 0.1}

\newcommand{\lrcell}{\cellcolor{lightred}}
\newcommand{\lbcell}{\cellcolor{lightblue}}

\newcommand\blcell{\cellcolor{Blue}}
\newcommand\pblcell{\cellcolor{lightblue}}
\newcommand\lblcell{\cellcolor{LightSkyBlue}}

\newcommand{\bmcell}{\cellcolor{darkblue}}

\newcommand{\gmcell}{\cellcolor{darkred}}

\newcommand{\llrcell}{\cellcolor{lightred!50}}
\newcommand{\llgcell}{\cellcolor{lightgray!40}}

\newcommand{\lpcell}{\cellcolor{lightpurple}}

\newcommand\ocell{\cellcolor{BurntOrange}}

\newcommand\apcell{\cellcolor{Apricot}}
\newcommand\tqcell{\cellcolor{Turquoise}}

\newcommand\lgcell{\cellcolor{lightgray}}

\global\long\def\rtext#1{\textcolor{red}{#1}}%
\global\long\def\gtext#1{\textcolor{Green}{#1}}%

\global\long\def\otext#1{\textcolor{Orange}{#1}}%
\newcommand{\wtext}[1]{\textcolor{white}{#1}}
\newcommand{\bbracks}[1]{\textcolor{lightblue}{\{} #1 \textcolor{lightblue}{\}} }
\newcommand{\gbracks}[1]{\textcolor{lightred}{(} #1 \textcolor{lightred}{)} }

\newcommand{\minus}{\scalebox{0.75}[1.0]{$-$}}

\newcommand\pcoord[1]{ {\scriptstyle  (#1)} }

\newcommand\up[1]{\left\langle #1 \right\rangle}
\newcommand\down[1]{\downarrow #1}

\global\long\def\into{\hookrightarrow}%
\newcommand{\ito}[1]{
\mathrel{\lhook\joinrel\xrightarrow{#1}}%
}
\global\long\def\onto{\twoheadrightarrow}%

\newcommand{\bigslant}[2]{{\raisebox{.2em}{$#1$}\left/\raisebox{-.2em}{$#2$}\right.}}

\newcommand{\stcomp}[1]{{#1}^{\mathsf{c}}}
\newcommand{\lmod}{{\operatorname{-Mod}}}
\newcommand{\N}{\mathbb{N}}
\newcommand{\Z}{\mathbb{Z}}
\newcommand{\F}{\mathbb{F}}
\newcommand{\R}{\mathbb{R}}
\newcommand{\B}{\mathscr{B}}
\newcommand{\Cs}{\mathscr{C}}

\newcommand{\Ds}{\mathscr{D}}

\newcommand{\Dec}{\mathbf{Dec}}

\newcommand{\Fc}{\mathcal{F}}
\newcommand{\Gr}{\mathbf{Gr}}
\newcommand{\GL}{\mathbf{GL}}
\newcommand{\K}{\mathbb{K}}
\newcommand{\T}{\mathscr{T}}

\newcommand{\Ic}{\mathcal{I}}

\newcommand{\Pc}{\mathcal{P}}

\newcommand{\Us}{\mathscr{U}}

\newcommand{\m}{\mathfrak{m}}

\newcommand{\Oc}{\mathcal{O}}
\newcommand{\vect}{\text{Vect}}
\newcommand{\grA}{A\text{-GrMod}}

\newcommand{\oto}[1]{\xrightarrow{#1}}

\newcommand{\iso}{\xrightarrow{\,\smash{\raisebox{-0.5ex}{\ensuremath{\scriptstyle\sim}}}\,}}
\newcommand{\incl}{\text{incl}}
\newcommand{\const}{\text{const}}
\newcommand{\sh}[1]{\text{sh}_{#1}}
\DeclareMathOperator\Hom{Hom}
\DeclareMathOperator{\Aut}{Aut}
\DeclareMathOperator{\End}{End}
\DeclareMathOperator{\Ext}{Ext}
\DeclareMathOperator{\fun}{Fun}
\DeclareMathOperator{\Tor}{Tor}
\DeclareMathOperator{\Id}{Id}

\DeclareMathOperator{\Ima}{im}

\DeclareMathOperator\coker{coker}

\DeclareMathOperator*{\colim}{colim}
\DeclareMathOperator{\rad}{rad}

\DeclareMathOperator\supp{supp}

\theoremstyle{plain}
\newtheorem{theorem}{\protect\theoremname}[section]

\theoremstyle{plain}
\newaliascnt{lemma}{theorem}
\newtheorem{lemma}[lemma]{\protect\lemmaname}
\aliascntresetthe{lemma}

\theoremstyle{plain}
\newaliascnt{conjecture}{theorem}

\aliascntresetthe{conjecture}

\theoremstyle{plain}
\newaliascnt{proposition}{theorem}
\newtheorem{proposition}[proposition]{\protect\propositionname}
\aliascntresetthe{proposition}

\theoremstyle{plain}
\newaliascnt{corollary}{theorem}
\newtheorem{corollary}[corollary]{\protect\corollaryname}
\aliascntresetthe{corollary}

\theoremstyle{definition}
\newaliascnt{definition}{theorem}
\newtheorem{definition}[definition]{\protect\definitionname}
\aliascntresetthe{definition}

\theoremstyle{definition}
\newaliascnt{example}{theorem}
\newtheorem{example}[example]{\protect\examplename}
\aliascntresetthe{example}

\theoremstyle{definition}
\newaliascnt{algorithm}{theorem}
\newtheorem{algorithm}[algorithm]{\protect\algorithmname}
\aliascntresetthe{algorithm}

\theoremstyle{definition}
\newaliascnt{construction}{theorem}
\newtheorem{construction}[construction]{\protect\constructionname}
\aliascntresetthe{construction}

\theoremstyle{remark}
\newaliascnt{remark}{theorem}
\newtheorem{remark}[remark]{\protect\remarkname}
\aliascntresetthe{remark}

\theoremstyle{remark}
\newaliascnt{question}{theorem}
\newtheorem{question}[question]{\protect\questionname}
\aliascntresetthe{question}

\theoremstyle{remark}
\newaliascnt{observation}{theorem}
\newtheorem{observation}[observation]{\protect\observationname}
\aliascntresetthe{observation}

\providecommand{\constructionname}{Construction}
\providecommand{\corollaryname}{Corollary}
\providecommand{\definitionname}{Definition}
\providecommand{\examplename}{Example}
\providecommand{\lemmaname}{Lemma}

\providecommand{\propositionname}{Proposition}
\providecommand{\questionname}{Question}
\providecommand{\remarkname}{Remark}
\providecommand{\theoremname}{Theorem}
\providecommand{\observationname}{Observation}
\providecommand{\conjecturename}{Conjecture}
\providecommand{\algorithmname}{Algorithm}

\newcommand{\CC}{C\nolinebreak\hspace{-.05em}\raisebox{.4ex}{\tiny\bf +}\nolinebreak\hspace{-.10em}\raisebox{.4ex}{\tiny\bf +}}
\def\CC{{C\nolinebreak[4]\hspace{-.05em}\raisebox{.4ex}{\tiny\bf ++}}}

\tikzset{
    circled/.style={circle, draw, minimum size=2.5em, inner sep=0pt, align=center, font=\sffamily},
    circled blue/.style={circled, fill=lightblue},
    circled red/.style={circled, fill=lightred},
    circled gray/.style={circled, fill=gray}
}

\title{Decomposing Multiparameter Persistence Modules}

\author{Tamal K. Dey, Jan Jendrysiak, Michael Kerber}

\begin{document}

\maketitle

\begin{abstract}

We describe an algorithm, \textsc{AIDA}, to decompose any finitely presented multiparameter persistence module. We do this by extending the work of Dey and Xin (J.Appl.Comput.Top., 2022), a matrix reduction algorithm for modules whose generators and relations are distinctly graded. 
We also introduce several improvements to their approach that lead to significant speed-ups
in practice.
Our algorithm is fixed parameter tractable with respect to the maximal number of relations
of the same degree, and with further optimisation we obtain 
an $O(n^3)$ algorithm for interval-decomposable modules. In particular, \textsc{AIDA} decides interval-decomposability in this time.

To prove the correctness of the algorithms, we develop a theory of parameter restriction for persistence modules.
We have implemented \textsc{AIDA} in a software library, which is the first to enable the decomposition
of large multiparameter persistence modules. 
We demonstrate its capabilities through extensive experimental evaluation.
\end{abstract}

\begin{figure}[H]\label{aida_painting}
\centering
    \includegraphics[width = 0.93\textwidth]{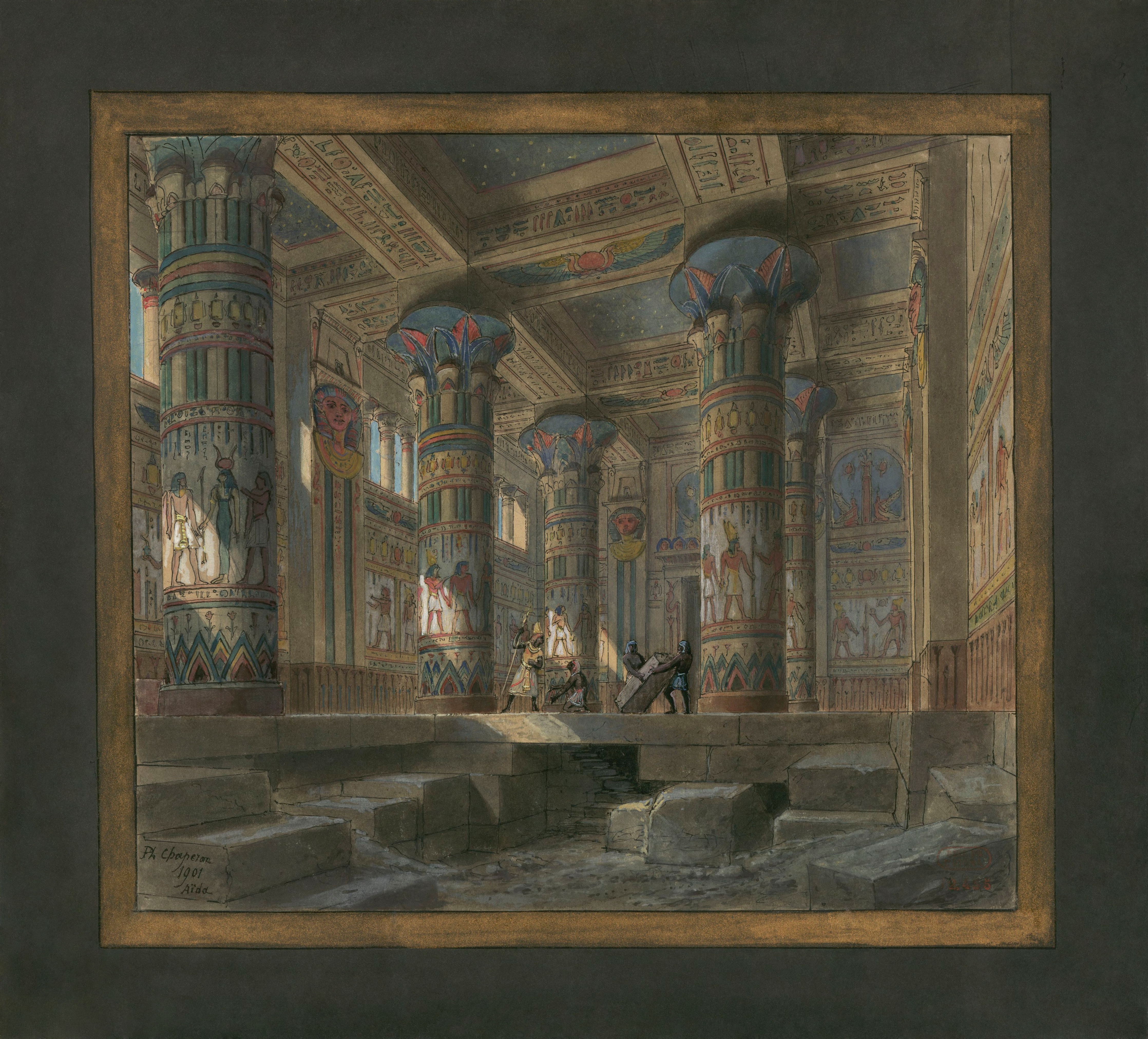}
\caption{Philippe Chaperon's set design for Giuseppe Verdi's opera Aida, 1880}
\label{fig:aida_sketch}
\end{figure}

\textbf{Acknowledgements} We want to thank Klaus Lux for his insights and discussions about the Meat-Axe algorithm and both him and Martin Kalck for helpful discussions about Representation Theory.

\section{Introduction}

Multiparameter Persistence Modules (MPM) as introduced by Carlsson and Zomorodian \cite{CZMP,CZMP-socg} make it possible to capture topological features of data with respect to more than one parameter. This, for example, allows detection of features in different scale-regions stably with respect to outliers \cite{blumberg_lesnick}. One of the main challenges of dealing with MPM is that their indecomposable summands do not provide a discrete invariant, as in the case of finitely generated $1$-parameter persistence modules, where this decomposition leads to the \emph{barcode}. 

In contrast, indecomposable MPM can be arbitrarily complex because the representation theory of the partially ordered set $\N^d$ is wild (\cite{Kleiner1975, Loupias}). In light of this difficulty, researchers in Topological Data Analysis have come up with various invariants with varying degree of discrimination powers and ease of computability (\cite{botnan2021signed, corbet2019kernel, dkm-computing, FEVT, loiseaux2023stable, mccleary2022edit, morozov2021output, vipond2018multiparameter, xin2023gril}). Nevertheless, the begging question remains how to compute the indecomposable summands of MPM. 

A fast algorithm for this task has numerous benefits, since splitting an MPM in parts is useful for many algorithmic tasks.
Primarily, every additive invariant of persistence module can, by definition, be computed from its indecomposable summands, possibly reducing computation time. For example, the computation of the matching distance~\cite{italian,bk-asymptotic} becomes more tractable for decomposed modules as the complexity of the barcode template~\cite{lw-interactive} reduces, which we show empirically in \autoref{sub:barcode_template}. 

\subsection*{Contributions}
We present an algorithm inspired by the Dey-Xin algorithm \cite{DeyXin}, speeding up various steps and lifting the assumption that the module needs to be distinctly graded. Our methods are collected in the software package \textsc{aida} (\autoref{fig:aida_sketch}), which allows the decomposition of presentations arising from large data sets with hundreds of thousands of generators and relations in seconds. We demonstrate the usefulness of our improvements in extensive experimental evaluation. Since the (linear) algebraic methods which are used in \textsc{aida} are of more general interest we implemented them in a stand-alone, header-only \CC library \textsc{Persistence-Algebra}. Both this library and the \textsc{aida} code are publicly available.\footnote{\url{https://github.com/JanJend/AIDA}  and \url{https://github.com/JanJend/Persistence-Algebra}}

While the distinctly graded case allows for an incremental column-by-column approach to compute a decomposition, the general case requires considering all columns of the same degree at once. Denote by $\alpha_i \in \Z^d$ this degree and by $k_i$ the number of columns. It does not suffice to consider all permutations of these columns \cite{ABENY}; instead, arbitrary changes of basis within the columns could be part of a correct solution. Our first approach requires a brute-force iteration over these changes of basis (which remains finite over a finite field) and is therefore not polynomial in $k_i$. Nevertheless, through a careful analysis, we limit the number of cases and provide a tractable solution if $k_i$ is small. Since we can then use the poly-time strategy of \cite{DeyXin}, this yields a decomposition algorithm, which is fixed-parameter tractable with respect to $k_{\max} \coloneqq \max_i k_i$.

We explain that the row operations one performs on the matrix actually correspond to homomorphisms between summands in the partial decomposition that has been achieved up to the $i$-th step.
To further reduce the number of iterations -- or even avoid it altogether -- we analyse between which summands in the partial there are no homomorphisms at all. This leads to another FPT-algorithm in a parameter $\kappa_{\max}$, which is smaller than $k_{\max}$ and promises to accelerate the computation for modules where $k_i$ is too large for the brute-force strategy to work. 

We observe that those homomorphisms which are zero at the degree $\alpha_i$ cannot actually contribute to the decomposition. 
This idea yields another extension of our algorithm. Combining this approach with a fast computation of homomorphisms for interval modules \cite{dey_xin} we arrive at another variant which runs in $\Oc(n^3)$ time for interval-decomposable modules, independent of $k_{\max}$. In particular, we can decide interval-decomposability in this time.

Most of our results and algorithms generalise to representation over \emph{any} poset by embedding them into $\Z^d$ for large enough $d$.

\subsection*{Related work}
Multi-parameter persistent homology is an active research area with
many theoretic, computational, and applied results in recent years ~\cite{BotnanLesnick}. In particular, efficient algorithms have been devised for the computation of a presentation
of standard bifiltrations~\cite{akll-delaunay,akp-filtration,bdk-sparse,bdk-sparse-journal,cklo-computing,ks-localized}, for minimizing presentations~\cite{DeyXin,fkr-compression,lw-computing} and for efficient distance computation~\cite{italian,bk-asymptotic,dey_xin}. Our decomposition algorithm complements this pipeline.

\subparagraph{Classical Techniques.} In theory we can decompose any module over a finite dimensional algebra (over a finite field) in polynomial time by decomposing its endomorphism algebra \cite{CGK97}. But this algorithm is not practical for the large sparse matrices appearing in TDA. Instead, the algorithms in the C MeatAxe Package (\cite{LuxSzoke}, \cite{Szoke}), in the MAGMA system \cite{MAGMA}, in the GAP system \cite{GAP4}, \cite{QPA}, rely on the Holt-Rees extension \cite{HoltRees92} of the MeatAxe algorithm \cite{Parker84}, both of which are \emph{Las Vegas} algorithms. The first step is always the computation of the endomorphism algebra, which is in practice already too expensive for large presentations. We also remark that that the run-time complexity of these algorithms is sometimes not precisely cited in the TDA literature. For example the bounds in  \cite[Section 8.5]{BotnanLesnick} and \cite[Introduction]{DeyXin}) are a misinterpretation of the usage of the MeatAxe algorithm in this pipeline and use the \emph{lower} bound given in \cite{Holt_1998} for the non-practical algorithm described in \cite{Ronyai} as an upper bound.

Our algorithm phrases the decomposition problem as a matrix reduction problem, where some operations are not allowed. This phrasing has a long history in representation theory and goes back to Rojter \cite{Rojter1980}.

\subparagraph{Previous Results for Multiparameter Persistence Modules.}
The only algorithm specifically designed for MPM is the aforementioned Dey-Xin algorithm~\cite{DeyXin} which has a time complexity of $\Oc (n^{2\omega+1})$, where $\omega<2.373$ is the matrix multiplication constant. It only works for uniquely graded modules, a property that is satisfied by some, but not all modules, that appear in practice. Our algorithm is a direct extension of - and heavily uses - their work.

A similar matrix reduction technique has been used by Asashiba, Escolar, Hiraoka, and Takeuchi in \cite{ladderdecomp} to decompose persistence modules on the commutative ladder of length $\leq 4$, the representation finite case, although we cannot directly infer its running time. This can be seen as the first intermediate step between the single parameter and multiparameter regime.

\subparagraph{Previous Results for Interval-Decopmposability.}
In \cite[Proposition 4.5, 4.15]{dkm-computing}, Dey, Kim, and Mémoli show that their algorithm to compute generalized rank invariants via zig-zags also decides interval-decomposability in $\Oc(n^4)$ and can be used to compute the interval-decomposition in $\Oc(n^{\omega +2})$.

In \cite[Theorem 44]{ABENY}, Asashiba, Buchet, Escolar, Nakashima, and Yoshiwaki provide an algorithm which decides interval-decomposability in time $\Oc\left( (n^\omega (\overline{ \dim} X)^\omega + n^{\omega +2}) \# I_{n,n} \right)$, where $\overline{ \dim}  X$ is the total dimension of the input module $X$ over the grid induced by its generators and relations and $\# I_{n,n}$ is the number of interval modules on an $n \times n$ grid. In the worst case $\overline{ \dim X} \sim n^3$ and $\# I_{n,n}$ is at least exponential in $n$.

\subparagraph{On Decompositions.}
Decompositions are not stable with respect to the interleaving distance;
every persistence module is even arbitrarily close to an indecomposable module by the work of Bauer and Scoccola~\cite{BL22}, but geometric bifiltrations typically do not create too many of such pathologies. Moreover, Bjerkevik~\cite[Thm 5.4]{Bjerkevik} constructs a module, which captures all decompositions of an entire $\epsilon$-interleaving-neighbourhood in a precise sense. 
This module has many generators and relations of the same degree,
so its decomposition requires an efficient algorithm, which we provide with \textsc{AIDA}.

Complicated indecomposables can arise already in relatively simple geometric examples ~\cite{BE22,BE22-socg} and most modules in practice are not interval-decomposable (e.g., most examples from our benchmark set contain non-interval indecomposables; see also \cite{AKS24}). 
However, we expect that most modules in practice will mainly consist of smaller indecomposables (see \cite{AlonsoKerber23}), and we designed our algorithm to work fast in such scenarios.

\subsection*{Outline.}
We want to direct the reader who is only interested in the algorithmic content to the relevant sections:
\begin{itemize}
\item An informal review of persistence modules and their presentations is found in \autoref{persistence_introduction} and  the basics about decompositions in \autoref{sec:decompositions}.  
\item The algorithm by Dey and Xin~\cite{DeyXin}, which is the basis of our work, is reviewed and improved in Sections ~\ref{sec:deyxin} and \ref{sec:modifications}.
\item In Sections~\ref{sec:obstructions} and \ref{sec:brute_force}, we describe how to decompose non-distinctly graded modules using an exhaustive approach.
\item Sections~\ref{sec:submodules} and \ref{sec:decomp_algo} introduce an extension of this algorithm which reduces the potentially restrictive iteration of the exhaustive algorithm. \item In \autoref{sec:localisation} we explain an idea, which lets us avoid this iteration in more cases. This enables us to decompose interval-decomposable modules in $\Oc(n^3)$ time, which we explain how to do in \autoref{sec:interval_decomp}.
\item All experiments are to be found in \autoref{sec:experiments}.
\end{itemize}

This paper is mostly self-contained, but to be able to read all proofs, we assume the reader to be familiar with homological algebra in the scope of e.g. the first three chapters of Weibel's book \cite{Weibel}.
\begin{itemize}
\item A thorough treatment of presentations is in Sections~\ref{formal_introduction} and \ref{sec:minimal_pres}.
\item  \autoref{sec:projectivity} will explain how forgetting relations of a module is the same as forming a projective cover for a different exact structure.
\item In Sections \ref{sec:effective} and \ref{sec:decomposing_pres} we prove the correctness of the exhaustive algorithm by relating decompositions with presentation matrices.
\item In \autoref{sec:restriction_and_localisation} we put the concepts introduced during the algorithm on firm mathematical ground, explore their homological features, and prove the correctness of \textsc{AIDA}.
\end{itemize}

\section{Persistence Modules and their Presentations}
\label{sec:persistence_modules}


\subsection{Introduction}\label{persistence_introduction}

\subparagraph{Notation}

\begin{itemize}

  \item $A \coloneqq \K[x_1, \dots, x_d]$ denotes the $\Z^d$-graded commutative polynomial algebra.
  \item $X, Y, \dots$ denote (persistence) modules.
  \item $\mathfrak{m}=(x_1, \dots, x_d) \subset A$ is the maximal homogeneous ideal.
  \item Greek minuscules $\alpha, \beta, \dots, \omega \in \Z^d$ denote points in the \emph{poset} $\Z^d$. 
  \item For a set $S \subset \Z^d$ we denote by $\up{S}$ the upper set and $\down{S}$ the lower set generated by $S$.
  \item We use multi-index notation; for $\alpha \in \Z^d$,  $x^\alpha \coloneqq \prod_{i \in [d]} x_i^{\alpha_i}$.
  \item  $M, N, \dots$ denote matrices and $f,g, \dots$ for non-matrix homomorphisms.

\end{itemize}

\subparagraph{Simplicial filtrations.}
Consider a simplicial complex $S$, discretely filtered with respect to $d$ parameters. That is for each $\alpha \in \N^d$ there is a subcomplex $S_{\alpha} \subset S$ and for each pair $\alpha \leq \beta$ there is an inclusion $S_{\alpha} \subset S_\beta$. Applying the simplicial homology functor $H_*(-; \K)$, for a field $\K$ results in a commutative diagram of $\K$-vector spaces over $\N^d$. For $d=2$, it has the following shape:
\[\begin{tikzcd}[ampersand replacement=\&]
	\vdots \& \vdots \& \vdots \\
	{H_*(S_{0,1} \, ; \, \K)} \& {H_*(S_{1,1} \, ; \, \K)} \& {H_*(S_{2,1} \, ; \, \K)} \& \cdots \\
	{H_*(S_{0,0} \, ; \, \K)} \& {H_*(S_{1,0} \, ; \, \K)} \& {H_*(S_{2,0} \, ; \, \K)} \& \cdots
	\arrow[from=2-1, to=1-1]
	\arrow[from=3-1, to=3-2]
	\arrow[from=3-1, to=2-1]
	\arrow[from=3-2, to=2-2]
	\arrow[from=2-1, to=2-2]
	\arrow[from=3-2, to=3-3]
	\arrow[from=2-2, to=2-3]
	\arrow[from=2-2, to=1-2]
	\arrow[from=3-3, to=3-4]
	\arrow[from=2-3, to=2-4]
	\arrow[from=3-3, to=2-3]
	\arrow[from=2-3, to=1-3]
\end{tikzcd}\]

The maps encode how classes in $H_*(S_{\alpha},\K)$ \emph{persist} when parameters are increased. \\
Typically the parameters will actually be real numbers, but for the finitely presented modules which we want to decompose with our algorithm we can work in the discrete setting (\autoref{discretisation}). We also imagine that the diagram is extended with $0$s to the left and bottom so that instead of $\N^d$ we can use $\Z^d$ to avoid some case-distinctions.

\begin{definition}
 A $d$-parameter \emph{persistence module} $X$ over $\K$ is a commutative diagram of $\K$-vector spaces over $\Z^d$. The category of persistence modules is denoted $\fun(\Z^d, \vect_\K)$.
\end{definition}

Although the ideas presented in this paper should work for any poset, we are mostly interested in the $\Z^d$ case. This allows a translation to commutative algebra: We may also view the evolution along a parameter $i$ as an action of a formal variable $x_i$. This makes a persistence module a $\Z^d$-graded modules over the commutative $\Z^d$-graded polynomial algebra $A \coloneqq \K[x_1, \dots, x_d]$. 

\begin{proposition}\label{equ_persistence}\cite[General version of Theorem 1]{CZMP}
 Let $\grA$ be the category of $\Z^d$-graded $A$-modules and linear maps. The direct sum of all vector spaces in a persistence module and the functor which sends a graded $A$-module to its homogeneous summands define an equivalence of categories
 \[ \bigoplus \colon \fun \left( \Z^d, \vect_\K \right) \rightleftarrows \grA \colon \left(\left( - \right)_\alpha \right)_{\alpha \in \Z^d}. \]
\end{proposition}

A class of modules that is especially easy to describe is the following.
\begin{definition}[Interval Module]\label{def:interval}
    Let $U \subset \Z^d$ be an \emph{interval}, i.e. a connected convex sub-poset. Define the module $\K_U$ by $\left( \K_U \right)_\alpha \simeq \K$ iff $\alpha \in U$ and having $\Id_\K$ as the structure map wherever possible. A persistence module is called \emph{Interval module} if it is isomorphic to a module of the form $\K_U$.
\end{definition}

\subparagraph{Presentations:}
A persistence module can be succinctly represented by a \emph{presentation}. The theory of presentations goes back to Hilbert \cite{Hilbert1890} and a modern treatment for graded modules is found in \cite{Peeva}. \\
Informally, a presentation consists of: 
\begin{itemize} 
\item Generators: a set of vectors such that all other vectors are $A$-linear combinations of these.
\item Relations: a set of linear combinations of generators which map to 0.
\end{itemize}
If these two sets are finite, they form a \emph{presentation matrix}, where the row indices correspond to the generators, the column indices to the relations and an entry at $(i,j)$ tells us which scalar coefficient the $i$-th generator appears with in the $j$-th relation.

\begin{definition}
 A \emph{graded matrix} is a matrix $M \in \K^{m \times n}$ together with functions
 $G \colon [m] \to \N^d$ and $R \colon [n] \to \N^d$ which decorate the rows and columns with degrees such that
 \[ \forall (i,j) \in [m] \times [n]: \ \text{ If } R(j) \not\geq G(i) \text{ then } M_{i,j}=0 .\] 
 Given two functions $G,R$ as above we write $\K^{G \times R}$ for the corresponding space of graded matrices. 
\end{definition}

\begin{figure}[H]
    \centering
    \begin{minipage}[l]{0.37\textwidth}
        \begin{tikzcd}[scale cd=1]
            \vdots & \vdots & \vdots \\
            \K & \K^2 &  \gbracks{\K}  & \cdots \\
            {\bbracks{\K}} & {\K^2} & {\K^2} & \cdots \\
            0 & {\bbracks{\K}} & \K & \cdots
            \arrow["{\left(\begin{smallmatrix}0 \\ 1 \end{smallmatrix}\right)}"', from=4-3, to=3-3]
            \arrow[from=3-2, to=3-3]
            \arrow[from=3-2, to=2-2]
            \arrow["{\left(\begin{smallmatrix}1 \\ 0 \end{smallmatrix}\right)}", from=2-1, to=2-2]
            \arrow[from=2-1, to=1-1]
            \arrow["{\left( 1 \ 1 \right)}"', from=3-3, to=2-3]
            \arrow["{\left( 1 \ 1 \right)}", from=2-2, to=2-3]
            \arrow[from=3-1, to=2-1]
            \arrow["0"', from=4-1, to=3-1]
            \arrow["0", from=4-1, to=4-2]
            \arrow[from=4-2, to=4-3]
            \arrow["{\left(\begin{smallmatrix}0 \\ 1 \end{smallmatrix}\right)}"', from=4-2, to=3-2]
            \arrow["{\left(\begin{smallmatrix}1 \\ 0 \end{smallmatrix}\right)}", from=3-1, to=3-2]
            \arrow[from=2-2, to=1-2]
            \arrow[from=2-3, to=1-3]
            \arrow[from=4-3, to=4-4]
            \arrow[from=3-3, to=3-4]
            \arrow[from=2-3, to=2-4]
        \end{tikzcd}
    \end{minipage}
    \hspace{0.5em}
    \begin{minipage}[c]{0.19\textwidth}
        \[
        \begin{array}{| c | c |}
            \hline 
             M & \lrcell (2,2) \\
            \hline
            \pblcell (0,1) &  1  \\
            \pblcell (1,0) &  -1  \\ 
         \hline
        \end{array}
        \]
    \end{minipage}
    \hspace{0.1em}
    \begin{minipage}[r]{0.35\textwidth}
        \begin{tikzpicture}[scale=0.8]
        
            \draw[->, thick] (0,0) -- (4,0) node[right] {$x$};
            \draw[->, thick] (0,0) -- (0,4) node[above] {$y$};

            \foreach \x in {0,1,2,3} {
                \node at (\x, -0.2) [below] {\x};
            }
            \foreach \y in {0,1,2,3} {
                \node at (-0.2, \y) [left] {\y};
            }
            
            \fill[shading=radial, inner color=lightblue, outer color=white, draw = white] (3,3) circle (0.5);           
            

            \fill[shading=axis, bottom color=lightblue, top color=white, draw=none, opacity=1] 
            (2,3) -- (2,3.5) -- (3,3.5) -- (3,3);
            
            \fill[shading=axis, left color=lightblue, right color=white, draw=none, opacity=1] 
            (3,2) -- (3.5,2) -- (3.5,3) -- (3,3);

            \fill[shading=axis, bottom color=darkblue, top color=white, draw=none, opacity=1]
            (1,3) -- (1,3.5) -- (2,3.5) -- (2,3) ;
            \fill[shading=axis, left color=darkblue, right color=white, draw=none, opacity=1]
            (3,1) -- (3.5,1) -- (3.5,2) -- (3,2);

            \fill[shading=axis, bottom color=lightblue, top color=white, draw=none, opacity=1] 
            (0,3) -- (0,3.5) -- (1,3.5) -- (1,3);
            
            \fill[shading=axis, left color=lightblue, right color=white, draw=none, opacity=1] 
            (3,0) -- (3.5,0) -- (3.5,1) -- (3,1);

            \draw[fill=lightblue, draw=none, opacity=1, thick] (0,3) -- (0,1) -- (1,1) -- (1,0) -- (3,0) -- (3,3);
            \draw[fill = darkblue, draw=none, opacity=1, thick]
            (1,3) -- (1,1) -- (3,1) -- (3,2) -- (2,2) -- (2,3);
            
            \draw[color = white]
            (0,3.5) -- (3,3.5);
            \draw[color = white]
            (3.5,3) -- (3.5,0);

            \draw[color = lightblue]
            (0,3) -- (1,3);
            \draw[color = lightblue]
            (2,3) -- (3,3);
            \draw[color = lightblue]
            (3,0) -- (3,1);
            \draw[color = lightblue]
            (3,2) -- (3,3);

            \draw[color = darkblue]
            (1,3) -- (2,3);
            \draw[color = darkblue]
            (3,1) -- (3,2);
            
            
            
            \foreach \x/\y in {0/1, 1/0} {
                \fill[lightblue, draw=black] (\x,\y) circle (3pt);
            }
        
            \foreach \x/\y in {2/2} {
                \node[rectangle, draw=black, fill=lightred, inner sep=2pt] at (\x,\y) {};
            }
        \end{tikzpicture}
    \end{minipage}
    
    \caption{The matrix in the middle presents the module on the left. The picture on the right indicates the dimensions of the vector spaces and locations of generators (blue points and brackets) and relations (red point and bracket) over a continuous plane.}
    \label{fig:module_example}
\end{figure}
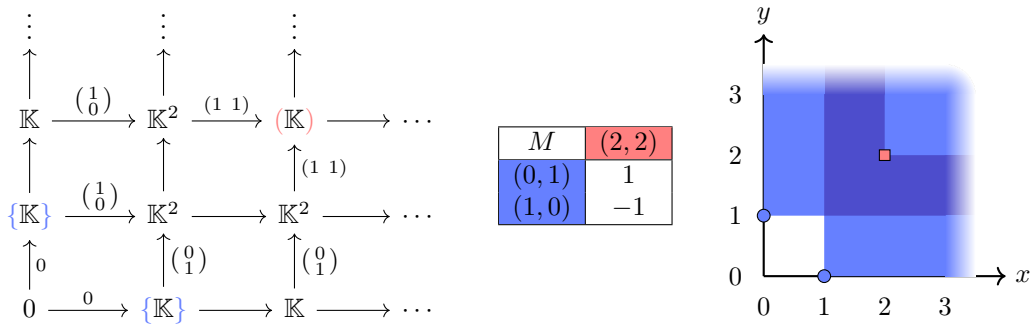

The structure of these matrices for $\Z^d$-graded modules was first noticed by Miller in \cite{miller00} where they are introduced as \emph{monomial matrices}.
We can alter presentation matrices with elementary row- or column-operations, but only some operations will preserve the isomorphism type of the presented module.

\begin{definition}[{\cite[Section 4.1]{DeyXin} }]For $M \in \K^{G \times R}$, an elementary row- (column-) operation from index $i$ to index $j$ is called \emph{admissible} if $G(i) \geq G(j)$ ($R(i) \leq R(j)$).
\end{definition}

Informally, we will say that a graded matrix is \emph{minimal} if there is no other graded matrix with fewer columns or rows that presents the same persistence module. This means that there is no sequence of admissible operations that either zeroes out an entire column or creates a column with a non-zero entry in a row of the same degree.

\begin{remark}
Every finitely generated ($fg$) persistence module over $\Z^d$ has a finite presentation (Hilbert's syzygy theorem \cite{Hilbert1890}). 
Furthermore, there are efficient algorithms to compute (minimal) presentations of persistence modules from simplicial bi-filtrations (\cite{cohomology,mpfree,lw-computing}). For $3$ or more parameters, computing a presentation from simplicial bi-filtrations can be done using Schreyer's algorithm (\cite{schreyer}) and minimising a presentation can always be done by a simple matrix reduction. Therefore, we assume from now on that modules are given to us via their minimal presentations.  
\end{remark}

\subsection{Presentations}\label{formal_introduction}

In this subsection we formalise and connect the notions just introduced. We will need to translate some linear algebra to our graded setting. 

\begin{definition}
 Let $\alpha \in \Z^d$. For a persistence module $X$ we denote by
  $X(\alpha)$
 the \emph{$\alpha$-shift} (-twist) of $X$ defined by $X(\alpha)_{\beta}\coloneqq X_{\beta+\alpha}$ for every $\beta \in  \Z^d$ while preserving the structure maps. \\
 For $G:[n]\rightarrow \Z^d$ a tuple of degrees, we define the free module generated by $G$ as 
 \[ A^G \coloneqq \bigoplus\limits_{i \in [n]} A\left(-G(i)\right).
  \] 
  
\end{definition}
In fact this defines a functor from $\Z^d$-graded sets to the category of persistence modules and it is left adjoint to the forgetful functor as expected.

\begin{definition}
 Let $X$ be a persistence module. If there is a surjection
 \[\varphi \colon A^G \onto X\]
 with $G$ a finite set of degrees ("generators") then $X$ is called \emph{finitely generated} (fg). 

 If $\varphi$ is an isomorphism then $X$ is called \emph{graded free} (gf) and $\varphi$ is called (ordered) \emph{basis} of $X$. 
\end{definition}

\begin{proposition}\label{free}
 Let $\alpha \in \Z^d$ and $X$ be a persistence module. The natural $\K$-linear map
 \[
  \Hom\left(A(-\alpha), X\right) \to X_\alpha \]
  given by considering the image of $1 \in A(-\alpha)$ is an isomorphism.
\end{proposition}

\begin{corollary}\label{matrix_representation}
 Let $\alpha, \beta \in \Z^d$, then there is a canonical isomorphism
\[
  \Hom\left(A\left(-\alpha\right), A\left(-\beta\right)\right) \iso \begin{cases}
                                      \K \text{ if } \alpha \geq \beta \\
                                      0 \text{ else.}
                                     \end{cases} \]
It follows that if $f \colon X \to Y$ is a homomorphism of fg graded free persistence modules and  \[ \varphi \colon A^R \iso X, \ \ \psi \colon A^{G} \iso Y \] are ordered bases then $\psi \circ f \circ \varphi^{-1}$ is  graded matrix in $\K^{G \times R}$.
\end{corollary}

\begin{proof}
    Using \autoref{free} we construct a canonical isomorphism
    \begin{align*} \Hom \left( A^R, \, A^G \right) 
    & \ = \ 
    \Hom \left( \bigoplus_{i =1}^{|R|} A(-R(i)), \, \bigoplus_{j =1}^{|G|} A(-G(j)) \right)  \\
    & \ \iso \ \bigoplus_{i =1}^{|R|} \bigoplus_{j =1}^{|G|} \Hom \left( A(-R(i)), A(-G(j)) \right)  
    \ \iso \ \bigoplus_{i =1}^{|R|} \bigoplus\limits_{\substack{j= 1 \text{ st. } \\ R(i) \geq G(j)}}^{|G|}
     \K \ = \ \K^{R \times G}. 
    \end{align*}
\end{proof}

\begin{remark}
 Let $G$ be a tuple of degrees, then $G$ is in particular pre-ordered and we denote this pre-ordered set by $\mathcal{P}_G$. By \autoref{matrix_representation} the endomorphism algebra $\End(A^{G})$ is the $\K$-algebra of graded matrices $\K^{G \times G}$. This algebra is usually called the \emph{incidence algebra} of $\mathcal{P}_G$ and denoted $\K \mathcal{P}_G$ (\cite{Assem} 1.1d).
\end{remark}

\begin{definition}\label{presentation}
 A \emph{presentation} of a persistence module $X$ is an exact sequence
 \[
  F_1 \oto{f} F_0 \oto{p} X \oto{} 0
 \]
- or only the map $f$ - where $F_1$ and $F_0$ are graded free $A$-modules.  It is called \emph{finite} if $F_1$ and $F_0$ are finitely generated. \\
A choice of ordered bases $\varphi, \, \psi$ for $F_0$ and $F_1$ is called \emph{generators} and \emph{relations} of $X$. 
\end{definition}

\autoref{matrix_representation} says that every choice of generators and relations (for an implicit presentation) induces a graded matrix $\varphi^{-1} \circ f \circ \psi$ from $f$ which presents $X$. 

\begin{example}
The presentation from \autoref{fig:module_example} as an exact sequence in a picture.
    \begin{align*}
    &
    \begin{tikzpicture}[baseline=-2ex, scale = 0.7]
            \draw[fill=lightblue, draw=none, opacity=1, thick] (2,2) -- (2,3) -- (3,3) -- (3,2);
            \foreach \x/\y in {2/2} {
                \node[rectangle, draw=black, fill=lightblue, inner sep=2pt] at (\x,\y) {};
            }
    \end{tikzpicture}
        \raisebox{3em}{$ \quad  
        \oto{\begin{bmatrix} 1 \\ -1\end{bmatrix}} \quad $ }
    \begin{tikzpicture}[baseline=-2ex, scale = 0.7]
            \draw[fill=lightblue, draw=none, opacity=1, thick] (0,3) -- (0,1) -- (1,1) -- (1,0) -- (3,0) -- (3,3);
            \draw[fill = darkblue, draw=none, opacity=1, thick]
            (1,3) -- (1,1) -- (3,1) -- (3,3);
            \foreach \x/\y in {0/1, 1/0} {
                \fill[lightblue, draw = black] (\x,\y) circle (3pt);
            }
    \end{tikzpicture}
        \raisebox{3em}{$ \quad \oto{\hspace{2em}} \quad$}
    \begin{tikzpicture}[baseline=-2ex, scale = 0.7]
            \draw[fill=lightblue, draw=none, opacity=1, thick] (0,3) -- (0,1) -- (1,1) -- (1,0) -- (3,0) -- (3,3);
            \draw[fill = darkblue, draw=none, opacity=1, thick]
            (1,3) -- (1,1) -- (3,1) -- (3,2) -- (2,2) -- (2,3);
            \foreach \x/\y in {0/1, 1/0} {
                \fill[lightblue, draw = black] (\x,\y) circle (3pt);
            }
            \foreach \x/\y in {2/2} {
                \node[rectangle, draw=black, fill=lightblue, inner sep=2pt] at (\x,\y) {};
            }
    \end{tikzpicture}
        \raisebox{3em}{$ \quad \oto{\hspace{2em}} \quad 0$}
    \end{align*}
\end{example}

\begin{definition}\label{morphism_presentation}
A morphism $\mathscr{F} \colon f \to g$ of presentations is defined to be a pair of homomorphisms $\mathscr{F}_1, \mathscr{F}_0$ which form a commutative diagram
\[\begin{tikzcd}[ampersand replacement=\&]
	{F_1} \& {F_0} \& X \& 0 \\
	{G_1} \& {G_0} \& Y \& 0
	\arrow["f", from=1-1, to=1-2]
	\arrow["g", from=2-1, to=2-2]
	\arrow["p", from=1-2, to=1-3]
	\arrow["q", from=2-2, to=2-3]
	\arrow["{\mathscr{F}_0}", from=1-2, to=2-2]
	\arrow["{\mathscr{F}_1}", from=1-1, to=2-1]
	\arrow[from=1-3, to=1-4]
	\arrow[from=2-3, to=2-4]
	\arrow[dashed, from=1-3, to=2-3]
\end{tikzcd}\]
\end{definition}

A morphism of presentations induces a homomorphism of persistence modules (dashed arrow). On the other hand, we can lift any homomorphism of persistence modules (not necessarily uniquely) to a morphism of presentations by the fundamental theorem of homological algebra (\cite{maclane} III. Theorem 6.1). Therefore this correspondence is surjective but generally not injective. 

\begin{definition}\label{min_pres}
A presentation is called \emph{minimal}, if it has no proper sub-presentation which presents the same module, that is, there is no injective morphism to it from another non-isomorphic presentation of the same module.
\end{definition}

\begin{proposition}\label{admissible_transformation}
Let $M \in \K^{R \times G}$ be a graded matrix. If a sequence of admissible operations turns $M$ into another graded matrix $M'$, then they present isomorphic modules.
\end{proposition}

\begin{proof}
Without loss of generality, consider a single row-addition from index $i$ to $j$ with coefficient $c \in \K$. Then the invertibleelementary matrix $L_{i,j}(c) \in \K^{|G| \times |G|}$ is actually a graded matrix in $\K^{G \times G}$ if and only if this row-addition is admissible. \\
In this case we can form a commutative diagram showing that the induced modules are isomorphic and this isomorphism is just a change of basis in the sense of \autoref{presentation}.
\[\begin{tikzcd}[ampersand replacement=\&]
	{A^R} \& {A^G} \& X \& 0 \\
	{A^R} \& {A^G} \& {X'} \& 0
	\arrow["M", from=1-1, to=1-2]
	\arrow["\Id", from=1-1, to=2-1]
	\arrow["{d_0}", from=1-2, to=1-3]
	\arrow["{L_{i,j}(c)}", from=1-2, to=2-2]
	\arrow[from=1-3, to=1-4]
	\arrow[dashed, from=1-3, to=2-3]
	\arrow["{M'}", from=2-1, to=2-2]
	\arrow["q", from=2-2, to=2-3]
	\arrow[from=2-3, to=2-4]
\end{tikzcd}\]
\end{proof}

The reverse is not always true if the matrices involved are not minimal, because in this case they do not even have to have the same dimension. To see that the reverse statement is true for minimal presentations we introduce some homological algebra.

\subsection{Minimal Presentations}\label{sec:minimal_pres}
We need a definition of minimality which is easier to handle than \autoref{min_pres}. 
 If $X$ is a persistence module, $\alpha \in \Z^d$ and $a \in X_\alpha$ is a homogeneous element, then we want to find a notion which tells us if $a$ is in the image of a structure map or not; i.e. if there is $b \in X, c \in \K$ st. there is $\beta \in \N^d$ with $a = cx^\beta \cdot b$. 
 
 We can see that this would allow us to find a minimal set of generators for $X$: Ignore all elements which are such images and find a pointwise basis for the remaining elements.

 To do this, recall that $A$ has a maximal homogeneous ideal $\mathfrak{m}=(x_1, \dots, x_d) \subset A$ consisting of all polynomials without a constant term. Consider the functor $- \otimes_A A/\mathfrak{m} \colon \grA \to \grA  $. For any $X$, the module $X \otimes_A A/\mathfrak{m}$ is also a graded module over $A/ \m \simeq \K$, so the image of this functor is actually  the category of $\Z^d$-graded vector spaces.  One may also visualize this as ignoring all arrows and their images in the grid $\Z^d$. 
 
If $a \in X$ is an image of a structure map, then in the tensor product $X \otimes A / \m$ we calculate $a \otimes_A 1 = cx^\beta \cdot b \otimes_A 1 =  b \otimes_A cx^\beta = 0$ because $cx^\beta \in \m$. 

To understand what the functor does on maps, consider first as a special case a graded matrix $M \in \K^{R \times G}$. To compute $M \otimes A/ \m$, consider an entry $M_{i,j} \in \K$ where $R(j) > G(i)$. It indicates that $M$ maps the homogeneous element $1 \in \K = A[-R(j)]_{R(j)}$ associated to the column $j$ to the homogeneous element $M_{i,j} \cdot x^{R(j)-G(i)} \in A[-G(i)]_{R(j)}$, which is an image of the generator $1 \in A[-G(i)]_{G(i)}$. But after applying the functor this last element is zero, so $(M \otimes A/ \m)_{i,j} = 0 $.

Summarising, $M \otimes A/ \m$ is again a graded matrix which contains only the entries of $M$ with  $G(i) = R(j)$ and is $0$ elsewhere.

\begin{proposition}\label{prop:minimal_pres}\cite[7.1]{Peeva}
Let  $ P_1 \oto{f} P_0 \to X$ be a presentation and $\ker f \oto{i}  P_1 $ its kernel.

The presentation $f$ is $\emph{minimal}$ if and only if
 \[f \otimes_A A/\mathfrak{m} = 0 \ \text{ and }  \ i \otimes_A A/\mathfrak{m}= 0.\]
\end{proposition}

We postpone the proof and highlight the advantages of this  definition. From an algorithmic point of view, it means that all diagonal blocks of the presentation matrix and the previous matrix in the resolution must be 0. For the latter one does not actually need to compute the map $i$, since this simply means that after column reduction of $f$ there are no zero columns. This also means that minimising can be done by column-reduction and subsequent deletion of generator-relation pairs which violate the first equation. From an algebraic point of view, it will allow us to verify minimality of presentations in a quick and concise way using the right exactness of $- \otimes A/ \m$ and long exact sequences. 

To find a minimal presentation for a module $X$, choose a basis of $X/ \m X = X \otimes X/ \m$ and any set of preimages along the quotient map $X \onto X / \m X$. This will be a set of minimal generators (formally, this is a consequence of Nakayama's lemma). Repeat this process with the kernel of the induced map from the free module.

 \begin{proposition}[\cite{CZMP}, Theorem 6]\label{free_conservative}
 A map between graded free persistence modules $\varphi \colon X \to Y$ is an isomorphism if and only if $\varphi \otimes A/ \m$ is an isomorphism.
\end{proposition}

The computational content of this proposition is: Given a graded matrix, we can decide if it is invertible by only looking at the entries which have the same row and column degree.

\begin{proof}
Choose bases $A^G \iso X$ and $A^{G'} \iso  Y$ and let $\varphi$ be a map $A^{G} \to A^{G'} $. Group the basis vectors together in sets of the same degree. For any $\alpha \in \Z^d$ let $k_\alpha$ denote the number basis vectors of this degree. This makes $\varphi$ a graded matrix in block form. Apply $- \otimes A/ \m$ and assume that the result is an isomorphism:
 \[ \left(\varphi_{\alpha, \beta} \right)_{\alpha, \beta \in \Z^d} \colon \bigoplus_{\alpha \in \Z^d} A^{k_\alpha}(-\alpha) \to \bigoplus_{\beta \in \Z^d} A^{k'_\beta}(-\beta) \]
\[\quad \quad \quad \Downarrow (- \otimes A / \m )\] 
\[ \left(\varphi_{ \alpha, \beta} \otimes A / \m \right)_{\alpha, \beta \in \Z^d} \colon \bigoplus_{\alpha \in \Z^d} \K^{k_\alpha}(-\alpha) \iso \bigoplus_{\beta \in \Z^d} \K^{k'_\beta}(-\beta) \]

The non-diagonal terms $\left(\varphi_{ \alpha, \beta} \otimes A / \m \right)$ where $\alpha \neq \beta$ vanish because the vector spaces do not have the same degree. Then the blocks on the diagonal $\varphi_{ \alpha, \alpha} \otimes A/ \m$ are invertible matrices. \\
Choose any linear order $<_l$ on $\Z^d$ which extends the product order. If $\alpha <_l \beta$ then $\varphi_{\alpha, \beta}= 0$ so $\varphi_{\alpha, \beta}$ is block-triangular after reordering the rows and columns according to $<_l$. Since the diagonal blocks are invertible the whole matrix must be invertible, too.
\end{proof}

We are now ready to prove the reverse of \autoref{admissible_transformation}. To translate the statement to admissible operations, observe that the standard proof which shows that an invertible matrix is a product of elementary matrices by Gaussian Elimination translates 1-1 to the graded setting.

\begin{proposition}\label{minimal_iso}
 Consider a morphism of minimal presentations.
\[\begin{tikzcd}
	{P_1} & {P_0} & X & 0 \\
	{Q_1} & {Q_0} & Y & 0
	\arrow["\varphi", from=1-3, to=2-3]
	\arrow["{\varphi_0}", from=1-2, to=2-2]
	\arrow["{\varphi_1}", from=1-1, to=2-1]
	\arrow["{d_1}", from=1-1, to=1-2]
	\arrow["{d_0}", from=1-2, to=1-3]
	\arrow["{e_0}", from=2-2, to=2-3]
	\arrow["{e_1}", from=2-1, to=2-2]
	\arrow[from=2-3, to=2-4]
	\arrow[from=1-3, to=1-4]
\end{tikzcd}\]
If $\varphi$ is an isomorphism, then so are $\varphi_0$ and $\varphi_1$.
\end{proposition}

\begin{proof}
 Apply the right-exact functor $- \otimes A / \m$ to the diagram. By exactness at $Q_0 / \m Q_0$ and $P_0 / \m P_0$ and using minimality, the maps $d_0 \otimes A/ \m$ and $e_0 \otimes A/ \m$ are isomorphisms. 
 
 Then by commutativity $\varphi_0 \otimes A / \m$ is an isomorphism and using \autoref{free_conservative} also $\varphi_0$ is an isomorphism. The statement for $\varphi_1$ follows analogously considering $\ker d_0$ and $\ker e_0$ instead of $X,Y$.
\end{proof}

\begin{remark}
 It is well known that every finitely generated projective persistence modules is free: 
 Let $P$ be projective.
 Choose a minimal free cover $F \to P$. This map splits by projectivity, but the elements in the kernel lie completely in $\mathfrak{m}F$ so by Nakayama's Lemma this kernel must be $0$.
\end{remark}

We will also need the following direct consequence of Nakayama's lemma:

\begin{proposition}\label{nakayama_surjection}
    Let $Y$ be finitely generated, then $f\colon X \to Y$ is surjective if and only if $f \otimes A/ \m$ is.
\end{proposition}

With this we can then prove that the two definitions of minimality are equivalent, via a third characterisation:

\begin{proposition}\label{prop:essential}
A homomorphism $f \colon X \to Y$ of fg persistence modules is an essential surjection if and only if $f \otimes A/ \m$ is an isomorphism.
\end{proposition}

\begin{proof}
"$\Leftarrow$": By \autoref{nakayama_surjection} $f$ is a surjection. Let $i \colon X' \into X$ be injective, but nor surjective, then by \autoref{nakayama_surjection} $i \otimes A/m $ is not surjective either. But then $f \circ i \otimes A/ \m$ is not surjective and so $f \circ i$ cannot be a surjection. \\
"$\Rightarrow$": By \autoref{nakayama_surjection} $f \otimes A/ \m$ is a surjection of (graded) vector spaces, so it splits and there is a graded sub-vector space $E \into X/ \m$ s.t. ${f \otimes A/ \m}_{|E}$ is an isomorphism. Lift $E$ to a submodule $\tilde E$ of $X$ by pulling back $E$ along $X \onto X/\m$. Then $f_{| \tilde E} \otimes A/\m = {f \otimes A/ \m}_{|E} $ is an isomorphism and so $f_{| \tilde E}$ is surjective. But since $f$ was essential, $\tilde E$ cannot have been a proper submodule. Then $E$ was not a proper graded sub-vector space either and $f \otimes A/ \m$ is an isomorphism.
\end{proof}

\begin{proof}[Proof of \autoref{prop:minimal_pres}]
Assume that a finite presentation $F_1 \oto{f} F_0 \oto{p} X \to 0$ satisfied $f \otimes A/\m = 0$, then $p \otimes A/\m$ is an isomorphism and $p$ an essential surjection by \autoref{prop:essential}. Now if there was another presentation $F'_1 \oto{f'} F'_0 \oto{p'} X \to 0$  of $X$, then the identity of $X$ lifts to a map $h\colon F'_0 \to F_0$. Since $p'$ is surjective we have $p(\Ima h) = X$ and since $p$ is an essential surjection $\Ima h$ could not have been a proper submodule, so $F'_0$ could not have a smaller basis than $F_0$.
\end{proof}

At last we can use the functor $-\otimes A/ \m$ to define the graded Betti numbers:

\begin{definition}[{\cite[Theorem 11.2.]{Peeva}}]
Let $X$ be a finitely generated persistence module, $\alpha \in \Z^d$, and $i \in \N$. Recall that the \emph{graded betti number} $b_{i,\alpha}(X)$ denotes the number of summands of the form $A(-\alpha)$ in the $i$-th part of a minimal resolution of $X$. By \autoref{prop:minimal_pres} we have
\[b_{i, \alpha}(X) \coloneqq \dim_\K \Tor_{i}(X, A/\m) _\alpha = \dim_\K  \Ext^{i}(X, A/\m)_\alpha\]
 and will say that a module $X$ \emph{has no generators at (above)} $\alpha$ if $b_{0,\alpha}=0$ ( $b_{0,\beta}=0 \ \forall \beta \geq \alpha$ ) and use the same language for subsets $U \subset \Z^d$ and relations.
\end{definition}

\begin{remark}
    To see that the $\Ext$-definition works, too, notice that we could just as well have used the functor $\Hom(-, A/ \m)$: Using the enrichment over $\grA$ and the $\otimes \dashv \Hom$-adjunction there is a natural equivalence:
    \begin{align*} 
    \Hom\left(-, A/ \m  \right)  \simeq \Hom_{A/ \m} \left( A/ \m, \Hom \left(-, A/ \m\ \right) \right) 
      \simeq \Hom_{A/ \m}(- \otimes A/ \m, A/ \m) = \left(- \otimes A/ \m \right)^{\vee}  \end{align*}
\end{remark}

\subsection{Decompositions}\label{sec:decompositions}

\begin{definition}
    Let $\Cs$ be an additive category and $X \in \Cs$. A decomposition is an isomorphism $\varphi \colon X \iso \bigoplus_{i \in I} X_i$ where $X_i \neq 0$ for each $i$. An object is called indecomposable if there is no decomposition with $|I| > 1$.
\end{definition}

As long as we are interested in finitely generated modules we can be sure that the indecomposable summands we find are unique up to isomorphism. 

\begin{theorem}[\autoref{krull-remak-schmidt}, Krull-Remak-Schmidt-Azumaya \cite{Azumaya}]
 Every fg persistence module $X$ has a finite decomposition into indecomposable submodules, $X \simeq \bigoplus_{i \in I} X_i$, and the isomorphism types of the submodules are independent of the chosen decomposition.
\end{theorem}

In fact by a result of Botnan and Crawley-Boevey \cite{CBBB} this is even true for pointwise finite dimensional persistence modules.

We need to understand how decompositions of modules relate to their presentations. 

\begin{definition}[{\cite[Definition 4.2]{DeyXin}}]\label{def:blockstructure}
 Let $M \in \K^{m \times n}$ be a (graded) matrix. 
 A \emph{block} $b$ of $M$ is a pair $(b_{\text{rows}}, \, b_{\text{cols}})$ where $b_{\text{rows}} \subset [m]$ and $ b_{\text{cols}} \subset [n]$ are index-sets and we write $M_b$ for the corresponding submatrix. 
 
 $M$ is called block \emph{decomposed} if there is a set of blocks $\B$ which partition the row and column indices such that $M$ is zero outside of the blocks. We write $M_\B = \bigoplus M_b$ for a matrix which is block-decomposed with $\B$ the set of blocks and if the matrix is graded, then $G_b$, $R_b$ denote the corresponding degrees of $b$.
\end{definition}
Equivalently, a block decomposition is a reordering of a direct sum of matrices. Therefore it induces a direct sum decomposition of the presented modules into the modules presented by the blocks.

\begin{example}\label{example:blockstructure} A matrix decomposed over $\left( b =  \left( \{1,2\}, \{1\} \right), \ c = \left( \{3\}, \{2\} \right), \ d = \left( \{4\}, \{\} \right) \right)$
 
 \begin{minipage}[c]{0.3\textwidth}
 \[ 
\begin{array}{| c | c  c | }
        \hline 
         M  & \lbcell (2,2) &  \lrcell (2,2)  \\
        \hline
         \lbcell (0,1)    &  \lbcell 1   &   0       \\
         \lbcell (1,0)    &  \lbcell 1   &   0       \\
        \lrcell (1,1)    &  0   & \lrcell 1         \\
        \lgcell (1,2)    &  0   &  0                \\
        \hline
    \end{array}
 \]
\end{minipage}
\hspace{3em}
 \begin{minipage}[c]{0.5\textwidth}
        \begin{tikzpicture}[scale = 0.9]
            \draw[->, thick] (0,0) -- (4,0) node[right] {$x$};
            \draw[->, thick] (0,0) -- (0,4) node[above] {$y$};
        
            \draw[fill=lightblue, draw=none, opacity=1, thick] (0,3) -- (0,1) -- (1,1) -- (1,0) -- (3,0) -- (3,3);
            \draw[fill = darkblue, draw=none, opacity=1, thick]
            (1,3) -- (1,1) -- (3,1) -- (3,2) -- (2,2) -- (2,3);
    
            \draw[fill = lightred, draw= none, opacity = 0.8, thick]
            (1.1,3.1) -- (1.1,1.1) -- (3.1,1.1) -- (3.1,2.1) -- (2.1,2.1) -- (2.1,3.1);

            \draw[fill = Gray!200, draw=none, opacity = 0.4, thick]
            (1.0,3.2) -- (1.0,2.0) -- (3.2, 2.0) -- (3.2, 3.2);
            
            \foreach \x/\y in {0/1, 1/0} {
                \fill[lightblue, draw = black] (\x,\y) circle (3pt);
            }
        
            \foreach \x/\y in {2/2} {
                \node[rectangle, draw=black, fill=lightblue, inner sep=2pt] at (\x,\y) {};
            }

            \foreach \x/\y in {1.1/1.1} {
                \fill[lightred, draw = black] (\x,\y) circle (3pt);
            }

            \foreach \x/\y in {2.1/2.1} {
                \node[rectangle, draw=black, fill=lightred, inner sep=2pt] at (\x,\y) {};
            }

            \foreach \x/\y in {1/2} {
                \fill[Gray, draw = black] (\x,\y) circle (3pt);
            }

            \foreach \x in {0,1,2,3} {
                \node at (\x, -0.2) [below] {\x};
            }
            \foreach \y in {0,1,2,3} {
                \node at (-0.2, \y) [left] {\y};
            }
        \end{tikzpicture}  \hfill
\end{minipage} \hfill \\
In the picture we have shifted the shapes a bit to the top right for better visibility.
\end{example}

The important consequence of the machinery in the preceding \hyperref[formal_introduction]{Subsection 2.1} which we need is the following.

\begin{proposition}\label{iso_lifting}
Consider a decomposition of a persistence module $\varphi \colon X \simeq \bigoplus_\B X_b$. If $M$ minimally presents $X$, then there is a sequence of admissible operations which transforms $M$ into a block decomposed matrix $M_\B$ such that $M_b$ minimally presents $X_b$.
\end{proposition}

\begin{proof}
Choose minimal presentations $M_b$ for each $X_b$. Using \autoref{prop:minimal_pres} the block-decomposed matrix $M_\B = \bigoplus M_b$ minimally presents $\bigoplus X_b$. We lift $\varphi$ to a pair of graded matrices $(Q, \, P)$ such that $QMP^{-1} =  M_\B$. Via \autoref{minimal_iso} we know that $Q$ and $P^{-1}$ are invertible, so they can be written as a product of graded elementary matrices.
\end{proof}

\begin{remark}\label{remark:story}
\autoref{iso_lifting} points to the main difficulty in writing a matrix-elimination type algorithm. Assume that we want to bring a presentation $M$ into a certain shape, that is with zeroes in a block $b$. Then this is equivalent to finding invertible graded matrices $Q,P$ such that $(QMP)_b=0$. A priori this is a \emph{quadratic} system of equations in the entries of $Q$ and $P$. 
\end{remark}

 We want to conclude this section with a warning. Since maps between modules have many homotopic but non-isomorphic lifts to minimal presentations, the correspondence from decompositions of minimal presentations to decompositions of modules is only surjective. Theorem 3.1. in \cite{DeyXin} claims this correspondence to be bijective. Of course this does not change the correctness of the algorithm, but we give a counterexample.

 \begin{definition}\label{decomposition}
 Two decompositions $\varphi \colon X \iso \bigoplus_{i \in I} X_i$ and $\varphi' \colon X \iso \bigoplus_{i \in I} X'_i$ are called \emph{equivalent} if there is a permutation/bijection $\sigma \colon I \iso I$ and a commutative diagram
\[\begin{tikzcd}[ampersand replacement=\&]
	\& {\bigoplus_{i \in I} X_i} \\
	X \\
	\& {\bigoplus_{i \in I} X'_{\sigma(i)}}
	\arrow["\varphi", from=2-1, to=1-2]
	\arrow["{\varphi'}", from=2-1, to=3-2]
	\arrow["{\bigoplus_{i \in I} f_i}", from=1-2, to=3-2]
\end{tikzcd}.\]
\end{definition}

 \begin{example}
 Consider the indecomposable modules $X, \, Y$ over some field $\K$ presented by 
 \[
  \begin{array}{| c | c | }
    \hline 
     M_X & \lrcell (1,1) \\
    \hline
    \lrcell (0,0)&  1   \\
 \hline
\end{array} \quad \quad 
\begin{array}{| c | }
    \hline 
     M_Y \\
    \hline
    \blcell \wtext{(1,1)}    \\
 \hline
\end{array}
\quad  \quad 
\begin{tikzpicture}[baseline=8ex, scale = 0.9]
            \draw[fill=lightred, draw=none, opacity=1, thick] (0,3) -- (0,0) -- (3,0) -- (3,1) -- (1,1) -- (1,3);
            
            \foreach \x/\y in {0/0} {
                \fill[lightred, draw = black] (\x,\y) circle (3pt);
            }
            \foreach \x/\y in {1/1} {
                \node[rectangle, draw=black, fill=lightred, inner sep=2pt] at (\x,\y) {};
            }
    \end{tikzpicture}
    \quad \quad 
    \begin{tikzpicture}[baseline=8ex, scale = 0.9]
            \draw[fill=lightblue, draw=none, opacity=1, thick] (1,3) -- (1,1) -- (3,1) -- (3,3);
            
            \foreach \x/\y in {1/1} {
                \fill[lightblue, draw = black] (\x,\y) circle (3pt);
            }

    \end{tikzpicture}.
 \]
Consider the module $X \oplus Y$. As there are no maps $X \to Y$ nor $Y \to X$, the only decomposition of this module is given by the identity $X \oplus Y \iso X \oplus Y$. On the other hand, consider the following two isomorphisms of presentations
\[\varphi = \left( \begin{bmatrix}
                             1 & 0 \\
                             0 & 1
                            \end{bmatrix}, 1 \right), \ \varphi' = \left( \begin{bmatrix}
                             1 & 1 \\
                             0 & 1
                            \end{bmatrix}, 1 \right) \colon
    \begin{array}{| c | c | }
    \hline 
     M_X \oplus M_Y &  \lrcell (1,1) \\
     \hline
     \lrcell (0,0) & 1 \\
    \hline
    \blcell \wtext{(1,1)}  & 0  \\
 \hline
\end{array}  \iso 
\begin{array}{| c | c | }
    \hline 
     M_X \oplus M_Y &  \lrcell (1,1) \\
     \hline
     \lrcell (0,0) & 1 \\
    \hline
    \blcell \wtext{(1,1)}  & 0  \\
 \hline
\end{array} 
\]

Both induce the identity map on the module $X \oplus Y$, but they are non-equivalent as decompositions of presentation matrices. By \autoref{decomposition} they would be equivalent if there exist automorphisms $f \in \Aut(M_X), \, g \in \Aut(M_Y)$ such that $(f \oplus g) \circ \varphi = \varphi'$. A quick calculation shows that $\Aut(M_X)= \Aut(M_Y) = \K^*$. That is, for each $a, b \neq 0 \in \K$ there is exactly one map in $\Aut(M_X) \times \Aut(M_Y)$ of the form
\[f_a \oplus g_b = \left( \begin{bmatrix}
                             a & 0 \\
                             0 & b
                            \end{bmatrix},  a \right).\]
        But no choice of $a, b$ satisfies the equation $(f_a \oplus g_b) \circ \varphi = \varphi'$.                     

\end{example}

\section{Decomposing Uniquely Graded Presentations}

The Generalized Persistence (GP) Algorithm introduced by Dey and Xin~\cite{DeyXin} finds an indecomposable decomposition of a presentation matrix $M \in \K^{G \times R}$ assuming that all relations have pairwise distinct degrees. We review the core functionality of their algorithm. 

\subsection{The Generalized Persistence Algorithm}\label{sec:deyxin}


The GP-Algorithm loops over the relations in an order compatible with the product order on $\N^d$,
similarly to the Persistence Algorithm for 1-parameter persistence modules~\cite{eh-computational,phat}. 
It manipulates the matrix via admissible
row and column operations and
maintains the invariant that after
having handled the $i$-th column,
the matrix consisting of the first
$i$ columns is block-decomposed into indecomposables.
We must therefore decompose a graded matrix that is already decomposed
except for its last column. This is reflected in the next definition; 
note that the GP algorithm assumes $N$ to consist of a single column,
but we already define the structure for $N$ of arbitrary size
to prepare our extension of the algorithm.

\begin{definition}\label{def:alpha_decomposed}
Let $\alpha \in \N^d$ and $\left[ M_\B \ N \right] \in \K^{G \times ( R \cup k \cdot \alpha) }$ be a presentation, where $N$ contains $k$ columns, all of degree $\alpha$. It is called
$\alpha$-\emph{decomposed}, if it is \emph{minimal}, $M_\B$ is an indecomposable decomposition, and the degree of no column of $M_\B$ is larger than $\alpha$.
For $b \in \B$ we denote by $N_b$ the submatrix $N_{b_{\text{rows}} \times k \cdot \alpha}$ of $N$ with the row indices of $b$.
\end{definition}

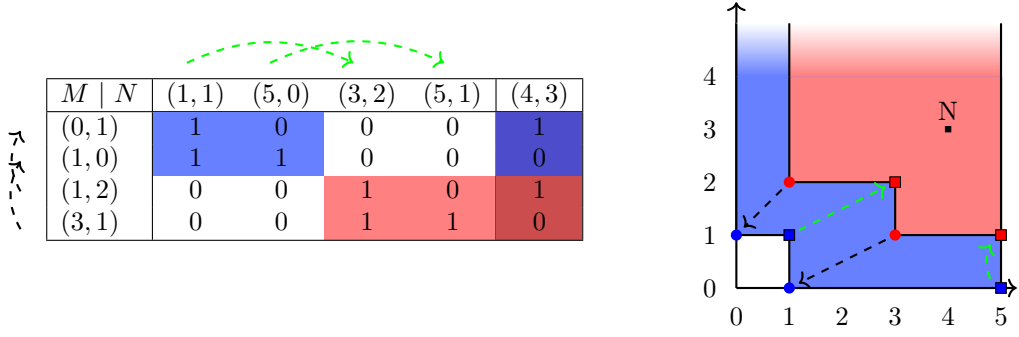
\begin{figure}[H]
    \centering
    \begin{minipage}[c]{0.57\textwidth}

    \[ \hspace{0.3em}
    \begin{array}{| l | c  c  c  c | c |}
        \hline 
        M \ | \ N & (1,1) & (5,0) & (3,2) & (5,1) & (4,3)   \\
        \hline
        (0,1)    &  \lbcell 1   &  \lbcell 0    &  0    & 0     & \bmcell 1 \\
        (1,0)    &  \lbcell 1   &  \lbcell 1    &  0    & 0     & \bmcell 0 \\
        (1,2)    &  0   &  0    & \lrcell 1    & \lrcell  0     & \gmcell 1 \\
        (3,1)    &  0   &  0    & \lrcell 1    & \lrcell  1     & \gmcell 0 \\
        \hline
    \end{array}
    \]
    \begin{tikzpicture}[overlay, remember picture]
        \draw[->, thick, black, bend left, dashed] (0.7, 0.4) to (0.7, 1.3); 
        \draw[->, thick, black, bend left, dashed] (0.7, 0.9) to (0.7, 1.7); 

        \draw[->, thick, green, bend left=25, dashed] (2.9, 2.6) to (5.1, 2.6);
        \draw[->, thick, green, bend left=25, dashed] (4.0, 2.6) to (6.3, 2.6);
    \end{tikzpicture}
\end{minipage} 
\hspace{0.3 cm}
\begin{minipage}[c]{0.32\textwidth}   
        \begin{tikzpicture}[scale=0.7]
        \draw[->, thick] (0,0) -- (5.3,0) ;
        \draw[->, thick] (0,0) -- (0,5.4) ;


            \fill[shading=axis, bottom color=lightblue, top color=white, draw=none, opacity=1] 
            (0,5) -- (0,4) -- (5,4) -- (5,5);

            \fill[shading=axis, bottom color=lightred, top color=white, draw=none, opacity=1] 
            (1,5) -- (1,4) -- (5,4) -- (5,5);
            
            \draw[color = white]
            (0,5) -- (5,5);

            \draw[color = lightblue]
            (0,4) -- (1,4);
            
            \draw[color = lightred]
            (1,4) -- (5,4);

        \draw[fill=lightblue, opacity=1, thick] (0,4) -- (0,1) -- (1,1) -- (1,0) -- (5,0) -- (5,4);
        \draw[thick]
        (0,4)--(0,5);
    
        \draw[fill=lightred, opacity=1, thick] (1,4) -- (1,2) -- (3,2) -- (3,1) -- (5,1) -- (5,4);
        \draw[thick]
        (1,4)--(1,5);
        \draw[thick]
        (5,4)--(5,5);
        
        \foreach \x/\y in {0/1, 1/0} {
            \fill[blue] (\x,\y) circle (3pt);
        }
    
        \foreach \x/\y in {1/1, 5/0} {
            \node[rectangle, draw=black, fill=blue, inner sep=2pt] at (\x,\y) {};
        }
    
        \foreach \x/\y in {1/2, 3/1} {
            \fill[red] (\x,\y) circle (3pt);
        }
    
        \foreach \x/\y in {3/2, 5/1} {
            \node[rectangle, draw=black, fill=red, inner sep=2pt] at (\x,\y) {};
        }
        
        \node[rectangle, draw= black, fill = black, inner sep=1pt] at (4,3) {};
    
        \node at (4,3) [above] {N};

        
        \draw[->, thick, black, dashed, shorten >=4pt, shorten <=4pt] (1,2) -- (0,1);
        \draw[->, thick, black, dashed, shorten >=4pt, shorten <=4pt] (3,1) -- (1,0);

        \draw[->, thick, green, dashed, bend left = 30, shorten >=4pt, shorten <=4pt] (4.9,0) to (4.9,1);
        \draw[->, thick,  green, dashed, shorten >=4pt, shorten <=4pt] (1,1) -- (3,2);

        \foreach \x in {0,1,2,3,4,5} {
            \node at (\x, -0.2) [below] {\x};
        }
        \foreach \y in {0,1,2,3,4} {
            \node at (-0.2, \y) [left] {\y};
        }
    \end{tikzpicture}
\end{minipage}
    \caption{A (4,3)-decomposed presentation. There are two blocks in $M_\B$ and they present the interval modules (\autoref{def:interval}) on the right. Some admissible operations have been indicated by arrows both in the matrix and between the corresponding generators and relations.}
    \label{fig:alpha_decomposed}
\end{figure}

\subparagraph{\texttt{BlockReduce}} \  If we can decompose $\alpha$-decomposed presentations, then by iterating over the relations we can decompose all finite presentations. If $N$ has only one column this means we have to zero out every sub-column $N_b$ of $N$ whenever possible to reach the block-decomposition with the most blocks. This is done by a subroutine which Dey and Xin called \texttt{BlockReduce}, a name we will keep. 

\begin{example}\label{example:blockreduce}
We illustrate \texttt{BlockReduce} on the example given in \autoref{fig:alpha_decomposed}. Let us call the (blue) upper block $b$ and the lower (red) block c. We will use two types of admissible operations to try to zero out $N_b$:
\begin{enumerate}[label=(\alph*)]
 \item Column operations from $M_b$ to $N$, e.g. adding the first column to the fifth column.
 \item Row operations from $M_c$ to $M_b$, e.g. adding the fourth row to the second row.
\end{enumerate}

Adding the third row to the first would delete all entries in $N_b$, but this would also destroy the block-decomposition of $M_\B$, placing a $1$ at position $M_{1,3}$: \vspace{1em}

\begin{minipage}[c]{0.45\textwidth}
\resizebox{\textwidth}{!}{ $
\begin{array}{| l | c  c  c  c | c |}
        \hline 
        M \ | \ N & (1,1) & (5,0) & (3,2) & (5,1) & (4,3)   \\
        \hline
        (0,1)    &  \lbcell 1   &  \lbcell 0    &  \otext{1}    & 0     & \bmcell \otext{0} \\
        (1,0)    &  \lbcell 1   &  \lbcell 1    &  0    & 0     & \bmcell 0 \\
        (1,2)    &  0   &  0    & \lrcell 1    & \lrcell  0     & \gmcell 1 \\
        (3,1)    &  0   &  0    & \lrcell 1    & \lrcell  1     & \gmcell 0 \\
        \hline
    \end{array}
    $
    \begin{tikzpicture}[overlay, remember picture]
        \draw[->, thick, black, bend left, dashed]
        (-3.2, -0.4) to (-3.2, 0.6); 

    \end{tikzpicture}
    }
\end{minipage} 
\begin{minipage}[c]{0.45\textwidth}
\resizebox{\textwidth}{!}{ 
    $
    \begin{array}{| l | c  c  c  c | c |}
        \hline 
        M \ | \ N & (1,1) & (5,0) & (3,2) & (5,1) & (4,3)   \\
        \hline
        (0,1)    &  \lbcell 1   &  \lbcell 0    &  \otext{1}    & 0     &  \otext{0} \\
        (1,0)    &  \lbcell 1   &  \lbcell 1    &  \otext{1}     & \otext{1}      &  0 \\
        (1,2)    &  0   &  0    & \lrcell 1    & \lrcell  0     & \gmcell 1 \\
        (3,1)    &  0   &  0    & \lrcell 1    & \lrcell  1     & \gmcell 0 \\
        \hline
    \end{array}
    $
    \begin{tikzpicture}[overlay, remember picture]

        \draw[->, thick, green, bend left=25, dashed] (-5.1, 0.6) to (-3.1, 0.6); 
        \draw[->, thick, green, bend left=25, dashed] (-4.0, 0.1) to (-2.0, 0.1); 

    \end{tikzpicture}
}
\end{minipage} \vspace{1em} \\
 However, if we also add the fourth to the second row and then use the two column operations indicated by the $2$ (green) arrows (right side), we can revert the alteration of $M_\B$.  \\ 
We thus need a third type of admissible operations: \textbf{\textcolor[gray]{0.3}{(c)}} Column operations from $b$ to $c$.
    
\end{example}

The operations of type $(a)$, $(b)$, and $(c)$ have a pleasant feature: their effect on the matrix does not change if the order of their application changes. Dey-Xin call such sets of operations \emph{independent}. In the setting of \autoref{remark:story}, when restricting the entries of $Q$ and $P$ to those belonging to independent operations, the system of equations becomes \emph{linear}. 

Dey and Xin observed that following this strategy, repeated application of the operations (a)-(c) actually suffices to decompose an $\alpha$-decomposed presentation with $k=1$ assuming that the generators are also uniquely graded. We prove a more general statement that removes this assumption. 

\begin{lemma}\label{lemma:principle}
Let $\left[ M_\B \ N \right]$ be $\alpha$-decomposed, then after a suitable sequence of elementary column operations on $N$, there is a sequence of independent admissible operations which finds an indecomposable decomposition of
$\left[M_\B \ N \right]$ without altering $M_\B$.
\end{lemma}

We will later prove a slightly upgraded version in the form of \autoref{matrices_cd}.

At last assume again that $k=1$. After having zeroed out as much as possible of $N$, the blocks $b$ for which $N_b$ are still non-zero must all \emph{merge} into a new block, since their non-zero entries now overlap in this last column. This finishes the loop of the GP Algorithm.

\subparagraph{Extending \texttt{BlockReduce}}\label{hom_algorithm}
Consider an $\alpha$-decomposed presentation $\left[M_\B \ N\right]$ and let $k\geq 1$ now be arbitrary, where we recall that $k$ is the number of columns of $N$.
At first, we are still interested in deleting the entire sub-matrices $N_b$ for every $b\in \B$ wherever this is possible. 
Consider now \emph{sets} of admissible operations of the types $(a)$-$(c)$ in \autoref{example:blockreduce}, encoded as graded matrices. Using $M_{b,c} \coloneqq M_{b_{\text{rows}} \times c_{\text{cols}}}$, their effect on $\left(M_\B \ N \right)$ can be described:
\begin{enumerate}[label=(\alph*)]
 \item A graded matrix $U \in \K^{R_b \times k \cdot \alpha}$ and $N_b \leftarrow N_b + M_b U$.
 \item A graded matrix $Q \in \K^{G_b \times G_c}$ and $N_b \leftarrow N_b + Q N_c, \quad  M_{b,c}  \leftarrow M_{b,c} + Q M_c $.
 \item A graded matrix $P \in \K^{R_b \times R_c}$ and $M_{b,c}  \leftarrow M_{b,c} + M_b P$.
 \end{enumerate}
Finding such matrices $Q$ and $P$ such that after performing $(b)$ and $(c)$ the whole matrix $M_\B$ is preserved is equivalent to demanding $M_{b,c} = 0$, which is a \emph{linear} system \hyperref[hom_computation]{($\ast$)}. Finding those which zero out $N_b$ is another one \hyperref[hom_computation]{($\ast \ast$)}.

\subparagraph{\texttt{BlockReduce} (General Version):}\label{subparagraph:blockreduce} \ \newline
\texttt{Input}: A $\alpha$-decomposed matrix $\left[M_\B \ N\right] \in K^{G \times (R \cup k \cdot \alpha)}$ with $k$ columns at $\alpha$, $b \in \B$.
\begin{align}
 \text{Find } U \in \K^{R_b \times  k \cdot \alpha} \text{ and } \forall c \neq b \in \B  & \colon 
 \quad  Q_{c} \in \K^{ G_b \times G_c}, \  P_c \in \K^{G_b \times R_c}, \text{ such that } \notag \\
 \ Q_c M_{c} + M_b P_c  = 0 \ \ (\ast) \ \ &
 \text{ and } 
  \ N_b + \sum Q_cN_{c} + M_bU = 0 \ \ (\ast \ast) \label{hom_computation}
\end{align}

In \cite{DeyXin}, Dey and Xin write these equations in a large matrix which they call $S$-matrix. 
 
\subparagraph{Run time of \hyperref[subparagraph:blockreduce]{\texttt{BlockReduce}.} }\label{deyxin_running_time}
Let $(M_\B \ N)$ be $\alpha$-decomposed with $M$ of size $m \times n$, then the size of every block is also bounded by $m \times n$. It follows that \hyperref[subparagraph:blockreduce]{($\ast$)} contains at most $(m^2 + n^2)$ variables and $mn$ equations and \hyperref[subparagraph:blockreduce]{($\ast \ast$)} contains at most $m^2 + nk$ variables and $mk$ equations. Let $\omega<2.373$ denote the exponent for matrix multiplication.

\begin{proposition}[cf {\cite[Algorithm 3]{DeyXin}}]\label{proposition:run_time_block}
 The time complexity of solving \\ \hyperref[hom_computation]{($\ast$)} is in 
 $\Oc \left( \left(n m\right)^{\omega-1} \left(m^2 + n^2 \right) \right)$ 
 and of solving \hyperref[hom_computation]{($\ast \ast $)} is in 
 $\Oc\left( \left(mk\right)^{\omega-1}\left(m^2 + nk\right) \right).$
\end{proposition}

\subsection{Modifications of the GP-Algorithm}\label{sec:modifications} 
\textsc{aida} implements many modifications to \hyperref[subparagraph:blockreduce]{\texttt{BlockReduce}} which make this and following subroutines faster and enable the \emph{Automorphism-invariant} $\alpha$-decomposition which we introduce in \autoref{sec:aida}.

\subparagraph{Column-Sweeps.}\label{colsweeps}
When processing a new set of columns $N$ of degreee $\alpha_i$, we first try to zero out each $N_b$ only with column-operations (i.e., ignoring row-operations first), which is often sufficient and, if not, will maybe reduce the number of non-zero entries. To this end we reduce the sub-matrix $M_b^{\leq \alpha_i}$ of $M_b$ containing the columns of degree $\leq \alpha_i$. Importantly, this matrix can be stored and sometimes even updated when the next batch is considered, in the same way as in the
standard persistence algorithm for one parameter.

\subparagraph{Hom-Variant of \texttt{BlockReduce}:}\label{subparagraph:hom_variant} The following change is subtle but crucial:
 \begin{enumerate} \vspace{0.5em}
 \item Find a basis $(Q_i, -P_i)_{i \in I}$ of the solution-space of \hyperref[hom_computation]{($\ast$)} for every $c$ and store it.
 \item Search for the solution of \hyperref[hom_computation]{($\ast \ast$)} in the vector spaces spanned by $(Q_i, -P_i)_{i \in I}$ for each $c$.
 \end{enumerate}
 \vspace{0.5em} This split makes no computational difference asymptotically, and its overhead is negligible.
In subsequent algorithms, we will call {\texttt{BlockReduce}} recursively
without changing $M_\B$ (thanks to Lemma~\ref{lemma:principle}).
By solving and storing \hyperref[hom_computation]{($\ast$)}, we thus avoid recomputations during the recursion.

\subparagraph{Morphisms of Presentations.}

Let $M_c \in \K^{G_1 \times R_1}, \ M_b \in \K^{G_2 \times R_2}$ present persistence modules $X_c, \, X_b$. Up to sign, the solutions of {\texttt{BlockReduce}} \hyperref[hom_computation]{($\ast$)} form the $\K$-vector space
 \[ \Hom\left(M_c, \, M_b \right) \coloneqq \left\{ (Q, \,P)  \in \K^{G_2 \times G_1} \times  \K^{R_2 \times R_1} \  | \  Q M_c = M_b P \right\}\]
and there is a natural $\K$-linear surjection $\Hom\left(M_c, \, M_b \right) \onto \Hom(X_c, \, X_b)$.

 So when solving {\texttt{BlockReduce}} \hyperref[hom_computation]{($\ast$)} for each pair $c \neq b \in \B$, we are actually computing, akin to elementary matrices, the generators of a subgroup of $\Aut \left( M_\B \right)$. Additionally, the property of a morphism $(Q, \, P) \in \Hom\left(M_c, \, M_b \right)$ to be part of a solution of {\texttt{BlockReduce}} \hyperref[hom_computation]{($ \ast \ast$)} depends only on the map \\
 $\widetilde Q \colon X_c \to X_b$ which it induces. 

\subparagraph{Fast $\Hom$-computation.}\label{sub:fast_hom} Let $X, Y$ be persistence modules, then $\Hom(X, Y)$ can be computed faster if $Y$ is pointwise "small". This is because we can restrict the target space of each generator of $X$. 
We will explore this thoroughly in future research and also point to the following known result for the "smallest" modules, also due to Dey and Xin.

\begin{proposition}[\cite{dey_xin}, Proposition 16 and Algorithm \textsc{interleaving} in sections 4.1 and 4.2  \text]\label{prop:interval_hom}
Let $X$ and $Y$ be interval modules with $n$ generators and relations, then $\Hom(X,Y)$ can be computed in $\Oc(n)$ time for $d=2$ parameters an in $\Oc(n^2)$ time for $d >2$.
\end{proposition}

\section{Presentations and Projectivity}\label{sec:projectivity}

All persistence modules are from now on assumed to be finitely presented. 

\subsection{Restriction}

We have considered presentations where we wanted to forget a certain part of the module, say at a degree $\alpha$. We will first construct a crude way of restricting the parameters this way. This will be helpful later but especially we can use it to justify that we restricted ourselves to considering only discrete persistence modules.

\begin{definition}
For an induced subposet $V \subset \Z^d$ (by which we mean a \emph{full} subcategory) we denote the corresponding functor by $\iota_V \colon V \rightarrow \Z^d$ and the restriction functor of persistence modules to $V$-modules as 
  \[\iota_V^* \colon \fun \left( \Z^d, \vect_\K \right) \longrightarrow \fun \left( V, \vect_\K \right) \]
 For any property of a persistence module we say that it holds ``in $V$'', ``at $V$'', or ``above $\alpha$'' for $V = \langle \alpha \rangle$ if it holds after applying $\iota_V^*$.
\end{definition}

\subparagraph{Extending Modules} 

  Let $\iota \colon \Cs \to \Ds$ be any monotonic map of posets viewed as a functor. Since $\vect_\K$ has all small colimits the pullback $\iota^*$ has a left adjoint $\iota_!$ given by left Kan-extension.  The composition of restriction and extension $\left(\iota_{V}\right)_!\iota_{V}^*$ is a standard construction (\cite{lesnick2023multiparameter, fersztand}) and is also called pixelisation or discretisation in the literature if $V$ is a discrete subset of $\R^d$ (\cite{botnan2020, CCGGO}).

To explain its effect on a general module we compute the functor on free modules.

\begin{lemma}\label{planing}
 The functor $\left(\iota_{V}\right)_!\iota_{V}^*$ is right exact. 
 Let $\beta \in \Z^d$.
 If either the poset $\up{\beta} \cap V$ has a minimal element or it contains all joins in $\Z^d$ of its elements, then 
 \[ \left(\iota_{V}\right)_!\iota_{V}^* \left( A\left(-\beta \right) \right) \simeq \K_{ \up{ \up{\beta} \cap V} }
 \]
 In particular, if $\beta \in V$, then
\[ \left(\iota_{V}\right)_!\iota_{V}^* \left( A\left(-\beta \right) \right) \simeq A\left(-\beta \right).
 \]
\end{lemma}

\begin{proof}
 The functor $\iota_V^*$ is exact, because exactness can be checked pointwise in functor categories and $(\iota_V)_!$ is right-exact as a left adjoint. 
 
  Let $\gamma \in \Z^d \setminus \up{ \up{\beta} \cap V}$. Equivalently, for each $\alpha \in V$ such that $\alpha \leq \gamma$ we have $\alpha \not \geq \beta$. Then the slice category $ \iota_{V} \downarrow {\gamma}$ consisting of pairs $\alpha \to \gamma$ with $\alpha \leq \gamma$ and $\alpha \in V$ consists only of pairs with $\alpha \not \geq \beta$ and by definition of Kan-extensions we compute
\[ \left( \left(\iota_{V}\right)_!\iota_{V}^*  A\left(-\beta \right) \right)_{\gamma} = \colim\limits_{\alpha \to \gamma \ \in \ \left( \iota_{V} \downarrow {\gamma} \right)} A\left(-\beta \right)_\alpha  \simeq \colim\limits_{\alpha \to \gamma \ \in \ \left( \iota_{V} \downarrow {\gamma} \right)} 0 \simeq 0. \]

Notice that for every triple $\alpha \to \delta \to \gamma$ in $\Z^n$, if $\delta \not\geq \beta$, then $\alpha \not\geq \beta$. This implies that the diagram $A(-\beta)_{(-)} \colon \iota_{V} \downarrow {\gamma} \to \vect_\K$ contains only arrows of the form $\K \oto{\Id} \K$, $0 \to \K$, and $0 \to 0$, but never $\K \to 0$.
Consider the subcategory $\iota_{\up{\beta} \cap V} \downarrow {\gamma} \into \iota_{V} \downarrow {\gamma}$ of all objects $\alpha \to \gamma$ where $A(-\beta)_{(-)}$ evaluates to $\K$. Any cone over the diagram $A(-\beta)_{(-)} \colon \iota_{\up{\beta} \cap V} \downarrow {\gamma} \to \vect_\K$ can by the above observation be trivially extended to one over $\iota_{V} \downarrow {\gamma}$, so that there is an equivalence
\[ \colim\limits_{\alpha \to \gamma \ \in \ \left( \iota_{V} \downarrow {\gamma} \right)} A\left(-\beta \right)_\alpha 
\simeq
\colim\limits_{\alpha \to \gamma \ \in \ \left( \iota_{V \cap \up{\beta}} \downarrow {\gamma} \right)} A\left(-\beta \right)_\alpha 
\simeq
\colim\limits_{\alpha \to \gamma \ \in \ \left( \iota_{V \cap \up{\beta}} \downarrow {\gamma} \right)} \const \K. \]

If $V \cap \up{\beta}$ contains all joins of its elements in $\Z^d$, then there is a cofinal element $\omega \in V \cap \up{\beta} \cap \down{\gamma}$ and so the above evaluates to $\K$. If instead $V \cap \up{\beta}$ contains a minimal element, we can use the following formula to compute this colimit.

\[\colim\limits_{\alpha \to \gamma \ \in \ \left( \iota_{V \cap \up{\beta}} \downarrow {\gamma} \right)} \const \K \ \simeq \ \bigslant{\bigoplus\limits_{ \iota_{V \cap \up{\beta}} \downarrow {\gamma} } \K}{\sim},\]
where $\sim$ identifies, for each morphism $\delta \to \epsilon \in \iota_{V \cap \up{\beta}}\downarrow {\gamma}$, every pair of elements $x \in A(-\beta)_{\delta}$ and $\Id(x) \in A(-\beta)_{\epsilon}$. 
\end{proof}

\begin{remark}\label{extension_remark}
For each $F \colon V \to \vect_\K$, the unit of the adjunction induces an isomorphism $u \colon F \iso \iota_V^*(\iota_V)_!F$ via the same computation as in the preceding proposition.
\end{remark}

\begin{definition}
  If $V$ is a sub poset of $\Z^d$ we define the category $\grA_V$ to be the full subcategory induced by persistence modules with all generators and relations in $V$.
\end{definition}

\begin{proposition}\label{finitisation}
The adjunction $(\iota_{V})_! \dashv \iota_{V}^*$ restricts to an equivalence of categories
\[ \iota_{V}^* \colon \grA_V \leftrightarrows \fun \left( V, \vect_\K \right) \colon (\iota_{V})_! . \]
\end{proposition}

\begin{proof}
By \autoref{extension_remark} we have $\iota_{V}^* \circ (\iota_{V})_! = \Id$. For the other direction consider a persistence module with all generators and relations in $V$ and any presentation $A^R \oto{f} A^G \to X$.  \autoref{planing} tells that the sequence
 $\left(\iota_{V}\right)_!\iota_{V}^* A^R \oto{\left(\iota_{V}\right)_!\iota_{V}^* f} \left(\iota_{V}\right)_!\iota_{V}^* A^G \to \left(\iota_{V}\right)_!\iota_{V}^* X$ is a presentation of $\left(\iota_{V}\right)_!\iota_{V}^* X$ as well as that since all generators and relations are in $V$ we have $\left(\iota_{V}\right)_!\iota_{V}^* f = f$.
 \end{proof}

\begin{remark}\label{fg_finite}
    For every finitely generated module we can find $\alpha$ and $\beta$ such that all of its generators and relations are in the interval $[\alpha, \beta]$ and we can consider the module as living in $\grA_{[\alpha, \beta]} \simeq \fun\left( [\alpha, \beta], \vect_K \right) \simeq \mathscr{P}([\alpha, \beta]) \lmod$ using \autoref{finitisation}. The path-algebra $\mathscr{P}([\alpha, \beta])$ is finite dimensional with the following nice consequences.
\end{remark}

\begin{proposition}\label{krull-remak-schmidt}
Let $X$ be a f.g. persistence module. Then $X$ is indecomposable if and only if its endomorphism ring $\End(X)$ is local. Its maximal ideal consists of the nilpotent endomorphisms. In particular Azumaya's theorem applies to fg persistence modules.
\end{proposition}

\begin{proof}
For example in \cite{Assem} 4.8 and 4.10. 
\end{proof}

\begin{remark}
 In \cite{CBBB} Theorem 1.1 the authors show that a Krull-Schmidt Theorem even holds for all pointwise finite-dimensional functors $\Cs \to \vect$ for $\Cs$ a small category. This is a more general result, but we do not need it here.
\end{remark}

\subparagraph{Discretisation of Continuous Persistence Modules}\label{discretisation}

In practice persistence modules arise as functors not from $\Z^d$ but from $\R_+^d$ to $ \vect_\K$ which are still finitely presented. We can choose a grid $g\colon\N^d \to \R^d$ on the generators and relations and by an argument analogous to \autoref{finitisation} do not lose any algebraic information (\cite[Definition 2.1]{botnan2020}, \cite[Theorem 11.22]{lesnick2023multiparameter}).

\subsection{Projectivity over a Sub Poset}

When ignoring the last columns in our definition of $\alpha$-decomposition we have constructed a presentation for a module which did not have any more relations at $\alpha$.
Notice that we could not have done this with relations at a degree which is not maximal since then the module being presented would not be unique any more but depend on the concrete presentation chosen. In fact the module being presented after omission of relations is unique also if we chose a whole set of degrees, as long as it is maximal, i.e. there are no other relations in the upper set generated by their degrees. \\
For the rest of this subsection let $U \subset \Z^d$ be an upper set and $V \subset \Z^d$ be its complement.

\begin{definition}
A persistence module $P$ is called \emph{$U$-projective} if it has no relations in $U$, i.e. its first graded Betti number satisfies $b_{1, \beta}(X) = 0$ for each $\beta \in U$. \\
We denote by $\Pc_U$ the class of $U$-projective modules.
\end{definition}

The name is justified because these modules are then projective/free when restricted to $U$. We can make this precise:

\begin{proposition}\label{alpha_projective}
 A persistence module $P$ is $U$-projective if and only if every diagram
\[\begin{tikzcd}
	& X \\
	P & Y
	\arrow["p", two heads, from=1-2, to=2-2]
	\arrow["f", from=2-1, to=2-2]
	\arrow["\xi", curve={height=-6pt}, dashed, from=2-1, to=1-2]
\end{tikzcd}\]
where $p$ is an isomorphism in $V$, has a lift $\xi$. In particular $\Z^d$-projectivity is just projectivity. We say $\alpha$-projective instead if $U = \up{\alpha}$.
\end{proposition}

\begin{proof}
 ``$\Leftarrow$'': $P$ has a minimal presentation without relations in $U$. It follows that every generator in $U$ defines a free summand of $X$ which can always be lifted, so wlog. assume that $P$ also has no generators in $U$.
 Apply $\iota_V^*$ to the diagram above to get a lift, because $p$ is an isomorphism in $V$:
\[\begin{tikzcd}[ampersand replacement=\&]
	\& {\iota^*_VX} \\
	{\iota^*_V P} \& {\iota^*_VY}
	\arrow["{\iota^*_V p }", from=1-2, to=2-2]
	\arrow["{(\iota^*_V p)^{-1}\circ \tilde f}", curve={height=-6pt}, dashed, from=2-1, to=1-2]
	\arrow["{\iota^*_V f}", from=2-1, to=2-2]
\end{tikzcd}\]
Now use the $(\iota_V)_! \dashv \iota^*_V $ and that $P \simeq (\iota_V)_! \iota^*_V P$ by \autoref{finitisation}. \\
 ``$\Rightarrow$'': Let $P$ be $U$-projective. Consider a minimal presentation $P_1 \to P_0 \to P$ and chose any decomposition of the relations $P_1 \iso P_1^{U} \oplus P_1^V$, splitting off all relations whose degree is in $U$. Notice that the image of the inclusion $P_1^{V} \into P_1$ is independent of the choice of decomposition. Therefore this induces a presentation $P_1^{V} \to P_0 \to P^{U}$ and a surjection $p \colon P^{U} \to P$ which is an isomorphism in $V$ and thus splits since $P$ is $U$-projective. The kernel of $p$ lies in $\m P^{U}$ since the presentation was minimal so by Nakayama's lemma $p$ induces an isomorphism. 
\end{proof}

\begin{definition}\label{U_cover}
 Let $X$ be a persistence module. A map $p \colon P \to X$ is a \emph{$U$-cover} (or $\alpha$-cover for $U=(\alpha)$) if $P$ is $U$-projective and $p$ is both an essential surjection and a $V$-isomorphism.
\end{definition}

\begin{proposition}\label{alpha_cover}
 The construction in the second half of the proof of \autoref{alpha_projective} produces a $U$-cover. If $X$ is a persistence module without generators in $U$, then the counit  $c \colon (\iota_V)_! \iota_V^*  X \to X$ is even a functorial $U$-cover.
\end{proposition}

\begin{proof}
The construction produces a $V$-isomorphism and surjection $p: X^{U} \to X$ from a $U$-projective module. To see that it is an essential surjection, recall that by \autoref{prop:essential} this is equivalent to $ p \otimes A/ \m$ being an isomorphism, but in the proof we already saw that $\ker p \subset \m X^{U}$. \\
In the case where $X$ has no generators in $U$, the counit is isomorphic to the first construction by \autoref{planing}.
\end{proof}

\begin{remark}\label{cover_minimality}
The standard definition of a cover (eg. \cite[5.1.1]{EnochsJenda}), requires another notion of minimality: Every endomorphism of the map $p$ is actually an automorphism. This property is recovered by considering the image of this endomorphism, which must be all of $P$ as $p$ is an essential surjection. So the endomorphism is a surjection, but then it must be a bijection. 

Together with \autoref{alpha_cover} this means that $\Pc_U$ is a \emph{covering} class.
\end{remark}

In the preceding terminology an $\alpha$-decomposed presentation \autoref{def:alpha_decomposed} is the presentation of an indecomposable decomposition of the $\alpha$-cover of a module together with its relations at $\alpha$.

\subsection{Presentations of Short Exact Sequences}

After having introduced the covering class $\mathcal{P}_U$, every module has a $\mathcal{P}_U$-resolution and in particular a $\mathcal{P}_U$-presentation, which we will later decompose separately. To see what this does to the normal presentation of a module, we need to understand the relationship between presentations (or more generally resolutions) and short exact sequences. 

 \begin{proposition}\label{pres_cd}
 Let $0 \to X \oto{f} Y \oto{g} Z \to 0$ be a short exact sequence of persistence modules. 
 
 If $X,Y$ are presented by graded matrices $O,M$, then for every lift $N$ of $f$, the following diagram commutes and supplies a presentation for $Z$.
\[\begin{tikzcd}[ampersand replacement=\&]
	\& {A^{R_1}} \& {A^{R_2}} \& {A^{R_2} \oplus A^{G_1}} \\
	\& {A^{G_1}} \& {A^{G_2}} \& {A^{G_2}} \\
	0 \& X \& Y \& Z \& 0
	\arrow["{\bar N}", from=1-2, to=1-3]
	\arrow["O", from=1-2, to=2-2]
	\arrow["{\incl_1}", from=1-3, to=1-4]
	\arrow["M", from=1-3, to=2-3]
	\arrow["{[M \ N]}", from=1-4, to=2-4]
	\arrow["N", from=2-2, to=2-3]
	\arrow[from=2-2, to=3-2]
	\arrow["\Id", from=2-3, to=2-4]
	\arrow["p_Y", from=2-3, to=3-3]
	\arrow[from=2-4, to=3-4]
	\arrow[from=3-1, to=3-2]
	\arrow["f", from=3-2, to=3-3]
	\arrow["g", from=3-3, to=3-4]
	\arrow[from=3-4, to=3-5]
\end{tikzcd}\]
Vice versa, any graded matrix $[M \ N]$ induces a diagram like this.

If $M$ is a minimal presentation and $O$ satisfies at least $O \otimes A/ \m = 0$ then the presentation $[M \ N]$ is minimal if and only if 
\[ N \otimes A/ \m = 0 \quad \text{ and } \quad \bar N \otimes A/ \m = 0\]

\end{proposition}

\begin{proof}
    To see that $[M \ N]$ is a presentation just chase the diagram. Given a presentation $[M \ N]$, define $Y$ as $\coker M$ and $X$ as $\Ima ( p_Y \circ N)$. \\
    For the second part, apply $-\otimes A/ \m$ to the whole diagram. Then computing the homology produces the long exact sequence
\[\begin{tikzcd}[ampersand replacement=\&]
	{A^{R_1} \otimes A / \m} \& {A^{R_2} \otimes A/\m} \\
	{\Tor_1(X, \, A/ \m)} \& {\Tor_1(Y, \, A/ \m)} \& {\Tor_1(Z, \, A/ \m)} \\
	\\
	{X \otimes A/ \m} \& {Y \otimes A/ \m} \& {Z \otimes A/ \m} \& 0 \\
	{A^{G_1} \otimes A/ \m} \& {A^{G_2} \otimes A/ \m}
	\arrow["{\bar N \otimes A/\m}", from=1-1, to=1-2]
	\arrow[two heads, from=1-1, to=2-1]
	\arrow["\wr"', from=1-2, to=2-2]
	\arrow["i", from=2-1, to=2-2]
	\arrow[from=2-2, to=2-3]
	\arrow[out=-30, in=150, looseness=1, from=2-3, to=4-1]
	\arrow["p", from=4-1, to=4-2]
	\arrow["h", from=4-2, to=4-3]
	\arrow[from=4-3, to=4-4]
	\arrow["\wr"', from=5-1, to=4-1]
	\arrow["{N \otimes A/ \m}", from=5-1, to=5-2]
	\arrow["\wr"', from=5-2, to=4-2]
\end{tikzcd}\]

Our assumptions on $M$ and $O$ give us the three isomorphisms. The surjection to the first term comes also from the fact that $O \otimes A / \m=0$ but here the homology at the term $A^{R_1}$ can be an actual quotient.\\
We use \autoref{prop:minimal_pres} to conclude that the presentation $[M \ N ]$ is minimal if and only if 
\[Z \otimes A / \m \simeq A^{G_2} \otimes A/ \m \quad \text{ and } \quad \Tor_1(Z, A / \m) \simeq \left( A^{R_2} \oplus A^{G_1} \right) \otimes A / \m .\]
By exactness of the long exact sequence this in turn is equivalent to $p = 0$ and $i = 0$ which in turn is equivalent to $N \otimes A/ \m = 0$  $\bar N \otimes A/ \m$ by the identifications.
\end{proof}

 \begin{proposition}\label{pres_ext}
Let $X \oto{f} Y \oto{g} Z \to 0$ be an exact sequence of persistence modules. If $X,Z$ are presented by graded matrices $M,N$, then there are maps $\pi, \ C$ making the following diagram of short exact sequences commute 


\[\begin{tikzcd}[ampersand replacement=\&]
	0 \& {A^{R_1}} \& {A^{R_1} \oplus A^{R_2}} \& {A^{R_2}} \& 0 \\
	0 \& {A^{G_1}} \& {A^{G_1} \oplus A^{G_1}} \& {A^{G_2}} \& 0 \\
	\& X \& Y \& Z \& 0
	\arrow[from=1-1, to=1-2]
	\arrow[from=1-2, to=1-3]
	\arrow["M", from=1-2, to=2-2]
	\arrow[no head, from=1-3, to=1-4]
	\arrow["{\begin{bmatrix} M & C \\ 0 & N \end{bmatrix}}", from=1-3, to=2-3]
	\arrow[from=1-4, to=1-5]
	\arrow["N", from=1-4, to=2-4]
	\arrow[from=2-1, to=2-2]
	\arrow[from=2-2, to=2-3]
	\arrow["{\pi_X}", from=2-2, to=3-2]
	\arrow[no head, from=2-3, to=2-4]
	\arrow["\pi", from=2-3, to=3-3]
	\arrow[from=2-4, to=2-5]
	\arrow["{\pi_Z}", from=2-4, to=3-4]
	\arrow["f", from=3-2, to=3-3]
	\arrow["g", from=3-3, to=3-4]
	\arrow[from=3-4, to=3-5]
\end{tikzcd}\]

Additionally, $f$ is injective if and only if
$C$ satisfies the cocycle condition

\begin{equation}\tag{$\dag$}\label{cocycle}
\forall t \in A^{R_2} \colon \ Nt=0 \Rightarrow \exists s \in A^{R_1} \colon \ Ct=Ms.
\end{equation}

\end{proposition}

\begin{remark}
 The presentation of $Z$ can be extended with a map $\ker N \oto{i} A^{R_2}$ to compute $\Ext^1(Z,-)$. Then $\pi_X \circ C$ is a cocycle in $\Hom(A^{R_2}, X)$ in the usual sense precisely if $\pi_X \circ C \circ i = 0 \Leftrightarrow \text{(\ref{cocycle})}$  and the class $[C]$ corresponds to this short exact sequence. We have only rephrased the condition to emphasise that in our case it can be checked by solving linear systems over $\K$. 
\end{remark}

\begin{proof}
This is a version of the horseshoe lemma (\cite{Weibel} 2.2.8) for presentations. Still, let us briefly explain how to construct $\pi$ and $C$. \\
Since $g$ is surjective and $A^{G_2}$ projective we can factor $\pi_Z$ through $g$ to find a map $\pi_2 \colon A^{G_2} \to Y$. Then set $\pi \coloneqq f \circ \pi_X + \pi_2$. Notice that $g \circ \pi \circ \incl_2 \circ N = g \circ \pi_2 \circ N = \pi_Z \circ N = 0$, so $\pi \circ \incl_2 \circ N$ factors through $X$ and this in turn through $\pi_X$, again because $A^{R_2}$ is projective and $\pi_X$ surjective, and we get a map $C' \colon A^{R_2} \to A^{G_1}$. Set $C \coloneqq - C'$. \\
For the second statement use the snake lemma to get an exact sequence
$\ker N \oto{\pi_X \circ C \circ i} X \oto{f} Y \to \dots$.
\end{proof}

\begin{lemma}\label{min_pres_criterion}
In the setting of \autoref{pres_ext}, let f be injective and $M, N$ minimal presentations, then the following are equivalent: \begin{enumerate}
\item The presentation of $Y$ is minimal.
\item $C \otimes A/\m = 0$.
\item $f \otimes A/\m$ is injective.
\end{enumerate}
\end{lemma}

\begin{proof}
Apply $-\otimes A/\m$ to the whole diagram. \\
$(1) \Leftrightarrow(2) \colon$
     If the presentation of $Y$ is minimal then by \autoref{prop:minimal_pres} $\begin{bmatrix} M & C \\ 0 & N \end{bmatrix} \otimes A/\m = 0 \ \Leftrightarrow \ M \otimes A/\m + C \otimes A/\m = 0 \Leftrightarrow M \otimes A/\m = 0$ using minimality of $M$ and $N$. \\
     For the other direction, we need to see that there is no split injection $\iota \colon A(-\alpha) \into A^{R_1} \oplus A^{R_2}$ such that $\begin{bmatrix} M & C \\ 0 & N \end{bmatrix} \circ \iota = 0$. But this is true regardless of $C$: If $N \circ \iota = 0$, then since $N$ is minimal $\Ima \iota \subset A^{R_1}$, but this contradicts the minimality of $M$. \\
     
$(2) \Leftrightarrow(3) \colon$
     Consider the homology of the resulting diagram and observe that $C \otimes A/\m$ induces the connecting homomorphism in the long exact sequence
     \[ \dots \to  \Tor_1(Z, A/\m) \oto{d_0} X \otimes A/m \oto{f \otimes A/ \m} Y \otimes A/\m \to \dots \quad .\]
     By a diagram chase $C \otimes A/\m = 0$ if and only if $d_0 = 0$.
\end{proof}

\section{Decomposing Non-uniquely Graded Presentations}
\label{sec:brute_force_decomposer}

We have seen how the repeated use of the \texttt{blockreduce} subroutine of Dey and Xin decomposes $\alpha$-decomposed presentations with a single column of degree $\alpha$. \\
We set out to find a subroutine that also decomposes $\alpha$-decomposed presentations with arbitrary number of relations. If we have it, then the following framework will decompose any presentation.

\begin{algorithm}[\textsc{AIDA} Main Routine]\label{main_routine} \ \newline \texttt{Input}: Any presentation $M \in \K^{G \times R}$ 

Partition the columns of $M$ into $\bar n$ maximal sets of equal degree. For each $i \in [\bar n]\colon$
\begin{enumerate}
 \item By induction, $\left(M_{<i} \, M_{i} \right)$ is $\alpha$-decomposed and $M_{<i}$ has a maximal block-structure $\B$.
 \item For each $b \in \B$ \begin{enumerate}
     \item Perform a \hyperref[colsweeps]{column-sweep}
     \item use the optimised Hom-variant of \hyperref[subparagraph:hom_variant]{\texttt{BlockReduce}} and store the $\Hom$-vector spaces.
 \end{enumerate}
 \item Gather the blocks $b \in \B$ for which $N_b$ is nonzero in a set $\B_{pie}$.
 \item\label{passing} Pass the $\Hom$-vector spaces and the presentation $(M_{\B_{pie}} \ N_{\B_{pie}})$ to a subroutine.
\end{enumerate}
\end{algorithm}

\subsection{Obstructions to extending the Generalized Persistence Algorithm}\label{sec:obstructions}

One might first hope that one could process the columns of $N$ in some order and just use \hyperref[subparagraph:hom_variant]{\texttt{BlockReduce}}. Unfortunately this does not work: A first counterexamples was given in  \cite[Example 5.14]{ABENY}. We construct a 2-parameter module over $\F_2$ which is smallest for presentations over $\F_2$ which cannot be decomposed with this method.

\begin{example}\label{example:no_permutation}
Consider the presentation $M$ on the left where all rows have incomparable degree. The invertible matrix in the middle is a decomposition of $M$.
\[
\begin{array}{| c | c  c | c  c |}
    \hline  
     M & \cellcolor{pink} \alpha  & \cellcolor{pink} \alpha  & \cellcolor{pink} \alpha  & \cellcolor{pink} \alpha  \\
    \hline
    \blcell \wtext{(0,5)} & \ocell 0 & \ocell 1 &  \tqcell 0 &  \tqcell 1 \\
    \blcell \wtext{(1,4)} & \ocell 1 & \ocell 0 &  \tqcell 1 &  \tqcell 0 \\
    \blcell \wtext{(2,3)} & \ocell 0 & \ocell 1 &  \tqcell 1 &  \tqcell 1 \\
    \blcell \wtext{(3,2)} & \apcell 0 & \apcell 1 & \lblcell 1 & \lblcell 0 \\ 
    \blcell \wtext{(4,1)} & \apcell 1 & \apcell 1 & \lblcell 0 & \lblcell 1 \\ 
    \blcell \wtext{(5,0)} & \apcell 1 & \apcell 0 & \lblcell 1 & \lblcell 1 \\
    \hline
\end{array}
\quad \cdot \quad 
\begin{array}{| c  c | c  c |}
\hline  
  1 &  1 &  1 &  0 \\
  0 &  1 &  0 &  1 \\
  0 &  1 &  1 &  0 \\
  1 &  0 &  0 &  1 \\
 \hline
\end{array}
\quad =
\quad 
\begin{array}{| c | c  c | c  c |}
    \hline  
     M & \cellcolor{pink} \alpha & \cellcolor{pink} \alpha  & \cellcolor{pink} \alpha  & \cellcolor{pink} \alpha  \\
    \hline
    \blcell \wtext{(0,5)} & \ocell 1 & \ocell 1 &  0 &  0 \\
    \blcell \wtext{(1,4)} & \ocell 1 & \ocell 0 &  0 &  0 \\
    \blcell \wtext{(2,3)} & \ocell 0 & \ocell 1 &  0 &  0 \\
    \blcell \wtext{(3,2)} &  0 &  0 & \lblcell 1 & \lblcell 1 \\ 
    \blcell \wtext{(4,1)} &  0 &  0 & \lblcell 1 & \lblcell 0 \\ 
    \blcell \wtext{(5,0)} &  0 &  0 & \lblcell 0 & \lblcell 1 \\
    \hline
\end{array}
\]
This decomposition can also be reached after arbitrary change of basis of the two sides of the second matrix, which would yield an equivalent decomposition. But no such change of basis and permutation of the columns makes the matrix in the middle triangular. Notice that there is also no other decomposition of $M$ up to equivalence. \\
Therefore, there is no order on these the relations such that column operations only following this order can decompose the presentation.
\end{example}

If on the other hand, we permit column operations which do not follow some order but still want to proceed column-by-column, we run into another problem. Say we want to reduce a new column $N$ but there are more columns of the same degree. If we want to delete the sub-columns $N_b$ then the newly added column operations make the maximal set of blocks $b$, for which $N_b$ can be deleted non-unique (also up to permutation).

\begin{example}
Consider the following graded matrix, where we try to reduce the first  column.
\[
\begin{array}{| c | c |  c | }
        \hline 
         M  &  \cellcolor{pink} (2,2) &   (2,2)  \\
        \hline
        \lbcell (0,1)    &  0   &  1       \\
        \lrcell (1,1)    &  1   &  1         \\
        \lgcell (2,0)    &  1   &  0                \\
        \hline
\end{array} \quad \rightarrow  \quad \begin{array}{| c | c |  c | }
        \hline 
         M  &  \lrcell (2,2) &  \cellcolor{pink}  (2,2)  \\
        \hline
        \lbcell (0,1)    &  \rtext{1}   &  0       \\
        \lrcell (1,1)    & \lrcell 1   &  1         \\
        \lrcell (2,0)    & \lrcell 1   &  0                \\
        \hline
\end{array} \]

There are three blocks at the beginning of the decomposition and the first entry, belonging to the blue block, has already been deleted. No combination of admissible row and column operations can delete more entries, and as we can see above, if we continue from here, then in the next column, the admissible row operation from second to first row would create another 1 in the first column which we cannot avoid.

\[
\begin{array}{| c | c |  c | }
        \hline 
         M  &  \cellcolor{pink} (2,2) &   (2,2)  \\
        \hline
        \lbcell (0,1)    &  1   &  1       \\
        \lrcell (1,1)    &  0   &  1         \\
        \lgcell (2,0)    &  1   &  0                \\
        \hline
\end{array} \quad \rightarrow  \quad \begin{array}{| c | c |  c | }
        \hline 
         M  &  \lbcell (2,2) &  \cellcolor{pink}  (2,2)  \\
        \hline
        \lbcell (0,1)    &  \lbcell 1   &  \gtext{0}       \\
        \lrcell (1,1)    &  0   &  1         \\
        \lbcell (2,0)    & \lbcell 1   &  0                \\
        \hline
\end{array} \]

But adding the second to the first column first would make it so that the entry to the red block is deleted instead and from here we would find a decomposition. But a priori we did not know that, as both choices had a minimal set of blocks for which the entry in the first column was non-zero.
\end{example}

We conclude from these two examples that we will have to consider all new columns of the same degree at once. 

\subsection{A Brute-Force Solution}\label{sec:brute_force}

One immediate difference to the distinctly graded case is that a decomposition might introduce several new blocks. Instead of finding all of them at once, we will try to split the presentation into two summands and do this recursively. We invoke \autoref{lemma:principle} again to find invertible matrices, which leave $M_\B$ unchanged and such that
\[ Q \left[ M_\B \ N \right] \begin{bmatrix} 
        P & U \\
        0 & T
       \end{bmatrix} = 
       \left[ Q M_\B P \ \ Q \left( M_\B U + NT \right) \, \right] =
       \left[ \begin{array}{ c c c c} 
        M_{\B_1} & 0 & N_1 & 0 \\
        0 & M_{\B_2} & 0 & N_2 \\
       \end{array} \right]. 
       \]
It follows that there is an invertible graded matrix $T \in \K^{k \cdot \alpha \times k \cdot \alpha}$ and a partition $[k] = \{1, \dots, j \} \cup \{ j+1, \dots, k\}$ of the columns of $N$, which allow us to decompose the whole presentation by zeroeing out submatrices with \hyperref[subparagraph:hom_variant]{\texttt{BlockReduce}}: After multiplying $N$ with $T$ from the right, for every block $b \in \B$ we can delete either $N_{b, [j]}$ or $N_{b, {[j]^\complement}}$. Here we see that we would have to solve \hyperref[hom_computation]{($\ast$)} again, but since we stored its solution this is not necessary and \hyperref[hom_computation]{($\ast \ast$)} is efficiently solved. Since we work over a finite field, this immediately implies a brute-force algorithm:

\subparagraph{Sketch: Exhaustive $\alpha$-decomposition}\label{subparagraph:exhaustive}  \
For each invertible matrix $T$ over $\F_q$ and $j \in [k]$ set $N_1 =  \left(NT\right)_{\{1, \dots, j \}}$, $\ N_2 =\left(NT\right)_{\{j+1, \dots, k \}}$ and decompose each separately with \hyperref[subparagraph:hom_variant]{\texttt{BlockReduce}}.

\vspace{0.5em}
This is computationally unfeasible even for smaller $k$: 
The maximal number of iterations of this loop for $q=2$ and $k \in \{2, \dots, 6\}$ are $6, \, 342, \, 60 \, 822, \, 40 \,  058 \, 262, \, 100 \, 833 \, 607 \, 062 $ and more generally $| \GL_k(\F_q)| \in q^{k^2 + \Oc\left(k \right)}$.
Fortunately, we can be much smarter about the matrices and partitions we use: Assume that $T$ is as above so that $N T$ decomposes into two parts, supported on $[j]$ and ${[j]^\complement}$, after using \hyperref[subparagraph:hom_variant]{\texttt{BlockReduce}} \hyperref[hom_computation]{($\ast \ast$)}. 
The two parts also induce a partition of $T = \left[T_1 \ T_2\right]$. If $T_1', \, T_2'$ are matrices with the same column-spaces as $T_1, \, T_2$ respectively then $N\left[T_1' \ T_2' \right]$ could be brought to the same decomposition with \hyperref[subparagraph:hom_variant]{\texttt{BlockReduce}}. That is, finding a decomposition with this strategy depends only on the \emph{decomposition} of $\F_q^k$ as a vector space which $T = \left[ T_1 \ T_2 \right]$ represents. Additionally, once we have deleted as many parts of $NT_1$ as possible, we can update the block structure and do the same again for $NT_2$, but now also being able to use column-operations from the columns of $NT_1$ to those of $NT_2$ without breaking linearity of the system. That is, we also do not need to consider a decomposition $\left[ T_1 \ T_2 \right]$, if a decomposition has been considered already that is of the form $\left[ T_1 \ T_2' + T_1 S \right]$, for any matrix $S$ and any $T_2'$ with the same column-space as $T_2$. But this is the case for \emph{every} other invertible matrix of the form $\left[ T_1 \ T_2' \right]$, so that we only need to iterate over all proper subspaces $\text{span}\left( T_1 \right)$ of $\F^k_q$, i.e. the Grassmanian $\Gr_l(\F_q^k)$ for every $l < k$. These matrices can be generated fast by leveraging the Schubert cell decomposition \cite{Schubert}:

\subparagraph{Algorithm \texttt{Generate} $\Dec_{q}(k, \, l)$:}\label{sub:generate_decompositions}
For each set $P \in \binom{[k]}{l}$ construct a  $k \times l$ matrix $T_1$ in column-echelon form with pivots given by $P$. Then construct all subspaces in the corresponding cell by changing the entries that are irrelevant for the echelon form and not in the same row as a pivot. For each $T_1$ generated this way, choose any $T_2$ with opposite pivots and store the pairs $\left[ T_1 \ T_2 \right]$ in a set $\Dec_{q}(k, \,  l)$.

\subparagraph{Avoiding Double Counting.}\label{sub:double_count}
Since exchanging $T_1$ and $T_2$ in $\Dec_{q}(k, \,  l)$ would not make a difference for the decomposition we can actually ignore the sets $\Dec_{q}(k, \,  l)$ for $l > k/2$. In the case where $k$ is even and $l = k/2$ we can avoid double counting even more. Consider the graph on $\Gr_{k/2}(\F_q^k)$ where $U \sim V$ if they form a decomposition (also called the $q$-Kneser graph). By symmetry we only need to find a minimal vertex cover of this graph and use these subspaces for the computation above. This can be done by considering in \hyperref[sub:generate_decompositions]{\texttt{Generate} $\Dec_{q}(k, \, k/2)$} only those pivot sets $P \in \binom{[k]}{k/2}$ which contain the first position. Let $\Dec_q(k) \coloneqq \bigcup\limits_{l \leq k/2} \Dec_q(k,l)$, then we can iterate over $\Dec_q(k)$ instead of $[k] \times \GL_k\left(\F_q\right)$ for the exhaustive decomposition. The first terms of $ | \Dec_2(k) |$ are $1, \, 2, \,  7, \, 43, \, 186, \, 1 965, \, 14  605, 297  181, \, \dots$ and the asymptotic growth is in $q^{k^2/4 + \Oc\left(k\right)}$.

There is a branch-and-bound strategy to avoid testing decomposition for a subspace again in the recursion. We can pull back decompositions we have already tried to the sub-spaces for which a decomposition has been found.

\begin{algorithm}[Exhaustive $\alpha$-decomposition]\label{exhaustive_full}  \ \\
\texttt{Input:} $\alpha$-decomposed presentation $\left[M_\B \ N \right]$ \\
For each $[T_1 \ T_2] \in {\Dec}_q(k)\colon$ 
\begin{enumerate}
    \item Set $N_1 \coloneqq NT_1$ and $N_2 \coloneqq  NT_2$.
    \item For each $b \in \B \colon $ try to clear $(N_1)_b$ with \hyperref[subparagraph:hom_variant]{\texttt{BlockReduce}} \hyperref[hom_computation]{($\ast \ast$)}.
    \item Gather the blocks where $(N_1)_b$ is non-zero in $\B_1$.
    \item Try to clear $(N_2)_{\B_1}$ with \hyperref[subparagraph:hom_variant]{\texttt{BlockReduce}} \hyperref[hom_computation]{($\ast \ast$)}, define $\B_2 \coloneqq \B_1^{\complement}.$
    \item If successful, recursively pass $[M_{\B_1} \ N_{\B_1}]$ and $[M_{\B_2} \ N_{\B_2}]$ to Exhaustive $\alpha$-decomp.\\
    Return $\begin{bmatrix} M_{\B_1} & 0 & N_{\B_1} & 0 \\
    0 & M_{\B_2} & 0 & N_{\B_2}
    \end{bmatrix}$.
\end{enumerate}
Return $\left[M_\B \ N \right]$. 

\end{algorithm}

\begin{theorem}
 Let $M$ be a minimal presentation of a persistence module $X$ over $\F_q$ with $n$ generators and relations and at most $k$ relations of the same degree. Then the exhaustive decomposition algorithm has a worst-case running time in
$\Oc \left( n^{2\omega +1} + n^{\omega+2}q^{k^2/4 + \Oc(k)}  \right)$.
\end{theorem}

\ignore{
\[\Oc \left(\bar n \left(n m\right)^{\omega-1} \left(m^2 + n^2 \right) + \bar n  (mk)^{\omega-1}(m^2+nk) q^{k^2/4 + \Oc(k)}  \right)  . \]
}

\begin{proof}
For at most $n$ times we need to compute {\texttt{BlockReduce}} \hyperref[hom_computation]{($\ast$)} and then solve {\texttt{BlockReduce}} \hyperref[hom_computation]{($\ast \ast $)} at most $q^{k^2/4 + \Oc(k)}$ many times. Using  
\autoref{proposition:run_time_block} gives the term above, because the polynomial factors in $k$ are dominated by the quadratic term.
\end{proof}

\subsection{Effective Decompositions}\label{sec:effective}

The results in this subsection hold for arbitrary modules over finite dimensional algebras. We deduce from \autoref{fg_finite} that they also hold for fg persistence modules.

\begin{proposition}\label{iso_block}
 Let $X_i \in \Cs$ be a finite set of indecomposable objects with local endomorphism rings. \\
 Consider a direct sum $X \coloneqq \bigoplus_{i \in I} X_i^{n_i}$ with $X_i \not\simeq X_j$ for each $i \neq j \in I$. Then let $\varphi \colon X \iso X$ be any automorphism and write it as an $I \times I$ matrix with entries $\varphi_{i,j} \colon X_j^{n_j} \to X_i^{n_i}$. For every $i \in I$ the endomorphism $\varphi_{i,i}$ is an isomorphism and $\varphi_{i,j}$ belongs to $\rad(X_j^{n_j}, X_i^{n_i})$.
\end{proposition}

\begin{proof}
 Recall that for any additive category $\Cs$ and $X, Y \in \Cs_0$ the radical $\rad(Y,Z)$ is defined to be the set of homomorphisms $f \colon Y \to Z$ such that for every $g \colon Z \to Y$ the homomorphism $\Id_{X_i} - gf$ is invertible (\cite{Assem} A.3.3-A.3.6). Since for $i \neq j \in I$ the modules $X_i, X_j$ are both indecomposable and non-isomorphic we have $\rad(X_i,Y_j) \simeq \Hom(X_i,Y_j)$ from where it follows that $\rad(X^{n_i}_i,Y^{n_j}_j) \simeq \Hom(X^{n_i}_i,Y^{n_j}_j)$ and finally that $( \varphi_{i,j})_{i \neq j \in I} \in \rad(X, X) = J \End(X)$. Then the diagonal $(\varphi_{i,i})_{i \in I}$ must be an isomorphism and so every endomorphism $\varphi_{i,i}$ is an isomorphism, too.
\end{proof}

\begin{lemma}\label{iso_perm}
 Let $X$ be an indecomposable object with local endomorphism algebra and $\varphi \colon X^k \iso X^k$ an isomorphism. Then there is a permutation $\sigma \in \Sigma_k$ such that the components $\varphi_{i, \sigma(i)}$ are all isomorphisms.
\end{lemma}

\begin{proof}
 Consider $\varphi$ as a $k \times k$ matrix over the local algebra $\End(X)$. Passing to the residue division algebra $D \coloneqq \End(X)/ \m$ sends precisely the non-invertible endomorphism to 0. Consider the induced matrix $\overline{\varphi} \in D^{k \times k}$ which is still invertible as a map $D^k \to D^k$. Then as a consequence of Hall's marriage theorem there is a permutation $\sigma \in \Sigma_k$ such that $\overline{\varphi}_{i, \sigma(i)}$ is non-zero.
\end{proof}

\begin{corollary}\label{permutation}
In the setting above, let $\varphi \colon \bigoplus\limits_{i \in I} X_i \to \bigoplus\limits_{i \in I} X_i$ be an automorphism between indecomposable decompositions. Then there is a permutation $\sigma \in \Sigma_I$ such that for each $i \in I$ the maps $\varphi_{i, \sigma(i)}$ are all isomorphisms.
\end{corollary}

\begin{proof}
 For each isomorphism class $[X]$ the morphism $\varphi_X$ from \autoref{iso_block} is an isomorphism. Then use \autoref{iso_perm} to get permutations for each such morphism and paste them together.
\end{proof}

 \begin{lemma}\label{decomp_resolution}
Let $\Cs$ be an abelian category and $\Fc$ be any covering class of objects (\cite{EnochsJenda} 5.1.1), so that minimal left $\Fc$-resolutions exist. If $p\colon P_* \to X$ is a left $\Fc$-resolution then any decomposition $\varphi \colon  X' \oplus X''  \iso X$ lifts to a decomposition of the resolution.
 \end{lemma}
 
 \begin{proof}
Let $i \colon X' \into X \colon r$ be the induced injection and its retract. Construct a minimal $\Fc$-resolution $p' \colon P'_* \to X'$ and choose lifts $i_0, r_0$ between $P_0$ and $P_0'$. By the diagram chase 
\[i \circ p' \circ r_0 \circ i_0 =  
p \circ i_0 \circ r_0 \circ i_0 = 
p \circ i_0 = i \circ p'\]
it follows that $p' \circ r_0 \circ i_0 = p'$ because $i$ is mono. Then by minimality, $r_0 \circ i_0 $ must be an isomorphism. Analogously for all higher lifts we find that $i$ lifts to a split injection of resolutions.
 \end{proof}

\subsection{Decomposing Presentation Matrices}\label{sec:decomposing_pres}

In this section we show how to get change-of-basis matrices from partial decompositions of modules which leave the decomposed part intact. 

\begin{lemma}\label{matrices_cd}
Consider an $\alpha$-decomposed presentation $\left[ M_\B \  N\right] \in \K^{G \times (R \cup k \cdot \alpha)}$ of a module $Y$. For any decomposition $\varphi \colon Y \iso Y_1 \oplus Y_2$ we can find invertible graded matrices
 \[Q \in \K^{G \times G}
  \quad \text{ and } \quad \begin{bmatrix} 
        P & U \\
        0 & T
       \end{bmatrix}
\in \K^{(R \cup k\cdot \alpha) \times (R \cup k\cdot \alpha) }\] 
with partitions $U = [U_1 \ U_2]$ and $T= [T_1 \ T_2]$ which induce a block decomposition of
\[ Q \left[M_\B \ N\right] \begin{bmatrix} 
        P & U_1 & U_2 \\
        0 & T_1 & T_2
       \end{bmatrix} = [ M_\B \ N_1 \ N_2 ] \]
where for every $b \in \B$ we have $(N_1)_b = 0$ or $(N_2)_b = 0$. Additionally, for every $b \in \B$, it holds that
 \[Q M_\B = M_\B P^{-1}  \quad (\ast) , \quad Q_{b,b} = \Id_{A^{G_b}}, \quad \text{and} \quad   P^{-1}_{b,b} = \Id_{A^{R_b}} \quad (\ast \ast ). \]
 
\end{lemma}

\begin{proof}

Let 
\[0 \to A^k(- \alpha) \oto{\tilde N} \bigoplus_{b \in \B} X_b \oto{p} Y \to 0 \] 
be the induced $\mathcal{P}_\alpha$-resolution of $Y$. The decomposition $\varphi$ of $Y$ lifts to a decomposition of resolutions by \autoref{decomp_resolution}. \\
The right hand summands of the decomposition  $ \varphi_0 \colon \,  \bigoplus_{b \in \B} X_b \iso X_1 \oplus X_2$ can be decomposed further, and by Krull-Remak-Schmidt \autoref{krull-remak-schmidt} this induces an automorphism
\[ \varphi' \colon \bigoplus_{b \in \B} X_b \oto{\varphi_0} X_1 \oplus X_2 \oto{\psi_1 \oplus \psi_2} {\bigoplus\limits_{b \in \B_1} X_b \oplus \bigoplus\limits_{b \in \B_2} X_b}. \]
Where $\B_1 \sqcup \B_2$ is a partition of $\B$. Moreover by post-composition with the map $\sigma_*$ from \autoref{permutation} we can relabel the summands $X_b$ to ensure that for every $b\in \B$ the diagonal element $(\sigma_* \circ \varphi')_{b,b}$ is an isomorphism as in the following diagram.

\[\begin{tikzcd}[ampersand replacement=\&]
	0 \& {A^k( - \omega)} \& {\bigoplus\limits_{b \in \B} X_b} \& Y \& 0 \\
	0 \& {A^{k_1}( - \omega) \oplus A^{k_1}( - \omega)} \& {\bigoplus\limits_{b \in \B_1} X_b \oplus \bigoplus\limits_{b \in \B_2} X_b} \& {Y_1 \oplus Y_2} \& 0 \\
	\&\& {\bigoplus\limits_{b \in \B} X_b}
	\arrow[from=1-1, to=1-2]
	\arrow["{\widetilde N}", from=1-2, to=1-3]
	\arrow["{\varphi''}"', from=1-2, to=2-2]
	\arrow["p", from=1-3, to=1-4]
	\arrow["{\varphi'}", from=1-3, to=2-3]
	\arrow[from=1-4, to=1-5]
	\arrow["\varphi", from=1-4, to=2-4]
	\arrow[from=2-1, to=2-2]
	\arrow["{\widetilde N_1 \oplus \widetilde N_2}", from=2-2, to=2-3]
	\arrow["{\sigma_* \circ \left( \widetilde N_1 \oplus \widetilde N_2 \right) }"', from=2-2, to=3-3]
	\arrow["{p_1 \oplus p_2}", from=2-3, to=2-4]
	\arrow["{\sigma_*}", from=2-3, to=3-3]
	\arrow[from=2-4, to=2-5]
	\arrow["{ \left(p_1 \oplus p_2\right) \circ \sigma_*^{-1}}"', from=3-3, to=2-4]
\end{tikzcd}\]

We choose lifts $N_1$ and $N_2$ of $\sigma_* \circ \widetilde N_1$ and $\sigma_* \widetilde N_2$ to the minimal cover of $\bigoplus_{b \in \B} X_b$. Notice that this can be chosen in a way such that for every $b \in \sigma(\B_1)$ we have $\left(N_2\right)_b = 0$ and analogously for $\B_2$, since $\sigma_*$ is just a permutation of the blocks. \\
We can now use \autoref{pres_cd} on the ses at the bottom to construct a presentation $\left[ M_\B \ N_1 \ N'_2 \right]$ for $Y_1 \oplus Y_2$ that is block-decomposed and which is also minimal because it has the same size as the one we started with. Lift $\varphi$ to a morphism of minimal presentations, where both matrices making up the morphism are invertible by (\autoref{minimal_iso}).
\[ \left( Q, \ \begin{bmatrix}
                       P & U_1 & U_2 \\
                       0 & T_1 & T_2 
                      \end{bmatrix}^{-1} \right) \colon \ [ M_\B \ N ] \to [ M_\B \ N_1 \ N_2 ] \]

The $0$ entry in the bottom left comes directly from the matrix being graded.
It follows that $Q$ and $P$ must be invertible, too. In particular $(\ast)$ is satisfied. 

Considering again the construction of a presentation from a short exact sequence in \autoref{pres_cd}, we deduce that the cokernel of the pair $(Q, P^{-1})$ is the isomorphism $\sigma_* \circ \varphi'$ which was an isomorphism on its diagonal. We may wlog assume that $\sigma_* \circ \varphi'$ is the identity on its diagonal by multiplying with the inverses of these automorphisms. Therefore we can choose $Q,P$ to have the same property and $(\ast \ast)$ holds.

\end{proof}

\begin{proposition}\label{prop:alpha_correct}
 The exhaustive $\alpha$-decomposition algorithm finds an indecomposable decomposition.
\end{proposition}

\begin{proof}
 Let $(M_\B \ N)$ be an $\alpha$-decomposed presentation of a module $Y$ 
 which we want to decompose. It will be enough to show that if there is a decomposition 
 into two summands, then we will find it. 
 
  So let $\varphi \colon Y \iso Y_1 \oplus Y_2$ be such a decomposition. 
 Use \autoref{matrices_cd} to produce matrices $Q, \, P, \, U, \, [T_1 \ T_2] \in \K^{k \times \left( k_1 + k_2 \right)}$ such that after application we reach a presentation $[M_\B \ N_1 \ N_2]$ and a partition $\B = \B_1 \cup \B_2$ such that $(N_1)_{\B_2}=0$ and $(N_2)_{\B_1}=0$. Also we denote again by $X_b$ the module presented by $M_b$, $X_1, X_2$ the modules presented by $M_{\B_1}$, $M_{\B_2}$. 
 
 We can assume that neither $Y_1$ nor $Y_2$ contain a summand without relations above $\alpha$. Otherwise we can use \ref{matrices_cd} on the resulting decomposition to get, again, matrices which would delete a submatrix $N_b$, but then {\texttt{BlockReduce}} \hyperref[hom_computation]{($\ast \ast $)} would have earlier deleted this already. \\

 The loop over ${\Dec}_q(k)$ in the algorithm will then at some point have chosen the subspace spanned by $T_1$, but maybe a different basis, say given by $T_1'$ such that $T_1' R_1  = T_1$ for some invertible matrix $R_1$. Then it will have chosen some arbitrary complement $T_2'$ and we can complete $R_1$ to a larger invertible matrix satisfying
  \[ \begin{bmatrix} T'_1 & T'_2 \end{bmatrix} \cdot \begin{bmatrix} R_1 & R_{1,2} \\ 0 & R_2 \end{bmatrix} =  \begin{bmatrix} T_1 & T_2\end{bmatrix}  \] 

 The next step in the algorithm then sets $N'_1 \coloneqq NT_1'$ and $N'_2 \coloneqq NT_2'$ and by using again the matrices which lift the decomposition $\varphi$, the new matrix now satisfies the equation

 \begin{align}\label{eq_1} Q \begin{bmatrix}
    M_{\B} & N'_1 & N'_2 \\
   \end{bmatrix} 
   \begin{bmatrix}
                       P & U_1 & U_2   \\
                       0 & R_1 & R_{1,2} \\
                       0 & 0 & R_2
                      \end{bmatrix} = 
   \begin{bmatrix}
    M_{\B}  & N_1 & N_2 \\
   \end{bmatrix} \end{align}
and we denote this isomorphism of presentations by $\psi$. We notice that since the third matrix is triangular, after ignoring the third column on both sides, this induces another isomorphism $\psi_1$ of the truncated presentations. 
Multiplying on both sides of the equation the second and third column with $R_1^{-1}$ and $R_2^{-1}$ from the right will also not change the block, structure of the result on the right, so we can assume wlog. that $R_1 = \Id_{k_1}$ and $R_2 = \Id_{k_2}$. 
   
    After the first step the algorithm will have zeroed out $\left(N'_1\right)_b$ for a maximal set $\B_2' \subset \B$ with $\B_1' \coloneqq \left( \B_2' \right)^{\complement}$ its complement. 
    Denote the isomorphism of presentations induced by this first step $\eta$ and again its restriction to the matrix above $\eta_1$. Concretely it has the following components:

    \[ \eta = \left( \begin{bmatrix}
        \Id & 0 \\
        Q_{2,1}' & Q_{2,2}
    \end{bmatrix},  \
    \begin{bmatrix}
    \Id & 0 & 0  \\
    P'_{2,1} & P'_{2,2} & U_{2,1}'  \\
    0 & 0 & \Id \\
    \end{bmatrix}^{-1} \right)
    \colon \
    \begin{bmatrix}
    M_{\B'_1} & 0 & {N_{1}'}_{\B_1}  \\
    0 & M_{\B'_2} & {N_{1}'}_{\B_2}
    \end{bmatrix}
     \iso 
    \begin{bmatrix}
    M_{\B'_1} & 0 & N'_{1} \\
    0 & M_{\B'_2} & 0
    \end{bmatrix}.
    \]
   
    Using the analogous names for the presented modules, the isomorphisms $\psi$ and $\eta$ induce an isomorphism 
    \[ Y_1 \oplus X_2 \oto{\left( \coker \eta_1 \right) \circ  \left( \coker \psi_1^{-1} \right) } Y_1' \oplus X'_2 \]
    with the right side also having the property that $Y_1'$ contains no direct summand of the form $X_b$ for the same reason as $Y_1$. Therefore $X_2 \iso X_2'$ and $Y_1 \iso Y'_1$ by Krull-Remak-Schmidt-Azumaya \autoref{krull-remak-schmidt}, so after an appropriate permutation of isomorphic blocks we can assume $\B_1' = \B_1$ and $\B_2' = \B_2$. \\
     Notice that in the components of $\eta$, the matrices $(Q_{2,2}, P^{'-1}_{2,2})$ form an automorphism of $M_{\B'_2}$, so that we may post-compose with it to assume wlog. that all diagonal blocks of $\eta$ are identity matrices. \\
    The pair of matrices making up $\psi \circ \eta^{-1} $ which bring us from were we are now in the algorithm to the decomposition from the beginning is then given by the compositions
    \[ \left( \begin{bmatrix}
    Q_{1,1} & Q_{1,2} \\
    Q_{2,1} & Q_{2,1}
    \end{bmatrix}
    \begin{bmatrix}
        \Id & 0 \\
        -Q_{2,1}' & \Id
    \end{bmatrix}, \
    \begin{bmatrix}
                       P_{1,1} & P_{1,2} & U_{1,1} & U_{1,2}   \\
                       P_{2,1} & P_{2,2} & U_{2,1} & U_{2,2} \\
                       0 & 0 & \Id & R_{1,2} \\
                       0 & 0 & 0 & \Id
                      \end{bmatrix}^{-1}
     \begin{bmatrix}
    \Id & 0 & 0  & 0 \\
    P'_{2,1} & \Id & U_{2,1}' & 0  \\
    0 & 0 & \Id & 0 \\
    0 & 0 & 0 & \Id  \\
    \end{bmatrix}^{-1}
         \right)
    \]
    Restricting to the top right corner ( $ \left(N'_2 \right)_{\B_1'}$) of the presentations this isomorphism implies the following equation.
    \begin{align*} 0 =  & \begin{bmatrix}
    Q_{1,1} - Q_{1,2}Q'_{2,1} & Q_{1,2} 
    \end{bmatrix} 
    \begin{bmatrix}
    M_{\B'_1} & 0 & N_{1}' &  {N_{2}'}_{\B_1}\\
    0 & M_{\B'_2} & 0 & {N_{2}'}_{\B_2}
    \end{bmatrix}
    \begin{bmatrix}
                        U_{1,2}   \\
                        P'_{2,1}U_{1,2} + U_{2,2} + U'_{2,1}R_{1,2}\\
                        R_{1,2} \\
                         \Id
                      \end{bmatrix} \\
    = & \underbrace{\left(Q_{1,1} - Q_{1,2}Q'_{2,1}\right)}_{\eqqcolon \widetilde Q}\left(M_{\B_1'}U_{1,2} + N'_{1} R_{1,2}+ {N_2'}_{\B'_1} \right) +  \underbrace{Q_{1,2} M_{\B'_2}}_{ = M_{\B'_1} P_{1,2}}  \underbrace{ \left(P'_{2,1}U_{1,2} + U_{2,2} + U'_{2,1}R_{1,2}\right)}_{\eqqcolon \widetilde P}  + Q_{1,2} {N'_{2}}_{\B'_2} \\
    & \Rightarrow  -{N_2'}_{\B'_1}  =  M_{\B_1'}U_{1,2} + N'_{1} R_{1,2} + \underbrace{ \widetilde Q^{-1} M_{\B'_1} }_{ = M_{\B'_1} P^{-1}_{1,1} } P_{1,2} \widetilde P + \widetilde Q^{-1} Q_{1,2} {N'_{2}}_{\B'_2} \\
     = & \  M_{\B_1'} \left( U_{1,2} + P_{1,1}^{-1} P_{1,2} \widetilde P \right) + N_1' R_{1,2} + \left( Q^{-1} Q_{1,2} \right) {N'_{2}}_{\B'_2}
    \end{align*}
We have used that $\psi \circ \eta$ fixes $M_{\B_1}$, so that $\widetilde Q M_{\B_1} = M_{\B_2} P_{1,1}$ by inspection of the top left corner of the presentations. We also saw that because we applied the permutation $\sigma$ and with the permutation used in \autoref{matrices_cd} to ensure the property $(\ast \ast)$, these matrices must be invertible, so that $\widetilde Q$ is invertible and it follows that $\widetilde Q^{-1} M_{\B_1} = M_{\B_2} P^{-1}_{1,1}$. This supplies the matrices for {\texttt{BlockReduce}} \hyperref[hom_computation]{($\ast \ast $)}.
\end{proof}

\section{An Automorphism-Invariant Decomposition Algorithm}\label{sec:aida}

 While exhaustive $\alpha$-decomposition terminates within reasonable time for small $k$, say, $k \leq 7$, we encounter presentations where this approach becomes unsustainable. We will see in the next example that iteration over subspaces as in \autoref{subparagraph:exhaustive} can often be avoided partially.

\subsection{Automorphism invariant submodules.}\label{sec:submodules}

\begin{example}\label{example:digraph}
 Consider the $(1,3)$-decomposed presentation over $\F_2$ extending \autoref{example:blockstructure}. We observe:  $\Hom\left(M_b, M_c\right) \simeq \K$, $\Hom\left(M_c, M_d\right) \simeq \K$, and $\Hom\left(M_b, M_d\right) \simeq \K^2$ and all other $\Hom$-sets are trivial.
 We depict this information in a digraph structure on $\B$.

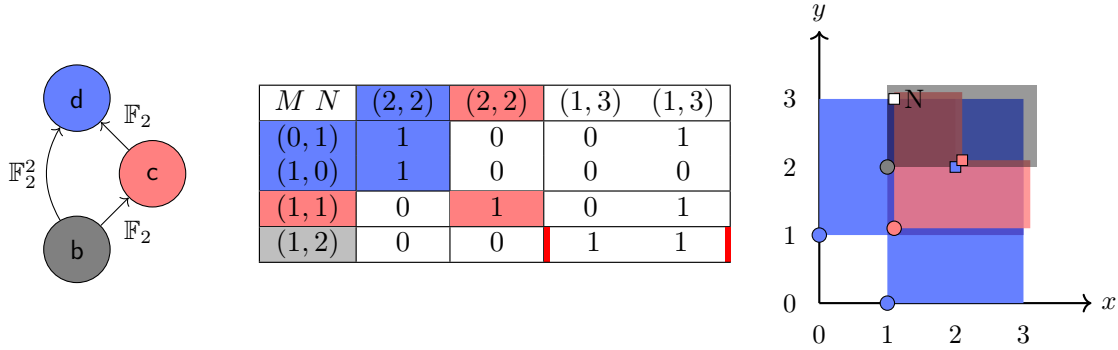
\begin{figure}[H]
    \centering
    \begin{minipage}[c]{0.2\textwidth}
        \begin{tikzpicture}
            \node[circled blue] (d) at (0,2) {d};
            \node[circled red] (c) at (1,1) {c};
            \node[circled gray] (b) at (0,0) {b};

            \draw[->] (b) -- node[below right] {$\F_2$} (c);
            \draw[->] (c) -- node[above right] {$\F_2$} (d);
            \draw[->] (b) to[bend left] node[left] {$\F_2^2$} (d);
        \end{tikzpicture}
    \end{minipage}
    \begin{minipage}[c]{0.4\textwidth}
        \resizebox{\textwidth}{!}{
        $
        \begin{array}{| c | c | c | c  c |}
            \hline 
             M \ N & \lbcell (2,2) &  \lrcell  (2,2) & (1,3) &  (1,3) \\
            \hline
            \lbcell (0,1) & \lbcell 1 & 0 &  0 & 1 \\
            \lbcell (1,0) & \lbcell 1 & 0 & 0 & 0 \\ 
            \hline
             \lrcell (1,1) & 0 &  \lrcell 1 & 0 & 1 \\ 
             \hline
             \lgcell (1,2) & 0 & 0 & \multicolumn{1}{!{\color{red}\vrule width 2pt}c!{}}{ 1} & \multicolumn{1}{!{}c!{\color{red}\vrule width 2pt}}{ 1} \\
         \hline
        \end{array}
        $}
    \end{minipage} \hspace{1em}
    \begin{minipage}[c]{0.35\textwidth}
    \begin{tikzpicture}[scale = 0.9]
            \draw[->, thick] (0,0) -- (4,0) node[right] {$x$};
            \draw[->, thick] (0,0) -- (0,4) node[above] {$y$};
        
            \draw[fill=lightblue, draw=none, opacity=1, thick] (0,3) -- (0,1) -- (1,1) -- (1,0) -- (3,0) -- (3,3);
            \draw[fill = darkblue, draw=none, opacity=1, thick]
            (1,3) -- (1,1) -- (3,1) -- (3,2) -- (2,2) -- (2,3);
    
            \draw[fill = lightred, draw= none, opacity = 0.8, thick]
            (1.1,3.1) -- (1.1,1.1) -- (3.1,1.1) -- (3.1,2.1) -- (2.1,2.1) -- (2.1,3.1);

            \draw[fill = Gray!200, draw=none, opacity = 0.4, thick]
            (1.0,3.2) -- (1.0,2.0) -- (3.2, 2.0) -- (3.2, 3.2);
            
            \foreach \x/\y in {0/1, 1/0} {
                \fill[lightblue, draw = black] (\x,\y) circle (3pt);
            }
        
            \foreach \x/\y in {2/2} {
                \node[rectangle, draw=black, fill=lightblue, inner sep=2pt] at (\x,\y) {};
            }

            \foreach \x/\y in {1.1/1.1} {
                \fill[lightred, draw = black] (\x,\y) circle (3pt);
            }

            \foreach \x/\y in {2.1/2.1} {
                \node[rectangle, draw=black, fill=lightred, inner sep=2pt] at (\x,\y) {};
            }

            \foreach \x/\y in {1/2} {
                \fill[Gray, draw = black] (\x,\y) circle (3pt);
            }

            \node[rectangle, draw = black, fill=white, inner sep=2pt] at (1.1,3) {};
            \node at (1.4, 3) {N};
            
            \foreach \x in {0,1,2,3} {
                \node at (\x, -0.2) [below] {\x};
            }
            \foreach \y in {0,1,2,3} {
                \node at (-0.2, \y) [left] {\y};
            }
        \end{tikzpicture} \hfill
    \end{minipage}
    \caption{The digraph on $\B$ tells us which sub-modules are \emph{invariant} under change of basis. If $\B$ is acyclic, then the group of automorphisms of $X$ is \emph{unipotent} as a subgroup of $\GL_{|G|}(\K)$.}
    \label{fig:tikz_and_array}
\end{figure}
There are no maps to $b$. So, if we want to change $N_b$ then we can only use admissible column operations. If we consider the last row as a presentation itself, it is not decomposed maximally: The two $1$s form a block, which can be made smaller by adding the 4th column to the 3rd column. That is, we are \emph{minimising} this sub-presentation, arriving at the presentation on left. The $1$ at the bottom of the last column index can now never be deleted through column or row-operations, so it must become a fixed part of the block $b$.
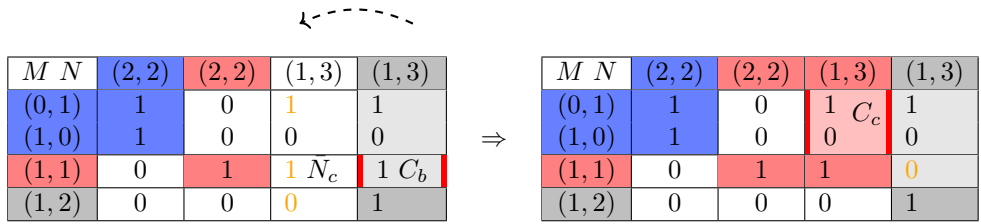
\begin{figure}[H]
\centering
 \[ 
\begin{array}{| c | c | c | l | l |}
            \hline 
             M \ N & \lbcell (2,2) &  \lrcell  (2,2) & (1,3) &  \lgcell (1,3) \\
            \hline
            \lbcell (0,1) & \lbcell 1 & 0 &  \otext{1} & \llgcell 1 \\
            \lbcell (1,0) & \lbcell 1 & 0 & 0 & \llgcell 0 \\ 
            \hline
             \lrcell (1,1) & 0 &  \lrcell 1  & \otext{1} \ \bar N_{c}  & \multicolumn{1}{!{\color{red}\vrule width 2pt}c!{\color{red}\vrule width 2pt}}{\llgcell 1 \ C_b}  \\ 
             \hline
             \lgcell (1,2) & 0 & 0 & \otext{0} & \lgcell  1 \\
         \hline
        \end{array} \quad \Rightarrow \quad 
        \begin{array}{| c | c | c | l | l |}
            \hline 
             M \ N & \lbcell (2,2) &  \lrcell  (2,2) & \lrcell (1,3) &  \lgcell (1,3) \\
            \hline
            \lbcell (0,1) & \lbcell 1 & 0 &  \multicolumn{1}{!{\color{red}\vrule width 2pt}l!{\color{red}\vrule width 2pt}}{\llrcell 1} & \llgcell  1 \\
            \lbcell (1,0) & \lbcell 1 & 0 &  \multicolumn{1}{!{\color{red}\vrule width 2pt}l!{\color{red}\vrule width 2pt}}{\llrcell 0}   & \llgcell  0 \\ 
            \hline
             \lrcell (1,1) & 0 &  \lrcell 1 & \lrcell 1 & \llgcell  \otext{0} \\ 
             \hline
             \lgcell (1,2) & 0 & 0 & 0 & \lgcell 1 \\
         \hline
        \end{array}
 \]
 \begin{tikzpicture}[overlay, remember picture]

        \draw[->, thick, black, bend right=25, dashed] (-1, 3) to (-2.5, 3); 

        \node at (5.0, 1.8) {$C_c$ };
        
    \end{tikzpicture}
    \caption{On the right side we can observe how the digraph on $\B$ also controls the column-operations between the columns of $N$ which we can use}
    \label{fig:alpha_decomp}
\end{figure}

Next we turn to the second block. The subbatch $N_c$ has a part $\bar N_{c}$ sitting above the cleared part of $N_b$ and a part $C_b$ in the new column of $b$. Now if we were again to add the 4th column to the 3rd, then we could not delete the 1 that would appear at $M_{4,3}$ with any row-operations, because there are no maps to $M_b$! So the non-existence of maps \emph{to} a block restricts the column operations on $N$ away \emph{from} the block. \\
 This implies that we can not use any operations from outside of $c$ to decompose $\left[ M_c \ \bar N_c \right]$ and since we cannot reduce $\bar N_c$ further, the 3rd column must become part of the block $c$. Then $C_b$ can be zeroed out using again {\texttt{BlockReduce}} \hyperref[hom_computation]{($\ast \ast$)} and we reach the presentation above on the right. Next we try to reduce $N_d$ and start with the part $C_c$ in the new column of $c$, because the column-operation from 4th to 3rd is still impossible. We cannot zero out $C_c$ with admissible operations, so the two blocks $c$ and $d$ need to \emph{merge}:
 \begin{figure}[H]
\centering
 \begin{minipage}[c]{0.5\textwidth}
 \[ \Rightarrow \quad 
\begin{array}{| c  c  c  l | l |}
            \hline 
             M \ N & \lpcell (2,2) &  \lpcell  (2,2) & \lpcell (1,3) &  \lgcell (1,3) \\
            \hline
            \lpcell (0,1) & \lpcell 1 & \lpcell  0 & \lpcell 1 & \multicolumn{1}{!{\color{red}\vrule width 2pt}l!{\color{red}\vrule width 2pt}}{\llgcell 1 } \\
            \lpcell (1,0) & \lpcell 1 & \lpcell  0 & \lpcell 0 & \multicolumn{1}{!{\color{red}\vrule width 2pt}l!{\color{red}\vrule width 2pt}}{\llgcell 0 } \\ 
             \lpcell (1,1) & \lpcell 0 &  \lpcell 1 & \lpcell 1 & \multicolumn{1}{!{\color{red}\vrule width 2pt}l!{\color{red}\vrule width 2pt}}{\llgcell 0 } \\ 
             \hline
             \lgcell (1,2) & 0 & 0 & 0 & \lgcell 1 \\
         \hline
        \end{array}
 \]
 \begin{tikzpicture}[overlay, remember picture]


        \node at (6.8, 1.5) {$C_b$ };
        
    \end{tikzpicture}
\end{minipage}
\hspace{1em}
\begin{minipage}[c]{0.4\textwidth}
    \begin{tikzpicture}
            \node[circled, fill = lightpurple] (d) at (0,2) {$c \cup d$};
            \node[circled gray] (b) at (0,0) {b};
            \draw[->] (b) to[bend right] node[right] {$\F_2^3$} (d);
        \end{tikzpicture}
\end{minipage}
    \label{fig:alpha_merge}
\end{figure}
We update all $\Hom(-,-)$ spaces and the resulting digraph. Then we can continue processing the new cocycle $C_b$.
 \end{example}
 
 \subparagraph{Extensions with cocycle at $\alpha$.}
 We have seen in \autoref{example:digraph}, that if $\B$ is acyclic with respect to $\Hom \neq 0$, it is possible to process $N$ in a way so that all admissible operations we need to consider only flow in the direction of the graph. Consequently, the change-of-basis matrices we need to find are block-triangular, and the system we need to solve is linear in every step so that there is no iteration over subspaces needed. In particular, the algorithm would then also work for any field $\K$, not only a finite one. \\
Let us formalise the structure of the matrix we need to decompose:

\begin{definition}\label{def:extension_presentation}
An $\alpha$-decomposed presentation 
$\begin{bmatrix}
         M_\Cs & 0 & N_\Cs &  C \\
         0 &  M_{\Ds} & 0 &  N_{\Ds}  
\end{bmatrix} 
$
is called  
\emph{$\alpha$-decomposed extension}, if $\left[M_\Cs \ N_\Cs\right]$ and 
$\left[M_{\Ds} \ N_{\Ds}\right]$ are minimal indecomposable block-decompositions. We denote by $\kappa$ the number of columns of $C$.
\end{definition}

\begin{remark}
Note that here $C$ trivially satisfies the cocycle condition (\ref{cocycle}) of \autoref{pres_ext}: $N_\Ds$ only has columns at $\alpha$, so if these column vectors were linearly dependent, then $\left[M_{\Ds} \ N_{\Ds}\right]$ would not be minimal.
\end{remark}

Let $X, Y$ be the persistence modules presented by the two submatrices $\left[M_{\Cs} \ N_{\Cs} \right]$ and $\left[M_\Ds \ N_\Ds\right]$ in \autoref{def:extension_presentation} and $Z$ the module presented by the entire matrix $\left[ M \ N \right]$. Since $C$ satisfiey the cocycle condition, the inclusion of $\left[ M_\Cs \ N_\Cs \right]$ into $\left[ M \ N \right]$ induces a short exact sequence 
$0 \to X \to Z \to Y \to 0$ classified by $[C] \in \Ext^1(Y, X)$. The map $C \colon A^{\kappa}(-\alpha)\to A^{G_\Cs}$
is supported only on the upper set $\up{\alpha}$, so it becomes $0$ after restricting to its complement.

\begin{definition}\label{def:concentrated}
Let $U \subset \Z^d$ be an upper set and $V \coloneqq U^\complement$. A ses. $0 \to X \oto{f} Y \oto{g} Z \to 0$ is called $U$-concentrated or $V$-split if it splits after restricting the modules from $\Z^n$ to $V$. \\
The same terminology applies to an injection $f$ or a surjection $g$ which induces such a sequence.
\end{definition}

With this definition, an $\alpha$-decomposed presentation corresponds exactly to a $U$-concentrated ses where both $X$ and $Z$ are indecomposably decomposed.

\subsection{Decomposition Algorithms}\label{sec:decomp_algo}

\subparagraph{Decomposing $\alpha$-decomposed extensions.} 
Assume that in an $\alpha$-decomposed extension matrix (\autoref{def:extension_presentation}) the set of blocks $\Ds$ contains only a single block. Then to see if we can split off $\Ds$ is equivalent to zeroeing out all of $C$. If we virtually merge all blocks in $\Cs$ and add the columns in $N_\Cs$ to these blocks, too, then the linear system which tests if this is possible is basically equivalent to {\texttt{BlockReduce}} \hyperref[hom_computation]{($\ast \ast$)}.

\subparagraph{Decomposition for acyclic $\B$.}

In \autoref{example:digraph} we have outlined a strategy how to decompose an $\alpha$-decomposed presentation if its associated digraph $\B$ is acyclic. We will first outline a slightly different version which does not yet care about the column-operations at $\alpha$ between the blocks.

\begin{algorithm}[Automorphism-invariant \texorpdfstring{$\alpha$}{alpha}-decomposition, acyclic version]\label{stable_decomp_acyclic}  \ \newline
\texttt{Input:}  $\alpha$-decomposed presentation $\left[M_\B \ N \right]$ \\
Endow $\B$ with a digraph structure: $b \rightarrow c$ iff $\Hom(M_b, \, M_c) \neq 0$ and assume $\B$ to be acyclic. 

$\B_p \subset \B$ will store those blocks in $\B$ which have been \emph{processed}.  \\
For each $b \in \B$ in an order compatible with the digraph $\B$: 
\begin{enumerate}
     \item Let $\bar N_b$ be the sub-matrix spanned by the columns where $N_{\B_p}$ is zero.
     \item  Minimise $\left[M_b  \ \bar N_b\right]$.
    \item For each $c \in \B_p$, define $C_{c}$ to be the  part of $N_b$ occupying the same columns as $c$.
    \item For $c \in \B^p$ in \emph{any} order:
     \begin{enumerate}
          \item Define $\Cs \coloneqq \{ d \in \B_p \, | \ C_d \neq 0 \}$.
          \item Treat all blocks in $\Cs$ and $b$ as one large block (without merging yet) and pass the $\alpha$-decomposed extension \\
          $\left[
            \begin{array}{ccccc|c}
            M_{b} & 0 & 0 & \bar N_b & C_{\Cs} & C_c \\
            0 & M_\Cs & 0 & 0     &   N_\Cs & 0 \\
            \hline
            0 & 0 &  M_c & 0 & 0 & N_c
            \end{array}
            \right]$
          to {\texttt{BlockReduce}} \hyperref[hom_computation]{($\ast \ast$)}. 
          
    \end{enumerate}
    \item Merge the blocks by setting $b \coloneqq b \cup \Cs$ and update the $\Hom$-spaces.
\end{enumerate}
Return $[M_\B \ N ]$.
\end{algorithm}

\subparagraph{Column-operations in $\alpha$-decomposed extensions.}
 In \autoref{example:digraph} we presented a strategy which was actually a bit different. Here we traversed the set $\B_p$ in an order opposite to the one in $\B$.

\begin{example}\label{column_operations}
Consider the $\alpha$-decomposed extensions on the bottom left. We assume that the relation $\Hom \neq 0$ turns the set of blocks $\{b, c, d\}$ into a graph without any maps in the directions $b \to c \to d$. The standing assumption is still that any change of basis should not change parts of the matrix we have already decomposed, in this case the red and blue blocks.  In \autoref{fig:alpha_decomp} we observed that we could then not use any column operations from $N_d$ to $N_c$, because there are no row-operations from $c$ to $d$ which form a morphism to reverse the change to the gray cell. Instead, we can compute for any pair of blocks, if there are column-operations whose effect on the whole matrix can be reverted and in this way endow $\B_p$ with another digraph structure.
\end{example}

\begin{construction}\label{con_col_ops}
For \emph{minimal} presentations $[ M_b \ N_b] \in \K^{\left( R_b \cup k_b \cdot \alpha \right) \times G_b}$, $[ M_c \ N_c] \in \K^{\left( R_c \cup k_c \cdot \alpha \right) \times G_c}$ with $\alpha$ maximal, the induced column operations at $\alpha$ are given by the image of the map
\[ \pi_\alpha \colon \, \Hom\left([ M_b \ N_b], \, [ M_c \ N_c] \right) \to \Hom\left( A^{k_b}(-\alpha), \, A^{k_c}(-\alpha) \right) \]
sending a pair $(Q, \, \begin{bmatrix} 
        P & U \\
        0 & T
       \end{bmatrix})$ to $T$.
\end{construction}

\begin{figure}[H]
\centering
\large\[
        \begin{bmatrix}
              M_b & 0 & 0 & 0    &   {N_{b,c}} &  {N_{b,d}}  \\[4pt]
             0 & \lrcell \rn{4}{M_c} & 0 & 0     &  \lrcell \rn{1}{ N_c}   & 0 \\[4pt]
             0 & \rn{5}{0} &  \lbcell M_d & 0  &  \llgcell \rn{2}{0} & \lbcell  \rn{3}{N_d} \\ 
        \end{bmatrix}
        \quad \quad 
        \begin{bmatrix}
             M_b & 0 & 0 & 0    &   N_{b,c} &  N_{b,d} & N_{b,e} \\[4pt]
             0 & \lrcell M_c & \rn{6}{0} & 0     &  \lrcell  N_c   & \rn{10}{0}  & \lrcell C \\[4pt]
             0 & 0 &  \lbcell \rn{7}{M_d} & \rn{8}{0}  &   0 & \lbcell \rn{11}{N_d} & \rn{12}{0} \\[4pt] 
             0 & 0 & 0 &  \lrcell \rn{9}{M_e} & 0 & 0 & \lrcell \rn{13}{N_e}
        \end{bmatrix}
\]
        \begin{tikzpicture}[overlay, remember picture]

        \draw [->] (1) -- (2) node[pos=0.5,red,sloped] {\tiny$/$};
        \draw [->] (3) -- (2) node[pos=0.5,red,sloped] {\tiny$/$};
        \draw [->] (4) -- (5) node[pos=0.5,red,sloped] {\tiny$/$};
        
        \draw [->, green] (7) -- (6);
        \draw [->, green] (9) -- (8);
        \draw [->, green] (11) -- (10);
        \draw [->, green] (13) -- (12);

        \draw[->, thick, black, bend right=25, dashed] (-1.8, 2.7) to (-2.8, 2.7);  
        \node[red] at (-2.3, 2.8)  {\small$/$};

        \draw[->, thick, black, bend left=25, dashed] (4, 2.9) to (4.9, 2.9);  
        \draw[->, thick, black, bend left=25, dashed] (5.1, 2.9) to (6.1, 2.9);  
        
        \end{tikzpicture}
\end{figure}

\begin{example}\label{cycle_creation}
Turn to the $\alpha$-decomposed extensions on the right. Again we assume that the relation $\Hom \neq 0$ turns the set of blocks $\{b, c,d,e\}$ into a graph without any maps in the directions $b \to c \to d \to e$. Assume also that using \autoref{con_col_ops} we found column operations in the order $c \to d \to e$.

A cocycle $C$ above $e$ was not deleted. In \autoref{example:digraph} we suggested that we should then merge the blocks $c$ and $e$. But if we did so, there would be both column operations at $\alpha$ from $\{c,e\}$ to $d$ and $d$ to $\{c,e \}$ creating a cycle in $\B_p$.
\end{example}

In practice we will compute the column-operations not by first computing all homomorphisms between the extended blocks, as in \autoref{con_col_ops}, but try to see which of the already computed homomorphisms will extend to the newly added columns ( \autoref{compute_col} Appendix).

In the full algorithm, if we have created a cycle on $\B_p$, then in the next step there is no consistent order in which we can traverse the processed blocks in this manner. It is possible to just not merge the blocks and keep track of the individual sub-matrices and while we employ this strategy in our implementation, it is tedious to write down explicitly. Instead, we shall explain how to deal with cycles, both on $\B$ and $\B_p$.

\subparagraph{Decomposition for general $\B$.}
In general there is no order in which we can traverse the blocks, such that maps to the current block come only from parts of the matrix which already form an indecomposable decomposition. The solution here is to contract all strongly connected components of $\B$ forming the \emph{condensation} $\T$ of this digraph and process all blocks in a component together. Then we can do the same with the already processed blocks $\B_p$. We also give this set a digraph structure by tracking in which direction the column operations go (\autoref{con_col_ops}). Again, if there are cycles on this graph, we contract the strongly connected components.

Following the strategy of \autoref{example:digraph}, in each step we will now need to decompose an \hyperref[def:extension_presentation]{$\alpha$-decomposed extension} where $\Cs$ and $\Ds$ may contain many blocks and so the question if it can be decomposed is no longer equivalent to being able to zero out $C$. We employ another algorithm \hyperref[extensiondecomp]{$\alpha$-\texttt{ExtensionDecomp}} (Appendix) which loops over subspaces like \autoref{exhaustive_full}, but of $\K^\kappa$, the column-space of $C$. The resulting algorithm is FPT in $\kappa_{\max}$, the largest occurrence of $\kappa$, and called the Automorphism-invariant Iterative Decomposition Algorithm.

\begin{algorithm}[Automorphism-Invariant $\alpha$-Decomposition]\label{stable_decomp} \
\newline
\texttt{Input:}  $\alpha$-decomposed presentation $\left[M_\B \ N \right]$ \\
Give $\B$ a digraph structure: $b \rightarrow c$ iff $\Hom(M_b, \, M_c) \neq 0$. \\
$\B_p$ will store those blocks in $\B$ which have been \emph{processed}. \\
Compute the condensation $\T$ of $\B$ and a topological order of $\T$. \\
In this order, let $\Cs \in \T$ be a set of blocks in $\B$: 
\begin{enumerate}
     \item Denote by $\bar N_\Cs$ be the submatrix spanned by the columns for which $N_{\B_p}$ is zero. Pass $[M_\Cs \ \bar N_\Cs]$ to \autoref{exhaustive_full}
     \item Compute the column operations (\autoref{con_col_ops}) between all $d \neq c \in \B_p$ via \autoref{compute_col}, thereby defining another digraph structure on $\B_p$ with  orientation opposite to $\B$.
     \item Compute the condensation $\Us$ of $\B_p$ and a topological order. For $\Ds \in \Us$: 
    \begin{itemize}
        \item Pass $\begin{bmatrix} 
      M_\Cs  & 0 & \bar N_\Cs  & C \\
      0 &  M_\Ds & 0 &  N_\Ds
      \end{bmatrix}$  to \hyperref[extensiondecomp]{$\alpha$-\texttt{ExtensionDecomp}} (Appendix)
      \item Where $C$ could not be deleted, merge blocks from $\Ds$ into the blocks in $\Cs$. 
    \end{itemize}
    \item Add $\Cs$ to $\B_p$.
\end{enumerate}
\item Return $[M_\B \ N_\B]$.
\end{algorithm}

\subsection{Reducing $\B$ via Localisation}\label{sec:localisation}

If we could somehow reduce the maps between the blocks, then this would either make the strongly connected components of $\B$ smaller and thus reduce $\kappa$ or even avoid iteration all together.

 \begin{observation}\label{obs:alpha_hom}
 Let $(Q, \, P) \colon \, M_b \to M_c$ be a morphism of presentations computed in the \hyperref[main_routine]{\textsc{aida} main routine} when processing the columns of degree $\alpha$. The induced map of persistence modules $\widetilde Q \colon X_c \rightarrow X_b$ itself restricts to a linear map of vector spaces $ \left(\widetilde Q \right)_\alpha \colon \left(X_c \right)_\alpha \rightarrow \left(X_b \right)_\alpha $. If this map is zero, it cannot contribute to the decomposition of the $\alpha$-decomposed presentation $( M_\B \ N)$. \\
Concretely since $ \Ima N_c$ is generated at $\alpha$, $\Ima QN_c \into A^{G_b} \onto X_b$ must be zero, so each column of $QN_c$ is a linear combination of the columns of $M_b^{\leq \alpha}$ (cf. \hyperref[colsweeps]{Column-Sweeps}).
 \end{observation}

\begin{definition}\label{def:U_hom}
Define for any $X, Y \in \grA$ and $U \subset \Z^d$ the $\K$-vector space 
\[ \Ic_U\left(X, Y\right) \coloneqq \{f \colon X \to Y \ |  \ \forall \alpha \in U \ f_\alpha = 0 \}.\] 
$V_U$ forms an ideal in $\grA$ and we form its quotient category $\grA^{U}$ whose $\Hom$-sets are
\[\Hom^U(X,Y) = \Hom(X, Z) / \Ic_U\left(X, Y\right).\]
If $U= \{\alpha\}$, then we write $\Ic_\alpha$ and $\Hom^\alpha$ instead.
\end{definition}

Thus, instead of finding a basis for the whole solution space of  {\texttt{BlockReduce}} \hyperref[hom_computation]{($\ast$)} we actually only need to compute a set of solutions which span $\Hom^\alpha \left(X_c, X_b \right)$.
This will decrease the size of the system \hyperref[hom_computation]{($\ast \ast$)}, but more importantly it will sometimes reduce the space of homomorphisms to zero. If that happens, the digraph $\B$ loses an edge and in the best case maybe even a cycle. In the $\alpha$-decomposed extensions which we then need to decompose, we can then only guarantee that all maps from $\Cs$ to $\Ds$ are $0$ at $\alpha$. Let us define this more generally on the algebraic level.

\begin{definition}\label{def:filtration}
Let $0 \to X \oto{f} Y \oto{g} Z \to 0$ be concentrated in $U$. Denote by $ W \coloneqq \supp{\beta_{*,1}(Y)}$ the support of the relations of $Y$. Then the short exact sequence ($f$, $g$) is called \emph{$U$-invariant} (monomorphism, epimorphism) if the $U$-projective covers (\autoref{U_cover}) satisfy $\Hom^{U\cap W}\left(X^{U}, Z^{U}\right) = 0$.
\end{definition}

Notice that the condition means that $X$ is "almost" invariant under the action of $\Aut(Z)$: No automorphism could alter $X_{W \cap U}$. The cocycle $[C]$ which classifies the sequence is generated at most in $W$ as a map from the relations of $Z$ and also in $U$ by definition. \\

To see that the algorithm works even with $\Hom^\alpha$-sets, we will need the following result. The existence of the necessary change-of-basis matrices is guaranteed by lifting the isomorphisms it produces.

\subparagraph{Computing $\Hom^\alpha$.}

 For graded matrices $M \in \K^{G \times R}, \ N \in \K^{G' \times R'}$, denote by $M^{\leq \alpha}, \, N^{\leq \alpha}$ the submatrices formed from column and rows of degree $\leq \alpha$. Given a morphism of minimal presentations $(Q,P)\colon \ M \to N$ the induced map at $\alpha$ is computed by the following diagram of \emph{vector spaces} just by forgetting all generators and relations not in the lower set $\down{\alpha}$:

\begin{align}\label{alpha_map}
\begin{tikzcd}[scale cd = 1.2, ampersand replacement=\&]
	{\K^{|R_{ \leq \alpha}| }} \&\& {\K^{|G_{ \leq \alpha}| }} \&\& {(\coker M)_\alpha} \& 0 \\
	{\K^{|R'_{ \leq \alpha}| }} \&\& {\K^{|G'_{ \leq \alpha}| }} \&\& {(\coker N)_\alpha} \& 0
	\arrow["{M^{\leq \alpha}}", from=1-1, to=1-3]
	\arrow["{P^{\leq \alpha}}", from=1-1, to=2-1]
	\arrow["{p_\alpha}", from=1-3, to=1-5]
	\arrow["{Q^{\leq \alpha}}", from=1-3, to=2-3]
	\arrow["{B^\alpha}"', shift right, curve={height=12pt}, dotted, from=1-5, to=1-3]
	\arrow[from=1-5, to=1-6]
	\arrow["{\widetilde Q^{\leq \alpha}}", from=1-5, to=2-5]
	\arrow["{N^{\leq \alpha}}", from=2-1, to=2-3]
	\arrow["{q_\alpha}", from=2-3, to=2-5]
	\arrow["{B^{'\alpha}}"', shift right, curve={height=12pt}, dotted, from=2-5, to=2-3]
	\arrow[from=2-5, to=2-6]
\end{tikzcd}
\end{align}

Then the diagram implies that, $\Hom^\alpha(X,Y)$ is isomorphic to the quotient vector space
\[ \bigslant{\left\{ (Q,P) \colon \, M \to N \right\}}
{\{ (Q,P) \ | \
q_\alpha \cdot \left(Q^{\leq \alpha}\right) \cdot B^\alpha = 0
\}}\]

Computing these cokernels is fast in our case, because during the algorithm, for every block $b \in \B$ we have already column-reduced the matrices $M_b^{\leq \alpha}$ in \hyperref[colsweeps]{Column-Sweeps}. Since the cokernel is the solution to $\left( M_b^{\leq \alpha} \right)^{T}y = 0$ we can copy it from the matrix directly and $B_b^\alpha$ is given by the standard basis vectors on the set of non-pivot rows of $M_b^{\leq \alpha}$.

\subsection{Fast Decomposition for Interval Decomposable Modules with \textsc{aida}}\label{sec:interval_decomp}

We show that, for interval-modules, the $\Hom^\alpha$-sets reduce the digraph of blocks sufficiently in a way that allows us to avoid the iteration over subspaces.


\begin{proposition}[Presentations of interval modules]\label{prop:pres_interval}
    Every interval module has a minimal presentation in which every column contains exactly one $1$ or two non-zero entries $1$ and $-1$.
\end{proposition}

\begin{proof}
Let $\varphi \colon \K_u \iso M$ for an interval $U \subset \Z^d$. Let $G$ be the graded set of minimal elements $\alpha_i \in U$. The canonical maps $\colon A(-\alpha_i) \to \K_U$ sending $1 \in A$ to $1 \in \left(\K_U\right)_{\alpha_i}$ together form a minimal cover $p_0 \colon A^G \to \K_U$. 
Choose a total order on $G$. For every join $\alpha_{i,j}$ where $i <_G j$ of two generators the analogous map $A(-\alpha_{i,j}) \to A^G$ sending $1$ to $(1 -1) \in A(-\alpha_i) \oplus A(-\alpha_j)$ is in the kernel of $p_0$. Gather all joins $\alpha_{i,j}$ in a graded set $R_1$ and use the aforementioned maps to build a map $A^{R_1} \to A^G$. 
Let $(U)$ be the upper set generated by $U$ and $R_2$ the graded set of minimal elements of $(U) \setminus U$. For each $\beta_k \in R_2$ choose an $\alpha_{i_k}$ such that $\alpha_{i_k} < \beta_k$. These choices induce another map $A^{R_2} \to A^G$ and we paste them together to a map $A^{R_1} \oplus A^{R_2} \oto{p_1} A^G \oto{\varphi \circ p_0} M \to 0$. This sequence is exact:

$p_0$ is surjective: For every $\beta \in U$ consider a homogeneous element $c \in (\K_U)_\beta$. There is a minimal element $\alpha_i < \beta$ and $ cx^{\beta - \alpha_i} \in A(-\alpha_i)$ is a preimage.

$\Ima p_1 \subset \ker p_0$: Since $M_{\beta} = 0$ for all $\beta (U) \setminus U$ we have $p_1\left( A^{R_2} \right) \subset \ker \varphi \circ p_u $ and for $A^{R_1}$ this follows by construction.

$\ker p_0 \subset \Ima p_1$: Let $\beta \in \Z^d$, $I \coloneqq \{i \, | \, \alpha_i \leq \beta \}$ and $v \in A^G_{\beta} \simeq \bigoplus_I \K$ be any homogeneous element such that $ p_0(v) = 0$. choose any re-indexing $[n] \simeq I$ compatible with the previously chosen order on $G$.
Notice that $(p_0)_\beta(v) = \sum_{i \in I} v_i$. We construct the vector
\[ w \coloneqq \left( \sum\limits_{j = 1}^{i} v_j \right)_{i = 1}^{n-1} \in \K^{n-1} \simeq \bigoplus\limits_{i = 1}^{n-1} A(-\alpha_{i,i+1})_\beta .\]
Since for each $i \in [n-1]$ the map $p_1$ sends $A(-\alpha_{i,i+1})$ only to $A(-\alpha_i) \oplus A(-\alpha_{i+1})$
\[ \text{ we have } \ p_1(w)_i =  \sum\limits_{j = 1}^{i}  v_j -   \sum\limits_{j = 1}^{i-1} v_j = v_i \ \text{ for } \ i < n \ \text{ and } \ p_1(w)_n = - \sum\limits_{j = 1}^{n-1} v_j = v_n. \]

This presentation may not yet be minimal, because there might be too many joins in $R_1$ of the same degree. Delete columns until it is. 
\end{proof}

\begin{remark}
The preceding proposition also gives an algorithm to determine interval-decomposability during the algorithm: Whenever blocks need to be merged after completing the decomposition with a new set of relations we can check if the new relations satisfies the condition and sits at the join of two generators.
\end{remark}

\subparagraph{Homomorphisms of interval modules.}\label{sub:interval_hom}
For $X, \, Y$ interval modules we can describe the reduction given by considering $\Hom^{\alpha}$ precisely. Notice that $\Hom\left(X, Y \right)$ has a basis consisting of "valid" overlaps (\cite[Proposition 16]{dey_xin}):
\begin{figure}[H]
\begin{minipage}[c]{0.4\textwidth}
\begin{tikzpicture}[scale=0.85]
    \draw[->, thick] (0,0) -- (5.6,0) ;
    \draw[->, thick] (0,0) -- (0,4.4) ;

    \draw[fill = lightblue] 
        (0.5, 3.7) -- (0.5,1) -- (1.5,1) -- (1.5,0.1) -- (4.7,0.1 ) -- (4.7,1) -- (3.5,1) -- (3.5,2) -- (1.5,2) --  (1.5, 3.7) -- cycle;

    \draw[fill = red, opacity=0.5] 
        (0.1,4.1) -- (0.1,3.1) -- (2.6, 3.1) -- (2.6, 1.5) -- (3.9, 1.5) -- (3.9,0.2) -- (5.1,0.2) -- (5.1,2.1) -- (4.6,2.1) -- (4.6,4.1) -- cycle;

    \node[rectangle, draw = black, fill=white, inner sep=1pt] at (4.3, 0.5) {};
    \node at (4.3, 0.7) {$\alpha$};

    \node at (3.7, 2.8) {$X$};
    \node at (2.3, 0.9) {$Y$};
            
    \draw[->, thick] (2.1, 3.8) to[out=170, in=20] (1, 3.4);
    \node at (2.1, 3.8) [right] {invalid};
    
    \foreach \x in {0,1,2,3,4,5} {
        \node at (\x, -0.2) [below] {\x};
    }
    \foreach \y in {0,1,2,3,4} {
        \node at (-0.2, \y) [left] {\y};
    }
\end{tikzpicture}
    \end{minipage}
        \caption{ $\Hom\left(X, Y\right) \simeq \K^2$, $\Hom^\alpha\left(X, Y\right) \simeq \K$}
    \label{fig:alpha_hom}
\end{figure}
 At most one of the valid overlaps can contain $\alpha$ and so $\dim_K\Hom^\alpha(X,Y) \leq 1$.  

\begin{lemma}\label{lem:alpha_interval}
For any $\alpha \in \N$, the full subcategory of $\grA^\alpha$ spanned by the interval-modules is the $\K$-linearisation of a thin category.
\end{lemma}

Define a relation $X \rightarrow_\alpha Y \colon\Leftrightarrow \Hom^\alpha(X, Y) \neq 0$. The Lemma says that this is a preorder on any set of interval modules and a partial order on a set of isomorphism classes.

\begin{proof}
 Let $X$ and $Y$ have supports $I,J$ which contain $\alpha$.
 From poposition 16 in \cite{dey_xin} we deduce that $\Hom^\alpha(X, Y) \simeq \K \neq 0$ precisely if for each $\gamma \in I \cap J$
 \[ \left( \eta \leq \gamma \text{ and } \eta \in I \right) \Rightarrow \eta \in J \text{ and } \left( \lambda \geq \gamma \text{ and } \lambda \in J \Rightarrow \lambda \in I \right).\]
 Assume first that both $X_b \rightarrow_\alpha X_c$ and $X_c \rightarrow_\alpha X_b$ and assume there was a$\delta \in J \setminus I$. By connectedness there is a path from $\delta$ to a point $\xi$ in $I \cap J$, but any downwards part must belong to $I$ by the first assumption and every upwards path must then belong to $I$, too giving a contradiction and by the same reasoning $I \setminus J$ is also empty. Transitivity follows directly.
\end{proof}

\subparagraph{Interval-Decomposables.}
Assume that the input to \hyperref[main_routine]{\textsc{aida}} is an interval-decomposable module. 
 
 When we consider a set of new relations $N$ in the algorithm, $M_b$ now presents an interval module and each sub-column $N_{b,i}$ encodes a map to it at $\alpha$. An interval module can have at most dimension $1$ at $\alpha$ and so $N_{b,i}$ can be reduced to only have a single non-zero entry indicating the rank of this map and so we treat $N_{b,i}$ as being this entry. This, together with the transitivity of $\alpha$-maps (\autoref{lem:alpha_interval}) allows Gaussian (row-)elimination on $N$ (Appendix~\ref{sub:interval_blockreduce}) avoiding the first occurrence of {\texttt{BlockReduce}} \hyperref[hom_computation]{($\ast \ast$)} in the \hyperref[main_routine]{\textsc{aida} main routine} (2b) to delete the whole sub-matrices $N_b$ where possible.

 In the following application of \hyperref[stable_decomp]{Automorphism-invariant $\alpha$-decomposition} we need to decompose $\alpha$-decomposed extension where $\Cs$ and $\Ds$ are strongly connected components of the digraph on $\B$. But $\B$ is transitive in the interval-decomposable case, so $\left[M_{\Cs} \ N_{\Cs} \right]$ and $\left[M_\Ds \ N_\Ds\right]$ must each present a set of isomorphic intervals. Then all row and column operations on $C$ are possible and we can use normal matrix reduction (\autoref{sub:interval_algorithm}).

\subparagraph{Automorphism-invariant \texorpdfstring{$\alpha$}{alpha}-decomposition for interval decomposable modules.}\label{sub:interval_algorithm} \ \\
\texttt{Input:} 
$\alpha$-decomposed$\left(M_\B \ N \right)$ with each $M_b$ presenting an interval-module. \\
The digraph structure on $\B$ is a preorder. \\
Whenever blocks merge it is easy to check if they still form an interval. If they do, then $\B_p$ will also carry a preorder. \\
Each strongly connected component of $\B$ and $\B_p$ is a complete graph. Let $\Cs \subset \B$ such a component (in a compatible order):
\begin{enumerate}
     \item $\bar N_\Cs$ is a equivalent to a matrix with $|\Cs|$ rows. Perform normal matrix reduction.
     \item For each strongly connected component $\Ds \subset \B^p$ (in a compatible order):
     \begin{enumerate}
     \item The matrix $(M_\Ds \ N_\Ds)$ must still present a set of isomorphic intervals, so each $N_d$, $d \in \Ds$ can only have a single non-zero column.
      \item Define $C$ as before. It must be equivalent to a $|\Cs| \times |\Ds|$-matrix where all operations are possible. Perform normal matrix reduction.
     \end{enumerate}
\end{enumerate}

 \begin{theorem}\label{theorem:interval_decomposition}
 The method outlined above decomposes interval decomposable modules in time $\Oc(n^3)$ where $n$ is the number of generators and relations.
 \end{theorem}
 
\begin{proof}\label{proof:interval}
Let $m$ be the number of generators and $n$ the number of relations of the input presentation.

Consider a new set of $k_i$ columns of degree $\alpha_i$.
Deciding if $\alpha_i$ is contained in an interval can be done in linear time, so wlog. assume that $\alpha$ is contained in every interval presented by an $M_b$.

We first do a column sweep to reduce the sub-matrices $N_b$. Since every column in $M_b$ has only $1$ or $2$ entries (\autoref{def:interval}), this whole process for each $b$ is linear in $n$ and $m$. For every $b$ there is a generator of largest row-index $b_i$ of degree $\leq \alpha_i$. After the reduction, for each $j \in [k]$ the entry $N_{b_i, j}$ is the only non-zero entry of $N_{b, j}$ and signals its rank. We restrict all $N_{b,j}$ to this entry and get a matrix $\widetilde N$ of dimensions $|\B| \times k$.

Next we compute all $\Hom^\alpha$ sets between pairs of interval modules. For every $b \in \B$ let $m_b$ be the number of generators and relations of the interval indexed by $b$, so that $\sum_i m_i \leq m+n$. By \autoref{prop:interval_hom} (from \cite{dey_xin}) the run time for this step is in $\sum_{j \neq i} \Oc(m_i m_j) \subset  \Oc \left( \left( \sum_i m_i \right) \left( \sum_j m_j \right) \right) =  \Oc((m+n)^2)$. Notice that each non-zero $\Hom^\alpha$-map we have computed this way corresponds to a single elementary row-operation on $\widetilde N$.

The Gaussian elimination replacing \texttt{BlockReduce} will then take $\Oc(m^2 k )$ time. 

For the Automorphism-invariant decomposition, we need to compute a condensation $\T$ which takes time at most $\Oc(m^2)$ and a topological order, which is in linear time.

Minimising each $(M_\Cs \ \bar N_\Cs)$ is just column elimination and for every $C$ we can use normal matrix reduction as explained above. Every entry of $\widetilde N$ can only appear in a single $C$ or $\bar N_\Cs$ so that in total we need another $\Oc(m^2 k )$ steps for the matrix reductions.
\end{proof}

\section{Restriction and Localisation}\label{sec:restriction_and_localisation}

\subsection{Exact Structures from Restriction}

Let again $U \subset \Z^d$ be an upper set and $V$ its complement. Recall that a $U$-concentrated/$V$-split sequence is one that splits after applying $\iota_V$.

\ignore{
\begin{proposition}
Consider two monomorphisms $X \ito{f} Y \ito{g} Z$. If $f, g$ are $U$-concentrated, then so is $g \circ f$. If $g$ and $g \circ f$ are $U$-concentrated, then so is $f$.
\end{proposition}

\begin{proof}
Apply $\iota_V \colon \grA \to \fun(V, \vect_\K)$ to replace "$U$-concentrated" by "split".
\end{proof}}

\begin{proposition}\label{concentrated}
 Let $0 \to X \oto{f} Y \oto{g} Z \to 0$ be a short exact sequence of persistence modules. The following are equivalent.
\begin{enumerate}[label=(\alph*)]
  \item If the graded matrices $ [ M_1 \ N_1 ]$, $[M_2 \ N_2]$ present $X,Z$ with the matrices $N_1, N_2$ containing all relations in $U$, then there is a presentation of $Y$ of the form
  \[ \begin{bmatrix}
      M_1 & 0 & N_1 & C \\
      0 & M_2 & 0 & N_2
     \end{bmatrix} \colon A^{R_1} \oplus A^{R_2} \oplus A^{R'_1} \oplus A^{R'_2} \to A^{G_1} \oplus A^{G_2}
\]
where $C$ is a cocycle.
\item If $[ C ] \in \Ext^1(Z, X)$ classifies the short exact sequence and $p \colon P \to Z$ is any map from a $U$-projective module, then $p^*[C]=0 \in \Ext^1(P, X)$.

\item For any choice of $U$-covers $p_X \colon X^U \to X$, $p_Z \colon Z^{U} \to Z$ we can construct maps $p, C$ which induce a short exact sequence of $U$-projective presentations as follows
\[\begin{tikzcd}[ampersand replacement=\&]
	0 \& {A^{R_X}} \& {A^{R_X} \oplus A^{R_Z}} \& {A^{R_Z}} \& 0 \\
	0 \& {X^{U}} \& {X^{U} \oplus Z^{U}} \& {Z^{U}} \& 0 \\
	0 \& X \& Y \& Z \& 0
	\arrow[from=1-1, to=1-2]
	\arrow["{i_1}", from=1-2, to=1-3]
	\arrow["{d_X}", from=1-2, to=2-2]
	\arrow["{p_2}", from=1-3, to=1-4]
	\arrow["{ \begin{bmatrix} d_X & C \\ 0 & d_Z \end{bmatrix} }", from=1-3, to=2-3]
	\arrow[from=1-4, to=1-5]
	\arrow["{d_Z}", from=1-4, to=2-4]
	\arrow[from=2-1, to=2-2]
	\arrow["{i_1}", from=2-2, to=2-3]
	\arrow["{p_X}", from=2-2, to=3-2]
	\arrow["{p_2}", from=2-3, to=2-4]
	\arrow["p", from=2-3, to=3-3]
	\arrow[from=2-4, to=2-5]
	\arrow["{p_Z}", from=2-4, to=3-4]
	\arrow[from=3-1, to=3-2]
	\arrow["{^f}", from=3-2, to=3-3]
	\arrow["g", from=3-3, to=3-4]
	\arrow[from=3-4, to=3-5]
\end{tikzcd}\]
where $R_X, R_Z$ are multi-sets in $U$ and $C$ is a cocycle.
\item The sequence is $U$-concentrated (\autoref{def:concentrated}).
 \end{enumerate}

\end{proposition}

\begin{proof}
 (a) $\Rightarrow$ (b): 
 
  The composition $C' \colon A^{R_2} \oto{C} A^{G_1} \oto{q} X$ is a cocycle in $\Hom(A^{R_2}, X)$ and the class $[C'] \in \Ext^1(Z, X)$ represents the given short exact sequence (\autoref{pres_ext}). 
  Consider a minimal presentation $ A^R \to A^G$ of $P$ and let $p_1 \colon A^R \to A^{R_2}$ be a lift of $p$. Then $p^*[C'] = [ C' \circ p] = [q \circ C \circ p] = 0$ since $P$ has no relations in $U$. \\
 (b) $\Rightarrow$ (c): 
 
 Choose $U$-covers $p_X \colon X^U \to X, p_Z \colon Z^U \to Z$. Form the pullback of $g$ along $p_Z$, then (b) says that the short exact sequence in the middle is split and we can compose with $p_X \oplus \Id$. The rest is analogous to \autoref{pres_ext}.
\[\begin{tikzcd}
	0 & {X^U} & {X^U \oplus Z^U} & {Z^U} & 0 \\
	0 & X & {X \oplus Z^U} & {Z^U} & 0 \\
	0 & X & Y & Z & 0
	\arrow[from=3-4, to=3-5]
	\arrow["g", from=3-3, to=3-4]
	\arrow["f", from=3-2, to=3-3]
	\arrow[from=3-1, to=3-2]
	\arrow["\Id", from=2-2, to=3-2]
	\arrow["{(f, p)}", from=2-3, to=3-3]
	\arrow["{p_2}", from=2-3, to=2-4]
	\arrow[from=2-4, to=2-5]
	\arrow["{p_Z}", from=2-4, to=3-4]
	\arrow["\lrcorner"{anchor=center, pos=0.125}, draw=none, from=2-3, to=3-4]
	\arrow[from=2-1, to=2-2]
	\arrow["{i_1}", from=2-2, to=2-3]
	\arrow["\Id", from=1-4, to=2-4]
	\arrow["{p_X}", from=1-2, to=2-2]
	\arrow["{i_1}", from=1-2, to=1-3]
	\arrow["{p_2}", from=1-3, to=1-4]
	\arrow[from=1-4, to=1-5]
	\arrow[from=1-1, to=1-2]
	\arrow["{p_X \oplus \Id}", from=1-3, to=2-3]
\end{tikzcd}\]
(c) ``$\Rightarrow$'' (d): 

Apply $\iota_V^*$ to the diagram, then $p_X$ and $p_Z$ become isomorphisms, so by the five lemma also $\iota^*_V p$ is an isomorphism. \\
(d) ``$\Rightarrow$'' (a): 

Construct, as in \autoref{pres_ext}, a diagram of presentations for $X, Y, Z$ with $A^{R_1'}$, $A^{R_2'}$, $A^{G_1'}$, $A^{G_2'}$ containing all free summands whose degree is in $U$. Again $C \colon A^{R_2 + R_2'} \to A^{G_1 + G_1'}$ is a cocycle representing the short exact sequence. Apply the (right exact) flattening functor to get a surjection of presentations and use that by \autoref{planing} this is the same as dropping all free summands whose degree is in $U$. Denote by $C'$ the restriction of $C$ to $A^{R_2}$ and $A^{G_2}$.
\[\begin{tikzcd}[ampersand replacement=\&]
	0 \& {A^{R_1}} \& { A^{R_1}\oplus A^{R_2} } \& {   A^{R_2} } \& 0 \\
	0 \& {A^{G_1}} \& { A^{G_1}\oplus A^{G_2}} \& {  A^{G_2} } \& 0 \\
	0 \& {\left(\iota_{V}\right)_!\iota_{V}^* X} \& {\left(\iota_{V}\right)_!\iota_{V}^* Y} \& {\left(\iota_{V}\right)_!\iota_{V}^* Z} \& 0
	\arrow[from=1-1, to=1-2]
	\arrow["{i_1}", from=1-2, to=1-3]
	\arrow["{M_1}"', from=1-2, to=2-2]
	\arrow["{p_2}", from=1-3, to=1-4]
	\arrow["{\begin{bmatrix} M_1 & C' \\ 0 & M_2 \end{bmatrix}} ", from=1-3, to=2-3]
	\arrow[from=1-4, to=1-5]
	\arrow["{ M_2 }", from=1-4, to=2-4]
	\arrow[from=2-1, to=2-2]
	\arrow["{i_1}", from=2-2, to=2-3]
	\arrow[from=2-2, to=3-2]
	\arrow["{p_2}", from=2-3, to=2-4]
	\arrow[from=2-3, to=3-3]
	\arrow[from=2-4, to=2-5]
	\arrow[from=2-4, to=3-4]
	\arrow[from=3-1, to=3-2]
	\arrow["{\left(\iota_{V}\right)_!\iota_{V}^* f}", from=3-2, to=3-3]
	\arrow["{\left(\iota_{V}\right)_!\iota_{V}^* g}", from=3-3, to=3-4]
	\arrow[from=3-4, to=3-5]
\end{tikzcd}\]
Now by (c) $\iota_V^*$ made the sequence split, and $(\iota_V)_!$ preserves direct products, so the bottom row is still not only split but also exact. Therefore $C'$ must be zero using the snake lemma.
\end{proof}

 Using \autoref{concentrated} (a) we see that the $U$-concentrated short exact sequences form a subgroup $\Ext^1_{\Pc_U}(Z, X) \subset \Ext^1(Z, X)$ and \autoref{concentrated} (b) together with the naturality of $\Ext^1$ shows that $\Ext^1_{\Pc_U}$ is a sub-bifunctor of $\Ext^1$.

\begin{remark}
A short exact sequence $0 \to X \to Y \to Z \to 0$ does not have a maximal lower set over which it splits. Consider the module $Y$ presented by \\
 \begin{minipage}[c]{0.3\textwidth}
 \[ 
\begin{array}{| c | c  c | }
        \hline 
         M  & \lgcell (1,2) & \lgcell  (2,1)  \\
        \hline
         \lbcell (0,0)    &   1   &   0       \\
         \lrcell (1,1)    &   1   &   1       \\
        \hline
    \end{array}
 \]
\end{minipage}
\hspace{3em}
 \begin{minipage}[c]{0.5\textwidth}
        \begin{tikzpicture}[scale = 0.9]
            \draw[->, thick] (0,0) -- (3.5,0) node[right] {$x$};
            \draw[->, thick] (0,0) -- (0,3.5) node[above] {$y$};
        
            
            \draw[fill=lightblue, draw=none, opacity=1, thick] (0,3) -- (0,0) -- (3,0) -- (3,2) -- (2,2) -- (2,3);
            
            \draw[fill = lightpurple, draw=none, opacity=0.8, thick]
            (1,3) -- (1,2) -- (2,2) -- (2,3);
    
            \draw[fill = lightred, draw= none, thick]
            (1,1) -- (2,1) -- (2,2) -- (1,2);

            \foreach \x/\y in {0/0} {
                \fill[lightblue, draw = black] (\x,\y) circle (3pt);
            }

            \foreach \x/\y in {1/1} {
                \fill[lightred, draw = black] (\x,\y) circle (3pt);
            }

            \foreach \x/\y in {1/2, 2/1} {
                \fill[Gray, draw = black] (\x,\y) circle (3pt);
            }

            \foreach \x in {0,1,2,3} {
                \node at (\x, -0.2) [below] {\x};
            }
            \foreach \y in {0,1,2,3} {
                \node at (-0.2, \y) [left] {\y};
            }
        \end{tikzpicture}  \hfill
\end{minipage} \hfill \\
and the injection of the free module $i \colon A \into Y$. The map $i$ splits over both $\stcomp{\langle (1,2) \rangle}$ and $\stcomp{\langle (2,1) \rangle}$ but not over any lower set strictly larger than either of the two.
\end{remark}

\begin{proposition}\label{proj_exact}
A surjection $g \colon Y \onto Z$ is $V$-split if and only if the pushforward $\Hom(P,Y) \oto{g_*} \Hom(P,Z)$ is surjective for every $U$-projectve module $P$.
\end{proposition}

In other words the $\Hom(\mathcal{P}_U, -)$ exact sequences are precisely the $V$-split sequences. 
We could have also defined $V$-splitness like this and then used the horseshoe lemma for a general covering class \cite[8.2.1]{EnochsJenda} to get \autoref{concentrated}.

\begin{proof}
Let $p \colon P\to Z$ be any homomorphism. The pullback $p^*\colon \Hom(Z, -) \to \Hom(P, -)$ induces a commutative diagram.
\[\begin{tikzcd}[ampersand replacement=\&]
	{ \Hom(P, Y)} \& { \Hom(P, Z)} \& {\Ext^1(P, \ker g)} \\
	{ \Hom(Z, Y)} \& { \Hom(Z, Z)} \& {\Ext^1(Z, \ker g)}
	\arrow["\delta", from=2-2, to=2-3]
	\arrow["{g_*}", from=2-1, to=2-2]
	\arrow["{g_*}", from=1-1, to=1-2]
	\arrow["{p^*}"', from=2-2, to=1-2]
	\arrow["{p^*}"', from=2-1, to=1-1]
	\arrow["{p^*}"', from=2-3, to=1-3]
	\arrow["\delta", from=1-2, to=1-3]
\end{tikzcd}\]
Consider $\Id_Z$ as an element in $\Hom(Z,Z)$. Its image under $\delta$ classifies the short exact sequence $0 \to \ker g \to Y \to Z \to 0$ and by commutativity
\[ p^* \circ \delta \left( \Id_Z \right) = \delta \circ p^* \left( \Id_Z \right) = \delta \left( p \right). \] 

The map $g$ is $V$-split if and only if  $0 = p^* \circ \delta \left( \Id_Z \right)$ by \autoref{concentrated} b) and by exactness the last term is zero if and only if $p \in \Ima g_*$ . 
\end{proof}

\begin{corollary}
    The $V$-split sequences form an exact structure for $\grA$ and $\Pc_U$ is the class of projective modules for this exact structure.
\end{corollary}
    
 \begin{proof}
     Either follows from \autoref{proj_exact} using (for example) \cite{exactcats} Proposition 1.4 and 1.7. or notice that Keller's axioms (\cite[1.1]{exactcats}) must hold using \autoref{concentrated} c).
 \end{proof}
 
\begin{corollary}
$\text{gl left-} \Pc_{\langle \alpha \rangle} \text{-} \dim (\grA) \leq d+1$.
\end{corollary}

\begin{proof}
Let $Z$ be a persistence module and $p \colon P \to Z$ be a $U$-cover. By construction this is a $V$-split surjection. We can choose any free resolution of $\ker p$ and attach this to get a left $\mathcal{P}_U$-resolution (\cite{EnochsJenda} 8.1.2.) and this is a minimal $\mathcal{P}_U$-resolution if the free resolution is chosen minimal. Since $\ker p$ has a resolution of length $d$ this means that the maximal length of a minimal $\mathcal{P}_U$-resolution is $d+1$.
\end{proof}

\begin{proposition}
In general this bound is best possible. For every $ \alpha \in \Z^d$ we have \\
$\text{gl left-} \Pc_{\langle \alpha \rangle} \text{-} \dim (\grA) = d+1$.
\end{proposition}

\begin{proof}
Assume wlog $\alpha = (0, \dots, 1)$ and define $c_j \coloneqq (-1, \dots, -1 , 1, -1, \dots, -1)$. Consider the "cube with missing corner" $X$ presented by the map
\[ \bigoplus_{j \in [d]} A(\minus e_j) \oplus A( \minus \alpha ) \oto{\left( 1 \ \dots \ 1\right)} A( \pcoord{1, \dots, 1}) \to X \to 0. \]

Its minimal $\alpha$-cover is the projection from the cube $p \colon \K [-1,1]^d \onto X$ and the kernel of this map is again a "cube" $\ker p \simeq \K [0,1]^d \simeq A/\m$. For $d=3$ a graphical representation of these modules is the following:

\begin{minipage}[t]{0.33\textwidth}
\begin{tikzpicture}[tdplot_main_coords]

\coordinate (A) at (-1,-1,-1);
\coordinate (B) at ( 1,-1,-1);
\coordinate (C) at ( 1, 1,-1);
\coordinate (D) at (-1, 1,-1);
\coordinate (E) at (-1,-1, 1);
\coordinate (F) at ( 1,-1, 1);
\coordinate (G) at ( 1, 1, 1);
\coordinate (H) at (-1, 1, 1);

\coordinate (P) at (0,0,0);
\coordinate (Q) at (1,0,0);
\coordinate (R) at (1,1,0);
\coordinate (S) at (0,1,0);
\coordinate (T) at (0,0,1);
\coordinate (U) at (1,0,1);
\coordinate (V) at (1,1,1);
\coordinate (W) at (0,1,1);

\draw[thick,->] (0,0,0) -- (2.3,0,0) node[anchor=north east]{$x$}; 
\draw[thick,->] (0,0,0) -- (0,2.3,0) node[anchor=north west]{$y$}; 
\draw[thick,->] (0,0,0) -- (0,0,2.3) node[anchor=south]{$z$}; 

\draw[fill=lightred,opacity=0.7] (P) -- (Q) -- (R) -- (S) -- cycle;
\draw[fill=lightred,opacity=0.7] (P) -- (T) -- (W) -- (S) -- cycle;
\draw[fill=lightred,opacity=0.7] (P) -- (Q) -- (U) -- (T) -- cycle;
\draw[fill=lightred,opacity=0.7] (T) -- (U) -- (V) -- (W) -- cycle;
\draw[fill=lightred,opacity=0.7] (Q) -- (U) -- (V) -- (R) -- cycle;
\draw[fill=lightred,opacity=0.7] (S) -- (R) -- (V) -- (W) -- cycle;

\end{tikzpicture}
\end{minipage}
\begin{minipage}[t]{0.33\textwidth}
\begin{tikzpicture}[tdplot_main_coords]

\coordinate (A) at (-1,-1,-1);
\coordinate (B) at ( 1,-1,-1);
\coordinate (C) at ( 1, 1,-1);
\coordinate (D) at (-1, 1,-1);
\coordinate (E) at (-1,-1, 1);
\coordinate (F) at ( 1,-1, 1);
\coordinate (G) at ( 1, 1, 1);
\coordinate (H) at (-1, 1, 1);

\coordinate (P) at (0,0,0);
\coordinate (Q) at (1,0,0);
\coordinate (R) at (1,1,0);
\coordinate (S) at (0,1,0);
\coordinate (T) at (0,0,1);
\coordinate (U) at (1,0,1);
\coordinate (V) at (1,1,1);
\coordinate (W) at (0,1,1);

\draw[fill=lightblue,opacity=0.7] (A) -- (B) -- (C) -- (D) -- cycle; 
\draw[fill=lightblue,opacity=0.7] (A) -- (B) -- (F) -- (E) -- cycle; 
\draw[fill=lightblue,opacity=0.7] (A) -- (E) -- (H) -- (D) -- cycle; 

\draw[fill=lightblue,opacity=0.7] (E) -- (F) -- (G) -- (H) -- cycle; 
\draw[fill=lightblue,opacity=0.7] (B) -- (F) -- (G) -- (C) -- cycle; 
\draw[fill=lightblue,opacity=0.7] (D) -- (H) -- (G) -- (C) -- cycle; 


\draw[thick,->] (1,0,0) -- (2.3,0,0) node[anchor=north east]{$x$}; 
\draw[thick,->] (0,1,0) -- (0,2.3,0) node[anchor=north west]{$y$}; 
\draw[thick,->] (0,0,1) -- (0,0,2.3) node[anchor=south]{$z$}; 

\end{tikzpicture}
\end{minipage}
\begin{minipage}[t]{0.33\textwidth}
\begin{tikzpicture}[tdplot_main_coords]

\coordinate (A) at (-1,-1,-1);
\coordinate (B) at ( 1,-1,-1);
\coordinate (C) at ( 1, 1,-1);
\coordinate (D) at (-1, 1,-1);
\coordinate (E) at (-1,-1, 1);
\coordinate (F) at ( 1,-1, 1);
\coordinate (G) at ( 1, 1, 1);
\coordinate (H) at (-1, 1, 1);

\coordinate (P) at (0,0,0);
\coordinate (Q) at (1,0,0);
\coordinate (R) at (1,1,0);
\coordinate (S) at (0,1,0);
\coordinate (T) at (0,0,1);
\coordinate (U) at (1,0,1);
\coordinate (V) at (1,1,1);
\coordinate (W) at (0,1,1);

\draw[fill=lightblue,opacity=0.7] (A) -- (B) -- (C) -- (D) -- cycle; 
\draw[fill=lightblue,opacity=0.7] (A) -- (B) -- (F) -- (E) -- cycle; 
\draw[fill=lightblue,opacity=0.7] (A) -- (E) -- (H) -- (D) -- cycle; 

\draw[fill=lightblue,opacity=0.7] (E) -- (F) -- (U) -- (T) -- (W) -- (H) -- cycle; 
\draw[fill=lightblue,opacity=0.7] (B) -- (F) -- (U) -- (Q) -- (R) -- (C) -- cycle; 
\draw[fill=lightblue,opacity=0.7] (D) -- (H) -- (W) -- (S) -- (R) -- (C) -- cycle; 


\draw[thick,->] (0,0,0) -- (2.3,0,0) node[anchor=north east]{$x$}; 
\draw[thick,->] (0,0,0) -- (0,2.3,0) node[anchor=north west]{$y$}; 
\draw[thick,->] (0,0,0) -- (0,0,2.3) node[anchor=south]{$z$}; 

\end{tikzpicture}
\end{minipage}
The minimal resolution of this module is just the (multi-graded) Koszul complex of $A$ and of length $d$. Concretely for every $J \subset [d]$ there is a corner of the cube given by $\sum_{j \in J} e_j$ and a corresponding syzygy in the minimal resolution at the $|J|$-th position.
\end{proof}

\begin{proposition}
Let $X$ be a persistence module. Then with respect to $\mathcal{P}_U$ the contravariant functor $\Hom(-,X)$ has right derived functors $R_U^*\Hom(-,X)$ which can be computed with $\mathcal{P}_U$-resolutions and any $U$-concentrated short exact sequence induces a long exact sequence.
\end{proposition}

\begin{proof}
 \cite{EnochsJenda} 8.2.3.
\end{proof}

\begin{remark}
 Consider the (non-unital) path-algebra $ \mathcal{P}(\Z^d)$ and a basis given by all finite paths in $\Z^d$. For the upper set $U \subset \Z^d$ every finite path $U$ corresponds to a basis element and the sub-algebra which they span is an ideal $I_U$. This ideal induces a surjective (non-unital) algebra homomorphism of the path algebras $p_V \colon \mathcal{P}(\Z^d) \to \mathcal{P}(\Z^d)/ I_U \iso \mathcal{P}(V)$. If we defined persistence modules as modules over this algebra, then the functors $ - \otimes_{\Pc(\Z^d)} \Pc (V) $ and $\iota_V^*$ are naturally isomorphic. \\
The constructions in this subsection basically describe the relative homological algebra \cite{Hochschild} for the projection $p_V$.
\end{remark}

\begin{proposition}\label{ext_fun}
 Let $Z$ be any persistence module. Consider a free resolution and a $\mathcal{P}_U$-resolution of $Z$, then the identity map $\Id_Z$ lifts to a map from the free resolution to the latter. For any $X$ this induces natural transformations $i^*_{\Pc_U} \colon R_U^* \Hom(Z, X) \to \Ext^*(Z, X)$, where $i^0_{\Pc_U} \colon R^0\Hom(Z, C) \iso \Hom(Z, X)$ is a natural isomorphism and $i^{1}_{\Pc_U}$ a natural inclusion which identifies $R_U^1 \Hom(Z, X)$ with the sub-bifunctor $\Ext_{\Pc_U}^1(Z,X)$.
\end{proposition}

\begin{proof}
 The first claim follows from left-exactness of $\Hom(-,X)$. For the second claim, let $A^R$ be the second term in the free resolution and notice that $R_U^* \Hom(Z, X)$ is precisely the subquotient of $\Hom(A^R, X)$ whose image is supported in $U$.
\end{proof}

\begin{remark}
In the decomposition algorithm we considered short exact sequences whose defining cocycle only lives in a single degree $\alpha$. If $U \subset \Z^d$ is not an upper set anymore then we can still define a class of $U$-concentrated exact sequences. The problem with this definition is that there is no corresponding class of projectives and pullbacks along arbitrary maps do not again produce a $U$-concentrated sequences. In particular this means that this does not produce an exact structure or a sub-bifunctor. 
\end{remark}

\begin{definition}
Let $0 \to X \to Y \to Z \to 0$ be a short exact sequence and $U \subset \Z^d$ any subset. We call it $U$-concentrated if it is $\langle U \rangle$-exact and additionally there is a presentation as in \autoref{concentrated} (a) or (c) where the cocycle can be chosen so that as a matrix it is only non-zero in columns with degree in $U$. \\
By abuse of notation we still denote by $\Ext_{\Pc_U}^1(Z, X)$ the corresponding subgroup.
\end{definition}

In the second entry functoriality is preserved.

\begin{proposition}\label{ext_alpha}
 For $U \subset \Z^d$ arbitrary, every homomorphism $f\colon X \to X'$ induces a pushforward 
 $\Ext^1_{\Pc_U}(Z,X) \oto{f_*} \Ext^1_{\Pc_U}(Z,X')$ and 
 if $f_U= 0$, then $f_*=0$.
 If $Z$ has no relations in $\langle U \rangle \setminus U$ then the inclusion $\Ext^1_{\Pc_U}(Z,-) \to \Ext^1_{\Pc_{\langle U \rangle}}(Z,-)$ is a natural isomorphism.
\end{proposition}

\begin{proof}
For the first statement, let $C \colon A^{R_Z} \to X$ be a cocycle classifying the exact sequence then $f \circ C$ represents the image in $\Ext^1(Z,X)$ but still satisfies the condition. For the second statement, if $f_U = 0$ then necessarily $f \circ C = 0$. \\
 The third statement follows directly from the definition, because all columns of the cocycle $C$ automatically have their degree in $U$.
\end{proof}

\subsection{Localisation}\label{sub:localisation}

For an upper set $U \subset \Z^d$ we have seen two constructions which let us study persistence modules through the lense of the lower set $V = U^\complement$: the functor $\left(\iota_{V}\right)_!\iota_{V}^*$ and $U$-covers. The former forgets everything in $U$, whereas the latter only forgets the relations. In \autoref{def:U_hom} we restricted the morphisms of the category $\grA$ to those which have an effect on an (now arbitrary) subset $U$ which should be seen as a localisation at $U$.

\begin{remark}
By \autoref{ext_alpha}, for any $Z \in \grA$ the functor $\Ext^1_{\Pc_U}(Z, -) \colon \grA \rightarrow \grA$ descends to a functor on $\grA^{U}$. These two constructions are also related in other ways, one of which we have exploited in \autoref{sec:aida}.
\end{remark}

\begin{question}
Which properties does $\grA^{U}$ have.
\end{question}

\begin{remark}
The assignments $U \mapsto \Hom^{U}(X, Y)$ and $U \mapsto \Ic_U(X,Y)$  make the two constructions functors
$\mathcal{P}(\Z^n)^{op} \to \grA $, because if $U \subset W$, then $\Ic_W \subset \Ic_U$. We can endow $\mathcal{P}(\Z^n)^{op}$ with the discrete topology and $\Ic_U$ becomes a sheaf by definition. $\Hom^{U}$ on the other hand does not, since it is the quotient of the constant sheaf $\Hom(X, Y)$ by $\Ic_{(-)}$ and would need to be sheafified.
\end{remark}

\subparagraph{Derived Functors}
In the next two propositions, we will briefly investigate the homological properties of the bi-functors we have introduced. This has no relation to the algorithm.

\begin{proposition}
The bifunctors $\Ic_U$ and $\Hom^U$ are additive and can be upgraded to functors $\grA \times \grA^{op} \to \grA$ analogous to the $\grA$-enriched $\Hom$-functor by setting for every $\alpha \leq \beta \in \Z^d$:
\[ \Ic_U\left(X, Y\right)_\alpha \coloneqq \Ic_U\left(X, Y(\alpha) \right) \ \text{ and } \ \Ic_U\left(X, Y\right)_{\alpha \to \beta} \coloneqq Y \left(\alpha \to \beta \right)_*.\]
The same construction works for $\Hom^{U}$.
In addition, $\Ic_U$ is left-exact.
\end{proposition}

\begin{proof}
Only left exactness is not completely direct from the definition. Let $0 \to X \oto{f} Y \oto{g} Z \to 0$ be a short exact sequence and $T \in \grA$. The only non-trivial part is exactness in the middle. Let $T \in \grA$ and consider $\Ic_U(T, -)$. Let $\varphi \colon T \to Y$ be such that $\varphi_{|U}= 0$ and $g \circ \varphi=0$. Then there is a map $\psi\colon T \to X$ such that $(f \circ \psi) = \varphi$, so $(f \circ \psi)_{|U} = 0$. As $f$ is injective, it must also hold that $\psi_{|U} = 0$. For $\Ic_U(-, T)$ the argument is dual. 
\end{proof}

\begin{construction}
The natural inclusion $\Ic_U(-,-) \into \Hom(-,-)$ induces a natural transformation of derived functors $\iota \colon R^*\Ic_U(-,-) \to \Ext^*(-,-)$. \\
Let $X, Y \in \grA$ and $P_* \to X$ be a projective resolution. Applying $\iota$ to $P_*$ we get a short exact sequence of chain complexes 
\[0 \to \Ic_U(P_*,Y) \to \Hom(P_*, Y) \to \Hom^{U}(P_*, Y) \to 0.\]
This induces a long exact sequence of persistence modules, natural in both entries
\begin{align*} 0 \to \Ic_U(X, Y) \oto{\iota} \Hom(X, Y) \oto{\pi} \Hom[U](X, Y) \oto{\delta} \\
R^1\Ic_U(X, Y) \oto{\iota} \Ext^1(X, Y) \oto{\pi} \Ext[U]^1(X, Y) \oto{\delta} \dots
\end{align*}
That is, we can define a bifunctor $\Hom[U](-,-)$ and its derived functors $\Ext[U]^*(-,-)$ in this way. They are computed by the (non-left-exact) functor $\Hom^{U}(-,-)$ and the same construction works, of course, for injective resolutions.

\end{construction}

\begin{proposition}\label{comparison_map}
The natural map $\Ext^1_{\Pc(U)}(X, Y) \into \Ext^1(X, Y) \oto{\pi} \Ext[U]^1(X,Y)$ is an injection.
\end{proposition}

\begin{proof}
Let $\eta \in \Ext^1_{\Pc(U)}(X, Y)$, if $\pi(\eta)=0$, then there is $\xi \in R^1\Ic_U(X, Y)$ st. $\iota(\xi) = \eta$, but $\eta$ is represented by a map whose image is in $U$ and so $\xi$ is, too. But if this map is $0$ in $U$, then $\xi = 0$ and so $\eta = 0$.
\end{proof}

\begin{question}
When is the map from \autoref{comparison_map} an isomorphism?
\end{question}

\subparagraph{Rigidity of $U$-invariant subspaces.}

Consider a short exact sequence $0 \to X \oto{f} Y \oto{g} Z \to 0$. If $\Hom(X, Z)= 0$, then we have argued in \autoref{sec:aida} that a decomposition of $Y$ can be found by first decomposing $X$ and $Z$. In fact, all maps of $X$ into $Y$ are isomorphic: If $f' \colon X \to Y$ is any other injection, then $g \circ f' = 0$, so $f'$ factors through $f$ via a map $\phi \colon Y \to Y$. Then $\phi$ must be an injection, too, so it is an isomorphism. 

If the exact sequence is $V$-split, then we saw that the condition on $\Hom(X, Z)$ can be relaxed. We recall the definition
\newtheorem*{definition*}{Definition}
\begin{definition*}[6.12] 
Let $0 \to X \oto{f} Y \oto{g} Z \to 0$ be concentrated in $U$. Denote by $ W \coloneqq \supp{\beta_{*,1}(Y)}$ the support of the relations of $Y$. Then the short exact sequence (or $f$, $g$) is called \emph{$U$-invariant} (or monomorphism, epimorphism) if $\Hom^{U\cap W}\left(X^{U}, Z^{U}\right) = 0$.
\end{definition*}

\begin{proposition}\label{invariant_minimal}
If $0 \to X \oto{f} Y \oto{g} Z \to 0$ is $U$-invariant, and if we choose minimal covers and minimal $U$-covers for $X$ and $Z$, then both the presentation and the $\mathcal{P}_U$-presentation of $Y$ constructed from these, as in \autoref{concentrated} (a) and (c), are minimal. 
\end{proposition}

\begin{proof}
     In the setting of \autoref{concentrated}, first define the map $\rho \colon A^{G_1} \to X^{U}$, from the projective cover of $X$ to the $U$-projective cover of $X$. Denote by $C$ the cocycle from (a) and by $C'$ the cocycle from (c), then $C' = \rho \circ C \colon A^{R_Z}\to X^{U}$.
     If the presentation from (a) was not minimal, then by
     \autoref{min_pres_criterion} we have $C \otimes A/ \m = 0 \Rightarrow C' \otimes A / \m = \rho \circ C \otimes A/ \m = 0$. If the $\mathcal{P}_U$-presentation from (c) was not minimal, then also $C' \otimes A / \m = 0$, since the $\mathcal{P}_U$-resolution is just a $U$-cover together with a projective resolution of its kernel.
     
      Since $R_Z$ contains only elements in $U$ and $X^{U}$ is $U$-projective, this implies that $C'\colon A^{R_Z}\to X^{U}$ contains a direct summand of the form $A{\small(-\beta)} \oto{\Id} A{\small(-\beta)}$. In particular, $X^{U}$ has a direct summand isomorphic to $A{\small(-\beta)}$.  
     
      Also $A^{R_Z}$ then contains such a summand and by minimality of the $\mathcal{P}_U$-resolution of $Z$, this summand cannot be sent to $0$ by $d_Z$, so $Z^{U}$ is nonzero at $\beta$. Then $\Hom^{U \cap W}(X^{U}, Z^{U})$ contains a direct summand $\Hom^{U \cap W}(A{\small(-\beta)}  , Z^{U}) \simeq Z^{U}_\beta \neq 0$, contradicting $U$-invariance.
\end{proof}

\begin{construction}\label{shear_map}
Let $0 \to X \oto{f} Y \oto{g} Z \to 0$ be $U$-invariant and $h\colon X^{U} \to Z^{U}$. Consider the map 
\[\tilde h \coloneqq 
\begin{bmatrix}
    \Id & 0 \\
    h & \Id 
\end{bmatrix}
\colon X^{U} \oplus Z^{U} \to X^{U} \oplus Z^{U}\]
Given a $\Pc_U$-resolution $A^{R_Y} \oto{d_0} X^{U} \oplus Z^{U} \oto{p_Y} Y$ we have $h \circ d_0 = 0$ by $U$-invariance, so $\left(\tilde h, \Id_{A^{R_Z}}\right)$ induces a \emph{shear} isomorphism $\sh{h} \colon Y \iso Y$.
\end{construction}

\begin{remark}
This defines a group-homomorphism $\sh{(-)} \colon \left( \Hom(X^{U}, Z^{U}), \ + \right) \to \left( \End(Y), \ \circ \right)$ and maps which induce the zero map in $\Hom(X,Z)$ do not have to be mapped to $\Id$.
\end{remark}

\begin{lemma}[Rigidity Lemma]\label{rigidity}
Let $0 \to X \oto{f} Y \oto{g} Z \to 0$ and $0 \to X' \oto{f'} Y \oto{g'} Z'\to 0 $ be $U$-invariant short exact sequences with $X^{U} \simeq {X'}^{U}$, then these sequences are isomorphic.
\end{lemma}

Informally: If we choose any two submodules of $Y$ with the same $U$-cover and if these inclusions are also $U$-invariant, then they are isomorphic.

\begin{proof}
Using \autoref{concentrated} and \autoref{invariant_minimal} both $f$ and $f'$ induce minimal $\Pc_U$-presentations 

\[ A^{R_X} \oplus A^{R_Z} \oto{\begin{bmatrix} M & C \\ 0 & N \end{bmatrix}} X^{U} \oplus Z^{U} \oto{p_Y} Y  \quad \text{ and } \quad A^{R'_X} \oplus A^{R'_Z} \oto{\begin{bmatrix} M' & C' \\ 0 & N' \end{bmatrix}} X^{U} \oplus {Z'}^{U} \oto{p_Y'} Y.\]

and it follows that $Z^{U'} \simeq Z^{U}$ because both covers are minimal. \\
We can lift the identity $\Id_Y$ to a map of $\Pc_U$-resolutions, with the first two components being
\[ \varphi \ =  \begin{bmatrix} \varphi_{1,1} & \varphi_{1,2} \\ \varphi_{2,1} & \varphi_{2,2} \end{bmatrix} \colon \ X^{U} \oplus Z^{U} \to  X^{U} \oplus Z^{U} \quad \text{ and } \quad \psi  \colon A^{R_X} \oplus A^{R_Z} \to A^{R'_X} \oplus A^{R'_Z}\]
 The map $\varphi_{2,1}$ tells us how far $\Ima(f')$ is from being in $\Ima(f)$ and for the sequences to be isomorphic, we must find an automorphism of $Y$ which makes the two overlap.\\
To this end we will now, without loss of generality, assume that no indecomposable direct summand of $X^{U}$ or $Z^{U}$ is zero in $U\cap W$, since such a summand would descend directly to a summand of $X$, $X'$ and $Y$ or $Z$, $Z'$ and $Y$ respectively. Then we could just split these isomorphic components off from each sequence. \\
Then, since $\Hom^{U\cap W}\left(X^{U}, Z^{U}\right) = 0$ we know that $X^{U}$ and $Z^{U}$ have no indecomposable direct summands in common, because the identity map on a component would be non-zero at $U \cap W$. \\
By \autoref{iso_block} we can infer that the restriction $\varphi_{1,1} \colon X^{U} \to X^{U}$ must be an isomorphism.
We can now use the shear map \autoref{shear_map} to perform a "row-operation" to delete $\varphi_{2,1}$ by pasting the maps of presentations together as seen on the right side of the following diagram. 
\[\begin{tikzcd}[ampersand replacement=\&]
	{A^{R_X}} \&\&\& {A^{R_X} \oplus A^{R_Z}} \\
	\&\&\&\& {A^{R_X} \oplus A^{R_Z}} \\
	{X^{U}} \&\& {\quad \quad \quad} \& {\quad X^{U} \oplus Z^{U}} \&\& {A^{R'_X} \oplus A^{R'_Z}} \\
	X \&\& {\quad \quad \quad} \& Y \& {\quad X^{U} \oplus Z^{U}} \\
	\&\& {X^{U}} \&\& Y \& {\quad X^{U} \oplus Z^{U}} \\
	\&\& {X'} \&\&\& Y
	\arrow["{i_1}", from=1-1, to=1-4]
	\arrow["M"', from=1-1, to=3-1]
	\arrow["\psi", from=1-4, to=2-5]
	\arrow[" {\begin{bmatrix} M & C \\ 0 & N \end{bmatrix}}"', from=1-4, to=3-4]
	\arrow["\Id", from=2-5, to=3-6]
	\arrow["{ \begin{bmatrix} M' & C' \\ 0 & N' \end{bmatrix} }", from=2-5, to=4-5]
	\arrow["{\quad i_1}"{pos=0.6}, from=3-1, to=3-4]
	\arrow["{p_X}"', from=3-1, to=4-1]
	\arrow["{\varphi_{1,1}}"'{pos=0.7}, from=3-1, to=5-3]
	\arrow["{p_Y}", from=3-4, to=4-4]
	\arrow[" {\begin{bmatrix} \varphi_{1,1} & \varphi_{1,2} \\ \varphi_{2,1} & \varphi_{2,2} \end{bmatrix}} "{pos=0.3}, from=3-4, to=4-5]
	\arrow[" {\begin{bmatrix} M' & C' \\ 0 & N' \end{bmatrix}} ", shift left=2, from=3-6, to=5-6]
	\arrow["{ \quad f}", from=4-1, to=4-4]
	\arrow["\Id", from=4-4, to=5-5]
	\arrow["{p'_Y}", from=4-5, to=5-5]
	\arrow[" {\begin{bmatrix} \Id & 0 \\ \varphi_{2,1} \circ \varphi_{1,1}^{-1} & \Id \end{bmatrix}} "{pos=0.1}, from=4-5, to=5-6]
	\arrow["{i_1}"'{pos=0.4}, shift right=2, from=5-3, to=5-6]
	\arrow["{p'_X}", from=5-3, to=6-3]
	\arrow["{\sh{\varphi_{2,1} \circ \varphi_{1,1}^{-1}}}"{description, pos=0.4}, from=5-5, to=6-6]
	\arrow["{p'_Y}", shift left=2, from=5-6, to=6-6]
	\arrow["{f'}"'{pos=0.4}, from=6-3, to=6-6]
    \arrow[dashed, from=4-1, to=6-3]
\end{tikzcd}\]

First, observe that with $\varphi_{1,1}$ as a map on the left, thanks to the shear map, the square in the middle layer of the diagram actually commutes. We are left with showing that there are maps $A^{R_X} \to A^{R'_X}$ and $A^{R'_X} \to A^{R_X}$ which commute with $\varphi_{1,1}$ and its inverse to see that the diagram can be completed with an isomorphism between $X$ and $X'$. \\

We will find them with a diagram chase as follows. Observe that we need the map $p'_X \circ \varphi_{1,1} \circ M$ to be $0$, then $\varphi_{1,1} \circ M$ factors through $\ker p'_X$ and by projectivity through $A^{R'_X} \oto{M'} X^{U}$. We underline those maps which are part of a square whose commutativity we use.

\begin{align*} \underline{ f' \circ p'_X } \circ \varphi_{1,1} \circ M = p'_Y \circ \underline{ i_1 \circ \varphi_{1,1} } \circ M = p'_Y \circ \begin{bmatrix} \varphi_{1,1} & \varphi_{1,2} \\ 0 & \varphi_{2,2} - \varphi_{2,1} \varphi_{1,1}^{-1} \varphi_{1,2} \end{bmatrix} \circ  \underline{ i_1 \circ M } \\
= p'_Y \circ \underline{ \begin{bmatrix} \varphi_{1,1} & \varphi_{1,2} \\ 0 & \varphi_{2,2} - \varphi_{2,1} \varphi_{1,1}^{-1} \varphi_{1,2} \end{bmatrix} \circ \begin{bmatrix} M & C \\ 0 & N \end{bmatrix}} \circ i_1  = \underbrace{ p'_Y \circ \begin{bmatrix} M' & C' \\ 0 & N' \end{bmatrix}}_{= 0} \circ \psi \circ i_1 = 0
\end{align*}

The map $f'$ is mono, so $p'_X  \circ \varphi_{1,1} \circ M = 0$ as desired. Since the diagram is symmetric, we find that $p_X  \circ \varphi_{1,1}^{-1} \circ M' = 0$ with the mirrored diagram chase.
\end{proof}

\begin{question}
Is there a proof of (parts of) \autoref{rigidity} without using \autoref{shear_map}?
\end{question}

We can now connect the notion of $U$-invariance more closely to \textsc{AIDA}.

Let $Y \in grA$ be the input module. The condensation $\Us$ of the graph $\left(\B, \, \Hom^\alpha \neq 0 \right)$ on the indecomposable summands of $Y^{U}$, as computed in \autoref{stable_decomp}, finds $U$-invariant monomorphisms with target $Y$. In fact, there are no more than these.

\begin{corollary}\label{u_invariant_char}
Let $Y \in \grA$ be finitely presented and $Y^{U} \iso \bigoplus_{i \in I} Y^{U}_i$ an indecomposable decomposition of its $\mathcal{P}_U$-cover $p \colon Y^{U} \to Y$. Every $U$-invariant monomorphism $X \into Y$ corresponds, up to isomorphism, exactly to one subset $J \subset I$ such that $\Hom^{U \cap W}\left(\bigoplus_{i \in J} Y^{U}_i,  \, \bigoplus_{i \not\in J} Y^{U}_i \right) = 0$.
\end{corollary}

\begin{proof}
 Let $J \subset I$ be any subset. We can factor the map $\bigoplus_{i \in J} Y^{U}_i \into Y^{U} \oto{p} Y$ through its image to get a diagram
\[\begin{tikzcd}[ampersand replacement=\&]
	0 \& {\bigoplus_{i \in J} Y^{U}_i} \& {Y^{U}} \& {\bigoplus_{i \not\in J} Y^{U}_i} \& 0 \\
	0 \& X \& Y \& Z \& 0
	\arrow[from=1-1, to=1-2]
	\arrow[from=1-2, to=1-3]
	\arrow[dashed, two heads, from=1-2, to=2-2]
	\arrow[from=1-3, to=1-4]
	\arrow["p"', two heads, from=1-3, to=2-3]
	\arrow[from=1-4, to=1-5]
	\arrow[dashed, two heads, from=1-4, to=2-4]
	\arrow[from=2-1, to=2-2]
	\arrow[dashed, hook, from=2-2, to=2-3]
	\arrow[dashed, two heads, from=2-3, to=2-4]
	\arrow[from=2-4, to=2-5]
\end{tikzcd}.\]
The bottom row is now $U$-invariant. 

Given a $U$-invariant monomorphism $f \colon X \into Y$, we can lift it to a map on minimal covers $X^{U} \to Y^{U}$. By \autoref{invariant_minimal} this map is split and by \autoref{rigidity} the isomorphism type of $f$ was only dependent on $X^{U}$.
\end{proof}

In \textsc{AIDA}, every topological order on the condensation $\Us$ induced a sequence of $U$-invariant subspaces. 
\[0 = Y_0 \ito{f_1} Y_1 \ito{f_2} \dots \ito{f_{n-1}} Y_{n-1} \ito{f_n} Y_n = Y\]
consisting only of $U$-invariant non-bijective monomorphisms. By the preceding corollary, every maximal filtration of this sort comes, up to isomorphism, from exactly one topological order on $\Us$.

\subsection{Proofs of Correctness}

Assume we are given a short exact sequence $0 \to X \oto{f} Y \oto{g} Z \to 0$ and a decomposition of $Y$. Can we extend it to the whole short exact sequence?
If $\Hom(X, Z) = 0$ then this always works. \\
Let $e$ be an idempotent in $\End(Y)$, then $g \circ e \circ f = 0$ so $e \circ f$ factors through $X$ and we can extend $e$ to an idempotent of the whole short exact sequence. \autoref{rigidity} tells us that if $f$ is $U$-invariant for some $U$ then we can proceed similarly.

\begin{theorem}\label{decomp_concentrated}
 Let $U \subset \Z^d$ and $0 \to X \oto{f} Y \oto{g} Z \to 0$ be a $U$-invariant short exact sequence of persistence modules. If $Y$ is decomposable, then so is the short exact sequence.
\end{theorem}

\begin{proof}
Choose minimal presentations of $X$ and $Z$ and again use \autoref{concentrated} (c) to construct a minimal $\Pc_U$-resolution of $Y$. By \autoref{decomp_resolution} the decomposition $\varphi$ lifts to a decomposition of the entire resolution.
\[\begin{tikzcd}[ampersand replacement=\&]
	\dots \& {A^{R_X} \oplus A^{R_Z}} \&\& {X^U\oplus Z^U} \&\& Y \& 0 \\
	\dots \& {A^{R_{Y_1}} \oplus A^{R_{Y_2}}} \&\& {{Y^U_1} \oplus {Y^U_2}} \&\& {Y_1 \oplus Y_2} \& 0
	\arrow[from=1-1, to=1-2]
	\arrow["N", from=1-2, to=1-4]
	\arrow["\psi", from=1-2, to=2-2]
	\arrow["{p_Y}", from=1-4, to=1-6]
	\arrow["{\varphi^U}", from=1-4, to=2-4]
	\arrow[from=1-6, to=1-7]
	\arrow["\varphi", from=1-6, to=2-6]
	\arrow[from=2-1, to=2-2]
	\arrow["{N_1 \oplus N_2}", from=2-2, to=2-4]
	\arrow["{p_{Y_1} \oplus p_{Y_2}}", from=2-4, to=2-6]
	\arrow[from=2-6, to=2-7]
\end{tikzcd}\]
Consider indecomposable decompositions 
\[X^U \iso \bigoplus_{i \in I} X^U_i \quad \text{, } \quad Z^U \iso \bigoplus_{j \in J} Z^U_j  \quad \text{ and } \quad  Y^{U}_1 \iso \bigoplus_{l \in L_1} Y^{U}_{l} \quad \text{, } \quad Y^{U}_2 \iso \bigoplus_{l \in L_2} Y^{U}_{l}.\]
By Krull-Remak-Schmidt-Azumaya the two sides contain the same indecomposables. That is, for every  $i$ we find $l_i$ in either $L_1$ or $L_2$ such that $X_i^{U} \simeq Y^{U}_{l_i}$. This choice defines a partition $I = I_1 \sqcup I_2$ and split injections
\[ \iota_1 \colon \ \bigoplus\limits_{i \in I_1} X_i \into Y_1^{U} \quad \text{ and } \quad \iota_2 \colon \ \bigoplus\limits_{i \in I_2} X_i \into Y_2^{U}. \]
We can then factor $p_{Y_1} \circ \iota_1 $ and $p_{Y_2} \circ \iota_2 $ through their images to get an injective map 
\[ \Ima \left( p_{Y_1} \circ \iota_1\right)\oplus \Ima \left( p_{Y_2} \circ \iota_2\right) \oto{f_1 \oplus f_2} Y_1 \oplus Y_2. \]
It must be $U$-invariant, because the $U$-covers of 
$\Ima \left( p_{Y_1} \circ \iota_1\right)\oplus \Ima \left( p_{Y_2} \circ \iota_2\right) $ and the quotient are just again $X^{U}$ and $Z^{U}$. Also $\varphi \circ f$ and $f_1 \oplus f_2$ satisfy the premises of \autoref{rigidity} and this constructed an isomorphism of short exact sequences
\begin{align}\label{ses_decomp}\begin{tikzcd}[ampersand replacement=\&]
	0 \& X \&\& {{Y_1} \oplus {Y_2}} \&\& Z \& 0 \\
	0 \& {\Ima \left( p_{Y_1} \circ \iota_1\right)\oplus \Ima \left( p_{Y_2} \circ \iota_2\right)} \&\& {{Y_1} \oplus {Y_2}} \&\& {\coker{f_1} \oplus \coker f_2} \& 0
	\arrow[from=1-1, to=1-2]
	\arrow["{\varphi \circ f}", from=1-2, to=1-4]
	\arrow["\wr", from=1-2, to=2-2]
	\arrow["{g \circ \varphi^{-1}}", from=1-4, to=1-6]
	\arrow["\wr", from=1-4, to=2-4]
	\arrow[from=1-6, to=1-7]
	\arrow["\wr", from=1-6, to=2-6]
	\arrow[from=2-1, to=2-2]
	\arrow["{f_1 \oplus f_2}", from=2-2, to=2-4]
	\arrow["{g_1 \oplus g_2}", from=2-4, to=2-6]
	\arrow[from=2-6, to=2-7]
\end{tikzcd}\end{align}
which supplies the decomposition.
\end{proof}

\subparagraph{Finding Change-of-Basis Matrices.}

\begin{proposition}\label{double_ext_algorithm}
 The algorithm \hyperref[extensiondecomp]{$\alpha$-\texttt{ExtensionDecomp}} (Appendix) finds an indecomposable decomposition.
\end{proposition}

\begin{proof}
Consider an $\alpha$-decomposed extension of a module $Y$.
\[ \left[ M \ N \right] \ = \  \begin{blockarray}{c c c c c}
           &  R_\Cs & R_\Ds  & k_\Cs \cdot \alpha & k_\Ds \cdot \alpha \\
        \begin{block}{c [ c c c c ]}
            G_\Cs & M_\Cs & 0       & N_\Cs & C \\
            G_\Ds &  0    & M_\Ds   & 0     & N_\Ds  \\
        \end{block}
    \end{blockarray} 
    \]
 If $Z$ is decomposable and 
 $\Hom^\alpha\left(M_\Cs, M_\Ds \right) = 0$
 then we can use
 \autoref{decomp_concentrated}
 to find a decomposition (\autoref{ses_decomp} above) of the entire induced short exact sequence. We also find indecomposable decompositions of the first and third terms of the decomposed ses and by Krull-Remak-Schmidt they agree with the decomposition induced by $M_\Cs$ and $M_\Ds$. We end up with the following diagram.
\[\begin{tikzcd}[ampersand replacement=\&]
	0 \& {\bigoplus\limits_{c \in \Cs} X_c  } \&\& Y \&\& {\bigoplus\limits_{d \in \Ds} Z_d } \& 0 \\
	0 \& {\bigoplus\limits_{c \in \Cs_1} X_c  \oplus \bigoplus\limits_{c \in \Cs_2} X_c } \&\& {{Y_1} \oplus {Y_2}} \&\& {\bigoplus\limits_{d \in \Ds_1} Z_d  \oplus  \bigoplus\limits_{d \in \Ds_2} Z_d  } \& 0
	\arrow[from=1-1, to=1-2]
	\arrow["{ f}", from=1-2, to=1-4]
	\arrow["\wr", from=1-2, to=2-2]
	\arrow["{g }", from=1-4, to=1-6]
	\arrow["\varphi", from=1-4, to=2-4]
	\arrow[from=1-6, to=1-7]
	\arrow["\wr", from=1-6, to=2-6]
	\arrow[from=2-1, to=2-2]
	\arrow["{f_1 \oplus f_2}", from=2-2, to=2-4]
	\arrow["{g_1 \oplus g_2}", from=2-4, to=2-6]
	\arrow[from=2-6, to=2-7]
\end{tikzcd}\]
 
Using \autoref{permutation} we rectify the first and third vertical isomorphism so that diagonal maps are isomorphisms. 
To get presentations of the two modules $Y_1$ and $Y_2$ we use \autoref{pres_ext} for each. To see that the two resulting presentations are again minimal observe that together they have the same size as the presentation of $Y$ above which was minimal. Alternatively for each ses this is a special case of \autoref{invariant_minimal} for $U = \Z^d$ \\.
 
 We can then lift $\varphi$ to morphism of presentation to get graded matrices 

\begin{align*}  Q =  
        \begin{bmatrix}
             Q_{\Cs, \Cs} & Q_{\Cs, \Ds} \\
             0 & Q_{\Ds, \Ds} 
        \end{bmatrix}
 \quad  \quad \text{ and } \quad \quad 
    \begin{bmatrix}
     P & U \\
     0 & T
    \end{bmatrix} =
     \begin{bmatrix}
            P_{\Cs, \Cs} & P_{\Cs, \Ds} & U_{\Cs, \Cs} & U_{\Cs, \Ds} \\
            0 & P_{\Ds,\Ds} & 0 &  U_{\Ds, \Ds}  \\
            0 & 0 & T_{\Cs, \Cs} & T_{\Cs, \Ds}  \\
            0 & 0 & 0 & T_{\Ds, \Ds}
    \end{bmatrix} 
    \end{align*}
    such that 
    \[Q \left[ M \ N \right] \begin{bmatrix}
     P & U \\
     0 & T
    \end{bmatrix} = \begin{bmatrix}
           M_\Cs & 0       & N_\Cs & C' \oplus C'' \\
           0    & M_\Ds   & 0     & N_\Ds  \\
        \end{bmatrix} \]
    where $C' \oplus C''$ is compatible with the block-structure. Again $T_{\Ds, \Ds}$ is partitioned into two matrices $T_1, \, T_2$ with the same columns as $C', \, C''$ respectively and additionally we will also get corresponding partitions $\Cs_1, \Cs_2, \Ds_1, \Ds_2$.

We can assume that the outer loop of the algorithm will at some point choose this $T_{\Ds, \Ds}$.

From here the proof is analogous to \autoref{prop:alpha_correct}. We can argue again that the first step of the algorithm will up to permutation of blocks delete $(C_1)_{\Cs_2}$ by considering the restricted isomorphisms. We avoid the tedious matrix-multiplication this time and just assume that the matrices above transform the presentation from the result of the first step to the decomposed presentation on the first. For the second step, let $c \in \Cs_1$ and consider the sub matrix $(C'')_{\Cs_1}$ $(=0)$ in the equation above and deduce that

 \begin{align*} 
(C_1)_c + \sum\limits_{b \neq c \in \Cs} Q_{c,b} (C_1)_b + M_c U_{c, \Ds_1} + Q_{c, \Ds} M_\Ds  U_{\Ds, \Ds_1} + Q_{c, \Ds_1} N_{\Ds_1} 
= 0.
\end{align*}

But $Q_{c, \Ds} M_\Ds + M_c P_{c, \Ds} = 0$, because $Q, P$ must form a morphism of presentations, so 
\begin{align*}
    M_c U_{c, \Ds_1} + Q_{c, \Ds} M_\Ds  U_{\Ds, \Ds_1} = M_c U_{c, \Ds_1} - M_c P_{c, \Ds} U_{\Ds, \Ds_1} = M_c \left( U_{c, \Ds_1} - P_{c, \Ds} U_{\Ds, \Ds_1} \right).
\end{align*} 

Again, \texttt{BlockReduce} will find a solution to the linear system.
\end{proof}

\begin{corollary}
    The Automorphism-invariant Iterative Decomposition Algorithm (\textsc{AIDA}) finds an indecomposable decomposition.
\end{corollary}

\begin{proof}
Assume that the algorithm did not find an indecomposable decomposition. Then there is a block which can be further split up, and there is a minimal degree $\alpha$ where a part of this block was formed by merging two blocks which represent indecomposable summands, but where this merge is not indecomposable. This must have been the result of $\alpha$-\texttt{ExtensionDecomp} (Appendix), but
by construction of the condensations of $\B$ and $\B_p$ we pass the presentation of an $\alpha$-invariant short exact sequence to this subroutine and we have just proved that it runs correctly in this case. 
\end{proof}

\section{Experiments}
\label{sec:experiments}
We have implemented our decomposition algorithm in a stand-alone C\texttt{++} library
\textsc{aida}.\footnote{\url{https://github.com/JanJend/AIDA}}
\textsc{aida} takes as input a minimal presentation in the form of a \textsc{scc2020} file%
\footnote{\url{https://bitbucket.org/mkerber/chain_complex_format/src/master/}}
which can be generated, for instance, using the \textsc{mpfree} library~\cite{mpfree}. 
All experiments were performed on a workstation with an Intel(R) Xeon(R)
CPU E5-1650 v3 CPU (6 cores, 12 threads, 3.5GHz) and 64 GB RAM,
running GNU/Linux (Ubuntu 20.04.2) and gcc 9.4.0.

\subparagraph{Test instances.}
We considered a large collection of data files~\cite{benchmark_repo}
including chain complex data generated by various bifiltrations,
including function-Rips, lower-star bifiltrations, and multi-covers~-- see \cite{fkr-compression}.
Also included are chain complexes of function-alpha bifiltrations
for various choices of point clouds and density functions as described in the paper of Alonso et al.~\cite{akll-delaunay}.
We additionally generated test cases which are not based on bifiltration data
by simple random processes that can create interval-decomposable or
random presentation inspired by the Erd\"os-Renyi-model of random graphs (\autoref{app:random_presentations}).
In total, the benchmark set contains around 1300 different minimal
presentation files. Our benchmarks scripts, the version of \textsc{aida} that
we have used, and the results of our experiments are available at~\cite{benchmarks_aida}.

Most instances in the above collection are distinctly graded, that is $k_{\max}=1$. This is perhaps surprising because, for instance, function-alpha bifiltrations
are generally not distinctly graded simplex-wise. Yet apparently their homology classes have unique degrees.
A few instances also have $k=2$ or $k=3$, but the only instances with large values of $k$ are the one coming from multi-cover bilfiltrations
as described in~\cite{cklo-computing}.

\subparagraph{Version of the Algorithm.}
For the following experiments we have only used the exhaustive algorithm from \autoref{sec:brute_force}. The full version of \textsc{AIDA} from \autoref{sec:decomp_algo} is only partially implemented at the time of writing. This partial implementation did not show any speed-up on our instances and will be completed, if necessary, in the future, for example if instances with large $k$ from Bjerkevik's work \cite{Bjerkevik} are no longer accessible with the exhaustive approach.

\subparagraph{Effectiveness of the speed-ups.}
We implemented our code such that some optimizations can be disabled
to gauge their importance.
Our first finding is that the column sweep optimization
as described at the beginning of Section~\ref{sec:deyxin}
has the most visible effect on practical performance.
In Table~\ref{tbl:distinctly_graded}, we display the running time
for some representative instances. We observe
that the improvement factor depends on the instance type
and even varies a lot for instances of the same type
(as the third and fourth row of the table indicate).
However, the improvement was consistent over all considered
examples. 

\begin{table}[h]
    \begin{tabular}{c|cc|ccc}
      type & \#Gens & \#Rels  & vanilla & +sweep & +homset  \\
      \hline
      fun-alpha (kde) & 10000 & 14014 & 118 & 0.39 & 0.42\\
      fun-alpha (random) & 6460 & 11154 & 7.64 & 0.37 & 0.33\\
      random & 800 & 782 & 6.85 & 2.12 & 3.31\\
      random & 800 & 776 & 54.0 & 2.54 & 3.03\\
      \hline
      off (hand)   &  595 &  596 &  508  &  0.96 & 0.02\\
      off (eros)   & 2819 & 2820 & >3600 & 22.7  & 0.40\\
      off (dragon) & 8229 & 8230 & >3600 & 833   & 9.4\\
      off (raptor) & 7203 & 7186 & >3600 & 1588  & 32.7\\
    \end{tabular}
    \caption{Running times (in seconds) for the algorithm in different configurations: ``vanilla'' means that no modification is enabaled. ``+sweep'' means that column sweep is enabled, ``+homset'' means that the separate and optimized
      computation of ($\ast$) is enabled.}
    \label{tbl:distinctly_graded}
\end{table}

We next consider the effect of the optimized computation of the system ($\ast$). 
In many instances, this technique does not speed up the computation or even slightly slows it down,
due to the technical overhead of solving ($\ast$) first. However, the slow-down
was never more than a factor of $2$ in the experiments. On the other hand, for some instances,
a speed-up is achieved; most significantly, this is the case for the off-datasets, which are triangular meshes
in 3D bi-filtered with two of the coordinate axis (bottom half of Table~\ref{tbl:distinctly_graded}). This happens because this filtration tends to create the large but pointwise flat summands which we have optimized the computation for. 

\subparagraph{Empirical complexity.}
We demonstrate in Table~\ref{tbl:progression}
that \textsc{aida} is able to decompose presentations
with hundreds of thousands of generators and relations within seconds in maany practical cases.
This is also true for non-distinctly graded cases:
the bottom half of the table displays multi-cover instances
where $k$ is larger than $1$ ($k=4$ on the left, $k=6$ on the right of the table).
We also display the progression of \textsc{aida}'s time and memory consumption when the
input sizes (roughly) doubles. We can see that in some cases, \textsc{aida} scales roughly linearly with
the size of the input presentation (the upper-left and lower-right part of the table)
whereas in other cases the time grows rather quadratically (the upper-right and lower-left part).
We also observed instances with a super-quadratic growth in our experiments; nevertheless,
the empirical behavior is much better than what the asymptotic bound predicts. This is not surprising for realistic sparse input matrices, in analogy to the persistence algorithm
for one parameter, because matrix reduction is generally faster. Also the summands typically stay small, so that \hyperref[hom_computation]{($\ast$)} is fast and a partial decomposition is found after a small number of iterations over subspaces. 

\begin{table}[h]
    \begin{tabular}{c|cc|cc||c|cc|cc}
      dim & \#Gens & \#Rels &  time & mem & dim & \#Gens & \#Rels  & time & mem \\
      \hline
      \multirow{4}{*}{0}
      &  5000 &  7431 &  0.11 &  17 MB & \multirow{4}{*}{1} & 12410  &  19890 &  22.3 &  108 MB\\
      & 10000 & 14701 &  0.22 &  29 MB &                    & 25681  &  40288 &  80.3 &  277 MB\\
      & 20000 & 29030 &  0.47 &  53 MB &                    & 53250  & 81093 &  562 & 928 MB\\
      & 40000 & 57691 &  1.23 & 101 MB &                    & 112229 & 165769 & 1952  & 1.55 GB \\ 
      \hline
      \multirow{4}{*}{2}
      &  15707 &  31414 &  0.47 &  50 MB & \multirow{4}{*}{1}
      &   6856 &  13711 & 0.12  &  23 MB\\
      &  32251 &  64503 &  1.25 &  97 MB & & 14570 &  29140 & 0.27 &  46 MB\\
      &  71166 & 142332 &  6.36 & 209 MB & & 33303 &  66606 & 0.76 &  99 MB\\
      & 154400 & 308800 & 17.6 &  447 MB & & 74462 & 148924 & 1.95 & 218 MB\\
      \hline
    \end{tabular}
    \caption{Running time and memory of \textsc{aida} for large data sets. The top half are function-alpha bifiltrations
      for point clouds of a torus with kde function value, in homology dimension $0$ and $1$.
      The bottom half are multi-cover bifiltration of a random point sample, with presentations in homology dimension $2$ and $1$. All numbers are averaged over $5$ instances.}
    \label{tbl:progression}
\end{table}

\subparagraph{Comparison with other software.}

We are not aware of any other implementation for the special case of multi-graded
persistence modules, but there exist some libraries that include a decomposition
algorithm for modules over finite dimensional algebras, for example the algorithm by Lux and Sz\"{o}ke \cite{LuxSzoke} in the C-MeatAxe library. Since using this algorithm would require too many preprocessing steps, we instead used the decomposition algorithm for quiver representations in the \textsc{QPA} pacakge \cite{QPA} for \textsc{GAP} \cite{GAP4} and implemented (also in the \textsc{aida} package) an algorithm to convert MPM to quiver representations.
Our tests show that this algorithm already needs several minutes for presentations
with around $100$ generators and relations and times out for all presentations we have considered for testing $\textsc{aida}$.

\subparagraph{Speeding up barcode template computation.}\label{sub:barcode_template}
We demonstrate how decomposing a module can speed up computational tasks
in multi-parameter persistence: barcode templates are a common tool
to capture all combinatorial barcodes that arise from a $2$-parameter
module through restrictions to one parameter. Such a barcode
template is computed via the arrangement of lines in the plane,
arising from the joins of generators and relations of a module
via point-line duality. It is evident that this construction
is compatible with decomposition, meaning that if a modules
$M=M_1\oplus M_2$ decomposes, the arrangements of $M_1$ and $M_2$
can be computed separately, and the arrangements can be overlaid.

We implemented a small test program that computes the barcode template
using the arrangement package of \textsc{Cgal}~\cite{arrangments}
once for the undecomposed modules and once after applying \textsc{aida}.
Table~\ref{tbl:barcode_template} shows that the arrangement size
as well as the computation time drastically reduces when using \textsc{aida}.
We remark that the variance of the aida arrangement in size was very large,
leading to the effect that the average arrangment size even decreases
in our sample.
In either case, we speculate that incorporating \textsc{aida} would speed-up
the visualization of persistence modules in \textsc{Rivet}
and is a necessary preprocessing step in the exact computation of matching distances
in practice.

\begin{table}[H]
  \begin{tabular}{ccc|ccc|ccc}
    & & & \multicolumn{3}{c}{With \textsc{aida}} & \multicolumn{3}{c}{Without \textsc{aida}} \\
      \#points & \#Gens & \#Rels &  \#lines & V & time & \#lines & V & time\\
      \hline
      20 & 18.4 & 28    & 24.6  & 210   & 0.002 & 176   & 6639      & 0.034\\
      40 & 59   & 97    & 144   &  6295 & 0.026 & 1877  & 692K      & 3.50\\
      60 & 99   & 166   & 382   & 39996 & 0.19  & 5096  & 5.076M    & 27.1\\
      80 & 137  & 224   & 346   & 33301 & 0.16  & 8801  & 14.3M     & 81.8\\
    \end{tabular}
  \caption{Running time for computing the barcode template of function-alpha bifiltrations
    of randomly sampled point in the unit square using random function values. The columns
    show the number of lines generated that define the arrangement, the number of vertices
    in the final arrangement and the time to compute it. All number are averages of $5$ randomly
    generated instances.}
    \label{tbl:barcode_template}
\end{table}

\section{Conclusion}
\label{sec:conclusion}
\textsc{aida} is the first practical even implemented algorithm to decompose
multi-parameter persistence modules. We emphasise that it
works for any number of parameters (although all experiments were on $2$ parameters)
and its strategy should be adaptable to any other poset or acyclic quiver.

Our improvements are a consequence of a better understanding of MPM on the algebraic level. More speed-ups are possible; for instance, the faster
computation of $\Hom(X,Y)$ that we employ is far from best possible and 
we plan to improve it further. We also believe that the algorithm from Section~\ref{sec:aida} should show improvement over the exhaustive algorithm of 
Section~\ref{sec:brute_force_decomposer} for sparse input with larger $k_{\max}$ and non-sparse input also for smaller $k_{\max}$.

There is also a connection with graphcodes~\cite{KerberRussold24}: first of all, decomposing a module with \textsc{aida}
can reduce the ambiguity of graphcodes in graph neural-network settings.
Moreover, it has been shown
recently~\cite{graphcode-socg} that graphcodes yield a partial decomposition of a module and can be combined
with \textsc{aida} to speed up the decomposition in some cases.

 \bibliographystyle{plainurl}
\bibliography{decomp}

@article{DeyXin,
  doi = {10.1007/s41468-022-00087-5},
  author = {Dey, Tamal K. and Xin, Cheng},
  title = {Generalized Persistence Algorithm for Decomposing Multiparameter Persistence Modules},
  Journal= {Journal of Applied and Computational Topology},
  Volume= {6},
  Pages = {271-322},
  year = {2022}
}

@article{exactcats,
author = {Draxler, Peter and Reiten, Idun and Smalø, Sverre and Solberg, Øyvind and Bernhard Keller},
year = {1999},
pages = {647-682},
title = {Exact categories and vector space categories},
volume = {351},
journal = {Transactions of the American Mathematical Society},
doi = {10.1090/S0002-9947-99-02322-3}
}

@article{Ronyai,
  title={Computing the Structure of Finite Algebras},
  author={Lajos R{\'o}nyai},
  journal={Journal of Symbolic Computation},
  year={1990},
  volume={9},
  pages={355-373},
  url={https://api.semanticscholar.org/CorpusID:21149033}
}

@article{HoltRees92,
title={Testing modules for irreducibility}, 
volume={57}, 
DOI={10.1017/S1446788700036016},
number={1}, 
journal={Journal of the Australian Mathematical Society. Series A. Pure Mathematics and Statistics},
author={Holt, Derek F. and Rees, Sarah}, 
year={1994},
pages={1–16}}

@article {MAGMA,
    AUTHOR = {Bosma, Wieb and Cannon, John and Playoust, Catherine},
     TITLE = {The {M}agma algebra system. {I}. {T}he user language},
   JOURNAL = {Journal of Symbolic Computation},
    VOLUME = {24},
      YEAR = {1997},
    NUMBER = {3-4},
     PAGES = {235--265},
       DOI = {10.1006/jsco.1996.0125},
}

@manual{GAP4,
organization = "The GAP~Group",
title        = "{GAP -- Groups, Algorithms, and Programming,
                Version 4.13.1}",
year         = 2024,
url          = "https://www.gap-system.org",
}

@manual{QPA,
author = "The QPA-team",
title = "QPA - Quivers, path algebras and representations - a GAP package, Version 1.33",
year = "2022",
url = "https://folk.ntnu.no/oyvinso/QPA/",
}

@inproceedings {CGK97,
    AUTHOR = {Chistov, Alexander and Ivanyos, G\'abor and Karpinski, Marek},
     TITLE = {Polynomial time algorithms for modules over finite dimensional
              algebras},
 BOOKTITLE = {{I}nternational {S}ymposium on
              {S}ymbolic and {A}lgebraic {C}omputation (ISSAC)},
      YEAR = {1997},
       DOI = {10.1145/258726.258751},
}

@inproceedings{KerberRussold24,
title={Graphcode: Learning from multiparameter persistent homology using graph neural networks},
author={Florian Russold and Michael Kerber},
booktitle={Advances in Neural Information Processing Systems (NeurIPS)},
year={2024},
 url          = {http://papers.nips.cc/paper\_files/paper/2024/hash/4822d9adc9cec7a39e254d007aa78276-Abstract-Conference.html}
}

@book{Weibel, 
place={Cambridge}, 
title={An Introduction to Homological Algebra}, 
publisher={Cambridge University Press}, 
author={Weibel, Charles A.}, 
year={1994}
}

@article{CZMP,
author = {Carlsson, Gunnar and Zomorodian, Afra},
year = {2007},
month = {06},
pages = {71-93},
title = {The Theory of Multidimensional Persistence},
volume = {42},
journal = {Discrete and Computational Geometry},
doi = {10.1007/s00454-009-9176-0}
}

@inproceedings{CZMP-socg,
  author       = {Gunnar E. Carlsson and
                  Afra Zomorodian},
  title        = {The Theory of Multidimensional Persistence},
  booktitle    = {The {ACM} Symposium on Computational Geometry (SoCG)},
  year         = {2007},
  doi          = {10.1145/1247069.1247105},
}

@book{Assem, 
place={Cambridge}, 
title={Elements of the Representation Theory of Associative Algebras: Techniques of Representation Theory}, 
DOI={10.1017/CBO9780511614309}, 
publisher={Cambridge University Press}, 
author={Assem, Ibrahim and Skowronski, Andrzej and Simson, Daniel}, 
year={2006}
}

@article{CBBB,
      title={Decomposition of persistence modules}, 
      author={Magnus Bakke Botnan and William Crawley-Boevey},
      year={2020},
    month = {08},
    pages = {4581-4596 },
    volume = {148},
    journal = {Proceedings of the American Mathematical Society},
    doi = {https://doi.org/10.1090/proc/14790}
}

@article{Azumaya,
author = {Gor{\^o} Azumaya},
title = {{Corrections and supplementaries to my paper concerning Krull-Remak-Schmidt's theorem}},
volume = {1},
journal = {Nagoya Mathematical Journal},
publisher = {Nagoya Mathematical Journal},
pages = {117 -- 124},
year = {1950}
}

@article{ABENY,
title = {On interval decomposability of 2D persistence modules},
journal = {Computational Geometry},
volume = {105-106},
pages = {101879},
year = {2022},
issn = {0925-7721},
doi = {https://doi.org/10.1016/j.comgeo.2022.101879},
author = {Hideto Asashiba and Mickaël Buchet and Emerson G. Escolar and Ken Nakashima and Michio Yoshiwaki},
}

@misc{BL22,
      title={Generic Two-Parameter Persistence Modules are Nearly Indecomposable}, 
      author={Ulrich Bauer and Luis Scoccola},
      year={2022},
      eprint={2211.15306},
      archivePrefix={arXiv},
      primaryClass={math.RT}
}

@article{BE22,
    title = {Realizations of Indecomposable Persistence Modules of Arbitrarily Large Dimension},
    author = {Micka{\"e}l Buchet and Escolar, {Emerson G.}},
    year = {2022},
    doi = {10.20382/jocg.v13i1a12},
    volume = {13},
    pages = {298-326},
    journal = {Journal of Computational Geometry},
    issn = {1920-180X},
    publisher = {Carleton University},
    number = {1}
}

@inproceedings{BE22-socg,
  author       = {Micka{\"{e}}l Buchet and
                  Emerson G. Escolar},
   title        = {Realizations of Indecomposable Persistence Modules of Arbitrarily
                  Large Dimension},
  booktitle    = {International Symposium on Computational Geometry (SoCG)},
  year         = {2018},
  doi          = {10.4230/LIPICS.SOCG.2018.15},
}

@inproceedings{AKS24,
title = "Probabilistic Analysis of Multiparameter Persistence Decompositions into Intervals",
author = "Alonso, {{\'A}ngel Javier} and Michael Kerber and Primoz Skraba",
year = "2024",
doi = "10.4230/LIPIcs.SoCG.2024.6",
booktitle = "International Symposium on Computational Geometry (SoCG)"
}

@inproceedings{AlonsoKerber23,
  author       = {{\'{A}}ngel Javier Alonso and
                  Michael Kerber},
  title        = {Decomposition of Zero-Dimensional Persistence Modules via Rooted Subsets},
  booktitle    = {International Symposium on Computational Geometry (SoCG)},
  year         = 2023,
  doi          = {10.4230/LIPICS.SOCG.2023.7}
}

@InProceedings{dey_xin,
  author ={Tamal K. Dey and Cheng Xin},
  title ={{Computing Bottleneck Distance for 2-D Interval Decomposable Modules}},
  booktitle ={International Symposium on Computational Geometry (SoCG)},
  year ={2018},
  doi ={10.4230/LIPIcs.SoCG.2018.32},
}

@book{maclane,
  title={Homology},
  author={MacLane, Saunders},
  isbn={9783642620294},
  series={Classics in Mathematics},
  url={https://books.google.at/books?id=ujRqCQAAQBAJ},
  year={2012},
  publisher={Springer}
}

@article{Hilbert1890,
author = {David Hilbert},
journal = {Mathematische Annalen},
pages = {473-534},
title = {{\"U}ber die {T}heorie der algebraischen {F}ormen},
url = {http://eudml.org/doc/157506},
volume = {36},
year = {1890}
}

@book{Peeva, 
title={Graded Syzygies}, 
DOI={10.1007/978-0-85729-177-6}, 
journal={Springer}, 
author={Peeva, Irena}, 
year={2011}, 
}

@book{EnochsJenda,
title = {Relative Homological Algebra},
author = {Edgar E. Enochs and Overtoun M. G. Jenda},
publisher = {De Gruyter},
address = {Berlin, New York},
doi = {doi:10.1515/9783110803662},
year = {2000},
}

@article{Hochschild,
 ISSN = {00029947},
 URL = {http://www.jstor.org/stable/1992988},
 author = {Gerhard Hochschild},
 journal = {Transactions of the American Mathematical Society},
 number = {1},
 pages = {246--269},
 publisher = {American Mathematical Society},
 title = {Relative Homological Algebra},
 urldate = {2023-10-04},
 volume = {82},
 year = {1956}
}

@misc{BotnanLesnick,
      title={An Introduction to Multiparameter Persistence}, 
      author={Magnus Bakke Botnan and Michael Lesnick},
      year={2023},
      eprint={2203.14289},
      archivePrefix={arXiv},
      primaryClass={math.AT}
}

@incollection{Parker84,
  author      ={Richard A. Parker},
  title       ={The Computer Calculation of
Modular Characters (the Meat-Axe)},
  booktitle   = {Computational
Group Theory},
  publisher   = {Academic Press},
  year        = {1984},
  pages       = {267-274}
}

@article{LuxSzoke,
author = {Klaus M. Lux and Magdolna Sz{\H o}ke},
title = {{Computing Decompositions of Modules over Finite-Dimensional Algebras}},
volume = {16},
journal = {Experimental Mathematics},
number = {1},
publisher = {A K Peters, Ltd.},
pages = {1--6},
year = {2007}
}

@phdthesis{Szoke,
    title    = {Examining Green Correspondents
of Weight Modules},
    school   = {RWTH Aachen},
    author   = {Magdolna Sz{\H o}ke},
    year     = {1998}
}

@Article{Bjerkevik,
author={Bjerkevik, H{\aa}vard Bakke},
title={Stabilizing Decomposition of Multiparameter Persistence Modules},
journal={Foundations of Computational Mathematics},
year={2025},
month={Jan},
day={27},
issn={1615-3383},
doi={10.1007/s10208-025-09695-w},
url={https://doi.org/10.1007/s10208-025-09695-w}
}

@article{bk-asymptotic,
  author       = {H{\aa}vard Bakke Bjerkevik and
                  Michael Kerber},
  title        = {Asymptotic improvements on the exact matching distance for 2-parameter
                  persistence},
  journal      = {Journal of Compututational Geometry},
  volume       = {14},
  number       = {1},
  pages        = {309--342},
  year         = {2023},
  doi          = {10.20382/JOCG.V14I1A12},
}

@article{italian,
  title={A new algorithm for computing the 2-dimensional matching distance between size functions},
  author={Biasotti, Silvia and Cerri, Andrea and Frosini, Patrizio and Giorgi, Daniela},
  journal={Pattern Recognition Letters},
  volume=32,
  number=14,
  pages={1735--1746},
  year=2011,
  doi          = {10.1016/J.PATREC.2011.07.014},
  publisher={Elsevier}
}

@article{lw-interactive,
	Author = {Lesnick, Michael and Wright, Matthew},
	eprint = {1512.00180},
	Title = {Interactive visualization of {2-D} persistence modules persistence modules},
	Year = {2015},
	archivePrefix={arXiv},
}

@article{fkr-compression,
title = {Compression for 2-parameter persistent homology},
journal = {Computational Geometry},
volume = {109},
year = {2023},
doi = {https://doi.org/10.1016/j.comgeo.2022.101940},
author = {Ulderico Fugacci and Michael Kerber and Alexander Rolle},
}

@InProceedings{akll-delaunay,
  author = 	 {\'{A}ngel Javier Alonso and Michael Kerber and Tung Lam and Michael Lesnick},
  title = 	 {Delaunay Bifiltrations of Functions on Point Clouds},
  booktitle = 	 {Symposium on Discrete Algorithms (SODA)},
  year =	 2024,
  doi          = {10.1137/1.9781611977912.173},
}

@inproceedings{akp-filtration,
  author       = {{\'{A}}ngel Javier Alonso and
                  Michael Kerber and
                  Siddharth Pritam},
  title        = {Filtration-Domination in Bifiltered Graphs},
  booktitle    = {Algorithm Engineering and Experiments (ALENEX)},
  year         = {2023},
  doi          = {10.1137/1.9781611977561.CH3},
}

@inproceedings{bdk-sparse,
  author       = {Micka{\"{e}}l Buchet and
                  Bianca B. Dornelas and
                  Michael Kerber},
  title        = {Sparse Higher Order {{\v{C}}}ech Filtrations},
  booktitle    = {International Symposium on Computational Geometry (SoCG)},
  year         = 2023,
  doi          = {10.4230/LIPICS.SOCG.2023.20},
}

@article{bdk-sparse-journal,
  author       = {Micka{\"{e}}l Buchet and
                  Bianca B. Dornelas and
                  Michael Kerber},
  title        = {Sparse Higher Order {{\v{C}}}ech Filtrations},
  journal      = {Journal of the {ACM}},
  volume       = {71},
  number       = {4},
  pages        = {28:1--28:23},
  year         = {2024},
  doi          = {10.1145/3666085},
}

@inproceedings{cklo-computing,
  author       = {Ren{\'{e}} Corbet and
                  Michael Kerber and
                  Michael Lesnick and
                  Georg Osang},
  title        = {Computing the Multicover Bifiltration},
  booktitle    = {International Symposium on Computational Geometry (SoCG)},
  year         = 2021,
  doi          = {10.4230/LIPICS.SOCG.2021.27},
}

@inproceedings{ks-localized,
  author       = {Michael Kerber and
                  Matthias S{\"{o}}ls},
  title        = {The Localized Union-Of-Balls Bifiltration},
  booktitle    = {International Symposium on Computational Geometry (SoCG)},
  year         = {2023},
  doi          = {10.4230/LIPICS.SOCG.2023.45},
}

@article{lw-computing,
author = {Lesnick, Michael and Wright, Matthew},
title = {Computing Minimal Presentations and Bigraded {B}etti Numbers of 2-Parameter Persistent Homology},
journal = {SIAM Journal on Applied Algebra and Geometry},
volume = 6,
number = 2,
pages = {267-298},
year = 2022,
doi = {10.1137/20M1388425}
}

@Article{ladderdecomp,
author={Asashiba, Hideto
and Escolar, Emerson G.
and Hiraoka, Yasuaki
and Takeuchi, Hiroshi},
title={Matrix method for persistence modules on commutative ladders of finite type},
journal={Japan Journal of Industrial and Applied Mathematics},
year={2019},
month={Jan},
day={01},
volume={36},
number={1},
pages={97-130},
issn={1868-937X},
doi={10.1007/s13160-018-0331-y},
url={https://doi.org/10.1007/s13160-018-0331-y}
}

@InProceedings{botnan2021signed,
  author =	{Botnan, Magnus Bakke and Oppermann, Steffen and Oudot, Steve},
  title =	{{Signed Barcodes for Multi-Parameter Persistence via Rank Decompositions}},
  booktitle =	{International Symposium on Computational Geometry (SoCG)},
  year =	{2022},
  doi =		{10.4230/LIPIcs.SoCG.2022.19},
}

@misc{morozov2021output,
      title={Output-sensitive Computation of Generalized Persistence Diagrams for 2-filtrations}, 
      author={Dmitriy Morozov and Amit Patel},
      year={2023},
      eprint={2112.03980},
      archivePrefix={arXiv},
      primaryClass={cs.CG},
}

@article{mccleary2022edit,
    AUTHOR = {McCleary, Alexander and Patel, Amit},
     TITLE = {Edit distance and persistence diagrams over lattices},
   JOURNAL = {SIAM Journal on Applied Algebra and Geometry},
    VOLUME = {6},
      YEAR = {2022},
    NUMBER = {2},
     PAGES = {134--155},
   MRCLASS = {55U15 (55N31)},
  MRNUMBER = {4405184},
MRREVIEWER = {Jingyan Li},
       DOI = {10.1137/20M1373700},
}

@InProceedings{xin2023gril,
  title = 	 {{GRIL}: A $2$-parameter Persistence Based Vectorization for Machine Learning},
  author =       {Xin, Cheng and Mukherjee, Soham and Samaga, Shreyas N. and Dey, Tamal K.},
  booktitle = 	 {Workshop on Topology, Algebra, and Geometry in Machine Learning (TAG-ML)},
  pages = 	 {313--333},
  year = 	 {2023},
  url = 	 {https://proceedings.mlr.press/v221/xin23a.html},
}

@article{vipond2018multiparameter,
  author  = {Oliver Vipond},
  title   = {Multiparameter Persistence Landscapes},
  journal = {Journal of Machine Learning Research},
  year    = {2020},
  volume  = {21},
  pages   = {1-38},
  url     = {http://jmlr.org/papers/v21/19-054.html}
}

@article{corbet2019kernel,
  title={A kernel for multi-parameter persistent homology},
  author={Corbet, Ren{\'e} and Fugacci, Ulderico and Kerber, Michael and Landi, Claudia and Wang, Bei},
  journal={Computers \& Graphics: X},
  volume={2},
  year={2019},
  publisher={Elsevier},
  doi={https://doi.org/10.1016/j.cagx.2019.100005},
}

@inproceedings{loiseaux2023stable,
 author = {Loiseaux, David and Scoccola, Luis and Carri\`{e}re, Mathieu and Botnan, Magnus Bakke and Oudot, Steve},
 booktitle = {Advances in Neural Information Processing Systems (NeurIPS)},
 title = {Stable Vectorization of Multiparameter Persistent Homology using Signed Barcodes as Measures},
 year = {2023},
 url          = {http://papers.nips.cc/paper\_files/paper/2023/hash/d75c474bc01735929a1fab5d0de3b189-Abstract-Conference.html}
}

@inproceedings{dkm-computing,
  author       = {Tamal K. Dey and
                  Woojin Kim and
                  Facundo M{\'{e}}moli},
  title        = {Computing Generalized Rank Invariant for 2-Parameter Persistence Modules
                  via Zigzag Persistence and Its Applications},
  booktitle    = {International Symposium on Computational Geometry (SoCG)},
  year         = {2022},
  doi          = {10.4230/LIPICS.SOCG.2022.34},
}

@inCollection{holt_1998, 
place={Cambridge}, 
title={The Meataxe as a tool in computational group theory}, 
DOI={10.1017/CBO9780511565830.011}, 
booktitle={The Atlas of Finite Groups - Ten Years On}, 
publisher={Cambridge University Press}, 
author={Holt, Derek F.}, 
year=1998, 
pages={74–81}, 
}

@article{phat,
  title={{P}hat--Persistent Homology Algorithms Toolbox},
  author={Bauer, Ulrich and Kerber, Michael and Reininghaus, Jan and Wagner, Hubert},
  journal={Journal of Symbolic Computation},
  volume=78,
  pages={76--90},
  year=2017,
  publisher={Elsevier}
}

@book{eh-computational,
  title={Computational Topology: an Introduction},
  author={Edelsbrunner, Herbert and Harer, John},
  year=2010,
  publisher={American Mathematical Society}
}

@inProceedings{mpfree,
author = {Michael Kerber and Alexander Rolle},
title = {Fast Minimal Presentations of Bi-graded Persistence Modules},
booktitle = {Algorithm Engineering and Experiments (ALENEX)},
year = {2021},
doi = {10.1137/1.9781611976472.16},
}

@InProceedings{cohomology,
  author =	{Bauer, Ulrich and Lenzen, Fabian and Lesnick, Michael},
  title =	{{Efficient Two-Parameter Persistence Computation via Cohomology}},
  booktitle =	{International Symposium on Computational Geometry (SoCG)},
  year =	{2023},
  doi =		{10.4230/LIPIcs.SoCG.2023.15},
}

@article{Schubert,
author = {Hermann Schubert},
title = {{{A}nzahl-{B}estimmungen f\"ur lineare {R}\"aume beliebiger {D}imension}},
volume = {8},
journal = {Acta Mathematica},
publisher = {Institut Mittag-Leffler},
pages = {97 -- 118},
year = {1900},
doi = {10.1007/BF02417085},
}

@InProceedings{Loupias,
author="Loupias, Mich{\`e}le",
title="Indecomposable representations of finite ordered sets",
booktitle="Representations of Algebras",
year="1975",
publisher="Springer",
pages="201--209",
}

@misc{botnan2020,
      title={A Relative Theory of Interleavings}, 
      author={Magnus Bakke Botnan and Justin Curry and Elizabeth Munch},
      year={2020},
      eprint={2004.14286},
      archivePrefix={arXiv},
      primaryClass={math.CT},
}

@misc{lesnick2023multiparameter,
  author       = {Michael Lesnick},
  title        = {Notes on Multiparameter Persistence (for AMAT 840)},
  year         = {2023},
  note         = {Lecture notes, University at Albany, SUNY, Tuesday 13th June, 2023},
  url          = {https://www.albany.edu/~ML644186/840_2022/Math840_Notes_22.pdf}
}

@misc{fersztand,
      title={{H}arder-{N}arasimhan filtrations of persistence modules: metric stability}, 
      author={Marc Fersztand},
      year={2024},
      eprint={2406.05069},
      archivePrefix={arXiv},
      primaryClass={math.RT},
}

@article{FEVT,
author = {Fersztand, Marc and Jacquard, Emile and Nanda, Vidit and Tillmann, Ulrike},
title = {Harder–{N}arasimhan filtrations of persistence modules},
journal = {Transactions of the London Mathematical Society},
volume = {11},
number = {1},
doi = {https://doi.org/10.1112/tlm3.70003},
year = {2024}
}

@inproceedings{CCGGO,
author = {Chazal, Fr\'{e}d\'{e}ric and Cohen-Steiner, David and Glisse, Marc and Guibas, Leonidas J. and Oudot, Steve Y.},
title = {Proximity of persistence modules and their diagrams},
year = {2009},
doi = {10.1145/1542362.1542407},
booktitle = {The {ACM} Symposium on Computational Geometry (SoCG)},
}

@article{schreyer,
title = {Refined algorithms to compute syzygies},
journal = {Journal of Symbolic Computation},
volume = {74},
pages = {308-327},
year = {2016},
issn = {0747-7171},
doi = {https://doi.org/10.1016/j.jsc.2015.07.004},
author = {Burçin Eröcal and Oleksandr Motsak and Frank-Olaf Schreyer and Andreas Steenpaß},
}

@Article{Kleiner1975,
author={Kleiner, Mark},
title={Partially ordered sets of finite type},
journal={Journal of Soviet Mathematics},
year={1975},
month={May},
day={01},
volume={3},
number={5},
pages={607-615},
issn={1573-8795},
doi={10.1007/BF01084663},
url={https://doi.org/10.1007/BF01084663}
}

@article{blumberg_lesnick,
  author       = {Andrew J. Blumberg and
                  Michael Lesnick},
  title        = {Stability of 2-Parameter Persistent Homology},
  journal      = {Foundations of Computational Mathematics},
  volume       = {24},
  number       = {2},
  pages        = {385--427},
  year         = {2024},
  doi          = {10.1007/S10208-022-09576-6},
}

@misc{graphcode-socg,
      title={Representing two-parameter persistence modules via graphcodes}, 
      author={Michael Kerber and Florian Russold},
      year={2025},
      eprint={2503.07368},
      archivePrefix={arXiv},
}

@incollection{arrangments,
  author = {Ron Wein and Eric Berberich and Efi Fogel and Dan Halperin and Michael Hemmer and Oren Salzman and Baruch Zukerman},
  title = {{2D} Arrangements},
  publisher = {{CGAL Editorial Board}},
  edition = {{6.0.1}},
  booktitle = {{CGAL} User and Reference Manual},
  url = {https://doc.cgal.org/6.0.1/Manual/packages.html#PkgArrangementOnSurface2},
  year = 2024
}

@Misc{benchmark_repo,
  author = 	 {Michael Kerber and Tung Lam},
  title = 	 {Benchmark data sets of minimal presentations of 2-parameter persistence modules [Data set]},
  howpublished = {Graz University of Technology},
  year = {2025},
  doi = 	 {10.3217/rxedk-qyq77}}

@Misc{benchmarks_aida,
  author = 	 {Jan Jendrysiak and Michael Kerber},
  title = 	 {Aida Benchmarks {SoCG} 2025 [Data set]},
  howpublished = {Graz University of Technology},
  year = 	 2025,
  doi = 	 {10.3217/ag750-fq560}
}

@article{miller00,
title = {The Alexander Duality Functors and Local Duality with Monomial Support},
journal = {Journal of Algebra},
volume = {231},
number = {1},
pages = {180-234},
year = {2000},
issn = {0021-8693},
doi = {https://doi.org/10.1006/jabr.2000.8359},
url = {https://www.sciencedirect.com/science/article/pii/S0021869300983595},
author = {Ezra Miller}
}

@Inbook{Rojter1980,
author="Rojter, A. V.",
editor="Dlab, Vlastimil
and Gabriel, Peter",
title="Matrix problems and representations of Bocs's",
bookTitle="Representation Theory I: Proceedings of the Workshop on the Present Trends in Representation Theory, Ottawa, Carleton University, August 13 -- 18, 1979",
year="1980",
publisher="Springer Berlin Heidelberg",
address="Berlin, Heidelberg",
pages="288--324",
isbn="978-3-540-38385-7",
doi="10.1007/BFb0089782",
url="https://doi.org/10.1007/BFb0089782"
}

\appendix

\section{Generating random presentations}
\label{app:random_presentations}

We briefly describe three processes to obtain random presentations
with different properties. The script to generate them is contained
in the code repository for this submission.

\begin{description}

\item [Intervals] Given a parameter $n$,
we generated interval-decomposable modules by generating
$n$ intervals of simple shape: With probability $0.1$, the interval is free
(with a single generator chosen uniformly in the unit square). With probability
$0.9$, the interval is an ``infinite L-shape'' with one generator and one
relation randomly chosen in the unit square.
Then, the direct sum of these interval is ``mixed''
by a random sequence of admissible row and column operations.
These instances were mostly created for validation as the decomposition
is known in advance. In general, \textsc{aida} decomposes them very fast.

\item [Random]
We created more challenging modules inspired by the Erd\"os-Renyi-model of random graphs: Pick $m$ degrees for the generators uniformly at random in $[0,1]^2$.
To generate the $n$ relations, pick a degree $\alpha$ uniformly at random in the unit square; now pick a subset of generators with degree $\leq\alpha$ and include
each generator in the subset independently with some probability $p\in(0,1)$. Now add the relation at degree $\alpha$ relating the
generators in the subset picked in this way (repeat the process if
the relation if the subset of generators is empty).
This results in a presentation of size $m\times n$ which is subsequently minimized using \textsc{mpfree}.

\item [Random on grid]
As we wanted to measure the performance of \textsc{aida} for large values of $k$
with more data sets,
we generated further test cases as a variation of the Erd\"os-Renyi construction above: instead of choosing degrees in $[0,1]^2$, we put them
on a square grid of size $N$ uniformly at random. This typically leads to a smaller minimal presentation as many generator-relation pairs cancel out but with enough generators and relations
and a sufficiently small grid, we can generate test cases with large values of $k$.
\end{description}

\section{Missing Algorithms}

\begin{algorithm}[Gaussian Elimination for row-operations along a preordered set.]\label{sub:elimination_row} \
\newline
\texttt{Input}: A matrix $M \in \K^{m \times n}$ and a preorder $P$ on $[m]$. \\
Compute any linear extension of $P$ and reorder $[m]$ accordingly. For $i \in [m]$: 
\begin{enumerate}
\item Let $p_i \in [n]$ be the position of the first non-zero entry of $M_{i, \cdot}$. 
\item For $j >_P i \in [m]$: \\
If $M_{j, p_i} \neq 0$ add the $i$-th to the $j$-th column. 
\end{enumerate}
\texttt{Effect}: For every principal upper set $(i) \subset [m]$ wrt. $P$ the submatrix $M_{U, \cdot}$ has unique pivots. In particular the number of non-zero rows is minimal. 
\end{algorithm}

\begin{algorithm}[\texttt{BlockReduce} for interval decomposable modules.]\label{sub:interval_blockreduce}
\texttt{Input}: $\alpha$-decomposed presentation $\left( M_\B \, N \right)$ where every block $M_b$ presents an interval in normal form \autoref{prop:pres_interval} \\
It holds that $\dim \left( \Ima N_{b,i} \right)_\alpha \leq \left( \coker M_b \right)_\alpha \leq 1$, so that up to column operations we can replace $N_{b,i}$ with a single entry in $\K$. These column operations can be performed in time linear wrt. the number of relations of $M_b$, since $M_b$ is in normal form. \\
After this reduction we are left with a $|\B| \times k$ Matrix and by \autoref{lem:alpha_interval} $\dim \Hom^\alpha\left(-,-\right) = 1$ is a poset relation on $\B$. This allows \hyperref[sub:elimination_row]{Gaussian Elimination along a poset}.
\end{algorithm}

\begin{algorithm}[Internal Column Operations]\label{compute_col} \ \newline
\texttt{Input:} $b, c \in \B$, $[ M_b \ N_b], \,  [ M_c \ N_c]$, a basis $(Q_i, \, P_i)_{i \in I}$ for $\Hom\left( M_b, M_c \right)$ \\
\texttt{Output:} A basis for $\Ima \pi_\alpha$
For each $i \in I$, solve
\[ Q_i N_b = N_c T_i + M_c U_i \text{ for } T_i, \, U_i \in \K^{k_c \times k_b}.\]
By viewing $T_i$ as vectors, reduce the matrix $\left( T_i \right)$. \\
For each cocycle $C$ in the same rows as $N_c$ merging $c$ with another block $d$ we need to enlarge the system by solving also
$C T_i = Q N_b + M_d U$ for $Q \in \Hom(M_b, M_d)$.
\end{algorithm}

\begin{algorithm}[$\alpha$-ExtensionDecomp]\label{extensiondecomp} \ \\
 \textbf{Input:} An $\alpha$-decomposed extension (\autoref{def:extension_presentation}) \\
 Use \autoref{compute_col} to compute all possible column-operations at $\alpha$.
Analogously to \autoref{subparagraph:exhaustive}, compute all decompositions $(T_1 \ T_2)$ of $\K^{\kappa}$, but only those which come from possible column-operations.\\
For each such pair $(T_1 \ T_2)$: 

 \begin{enumerate}
 \item There is a corresponding partition $\Ds = \Ds_1 \cup \Ds_2$.
  \item Set $C_1 \coloneqq CT_1$ and $C_2 \coloneqq CT_2$. For $c \in \Cs \colon$
  \begin{itemize}
   \item Use {\texttt{BlockReduce}} \hyperref[hom_computation]{($\ast \ast $)} to zero out $(C_1)_c$. If successful add $c$ to $\Cs_2$ otherwise to $\Cs_1$.
    \end{itemize}
  \item Use {\texttt{BlockReduce}} \hyperref[hom_computation]{($\ast \ast $)} to zero out $(C_2)_{\Cs_2}$. If successful, pass 
  the two $\alpha$-decomposed extension induced by $\Cs_1$, $\Ds_1$ and $\Cs_2$, $\Ds_2$ recursively to the algorithm. If not continue.
\end{enumerate}
After having tried all $(T_1 \ T_2)$ without success merge all blocks in $\Cs$ and $\Ds$.
\end{algorithm}

\end{document}